\documentclass[12pt,a4paper,reqno]{amsart}
\usepackage[T1]{fontenc}
\usepackage[utf8]{inputenc}
\usepackage[english]{babel}
\numberwithin{equation}{section}
\usepackage[margin=1in]{geometry}

\usepackage{amssymb}
\usepackage{stmaryrd} 
\usepackage{mathrsfs} 

\usepackage{tikz-cd}
\usepackage{tikz}
\usetikzlibrary{positioning, shapes.geometric, arrows.meta}

\usepackage{graphicx}
\usepackage{subcaption}
\usepackage{yhmath}
\usepackage[pdftex,unicode,
colorlinks=true,
linkcolor = blue,
citecolor = blue]{hyperref}

\newtheorem{theorem}{Theorem}[section]
\newtheorem{proposition}[theorem]{Proposition}

\newtheorem{lemma}[theorem]{Lemma}
\newtheorem{deflemma}[theorem]{Definition/Lemma}
\newtheorem{defproposition}[theorem]{Definition/Proposition}
\newtheorem{corollary}[theorem]{Corollary}
\newtheorem{definition}[theorem]{Definition}
\theoremstyle{definition}
\newtheorem{example}[theorem]{Example}
\newtheorem{remark}[theorem]{Remark}

\newtheorem{thmalph}{Theorem}

\newtheorem{setting}[theorem]{Setting}

\newcommand{\g}{\mathfrak{g}}

\renewcommand{\sl}{\mathfrak{sl}}

\newcommand{\rk}{\operatorname{rank}}
\newcommand{\Hom}{\operatorname{Hom}}
\newcommand{\End}{\operatorname{End}}

\newcommand{\Aut}{\operatorname{Aut}}
\newcommand{\AAut}{\mathscr{A}\textit{ut}}
\renewcommand{\ker}{\operatorname{Ker}}
\newcommand{\im}{\operatorname{Im}}
\newcommand{\tr}{\operatorname{tr}}
\newcommand{\id}{\operatorname{id}}
\newcommand{\C}{\mathbb{C}}
\newcommand{\N}{\mathbb{N}}
\newcommand{\Z}{\mathbb{Z}}
\newcommand{\Q}{\mathbb{Q}}

\renewcommand{\P}{\mathbb{P}}

\renewcommand{\S}{\mathcal{S}}
\newcommand{\U}{\mathcal{U}}
\newcommand{\I}{\mathscr{I}}
\newcommand{\gl}{\mathfrak{gl}}
\newcommand{\M}{\mathcal{M}}

\newcommand{\CCC}{\mathscr{C}}

\newcommand{\LL}{\mathcal{L}}
\newcommand{\lb}{\llbracket}
\newcommand{\rb}{\rrbracket}
\newcommand{\fusion}[3]{{\binom{#3}{#1\;#2}}}

\def\O{\mathcal{O}}
\def\Ord{\mathsf{Ord}}
\def\Adm{\mathsf{Adm}}
\def\Mod{\mathsf{Mod}}
\def\fin{\mathsf{fin}}
\def\L{\mathsf{L}}
\def\R{\mathsf{R}}
\def\Om{\Omega}
\def\Pr{\mathrm{Pr}}
\def\la{\lambda}
\def \al{\alpha}
\def \b{\beta}

\def \h{\mathfrak{h}}

\def\vac{\mathbf{1}}
\def\spn{\mathrm{span}}

\def\bs{\backslash}
\def\om{\omega}
\def\o{\otimes}

\def\wt{\mathrm{wt}}
\def\ra{\rightarrow}
\def\op{\oplus}
\def\n{\mathfrak{n}}
\def\ad{\operatorname{ad}}
\def\scrU{\mathscr{U}}
\def\fA{\mathfrak{A}}
\def\fB{\mathfrak{B}}
\def\<{\langle}
\def\>{\rangle}
\def\ssq{\subseteq}

\def\rN{\mathrm{N}}

\def\Id{\mathrm{Id}}
\def\ds{\dots}
\def\A{\mathsf{A}}
\def\fa{\mathfrak{a}}
\def\fb{\mathfrak{b}}
\def\Res{\mathrm{Res}}

\def\si{\sigma}

\def\VV{\mathbb{V}}
\def\reg{\mathrm{reg}}

\title[Coinvariants of admissible-level affine VOAs]{Vector Bundles of Coinvariants for Admissible Affine Vertex Operator Algebras}

\author[V. Alekseev, J. Liu]{Viktor Alekseev, Jianqi Liu}
\begin{document}
    \begin{abstract}
        Simple affine vertex operator algebras $L_k(\g)$ at non-integral admissible levels $k$ are generally neither $C_2$-cofinite nor rational. Despite the absence of these standard finiteness conditions, we prove that the sheaf of coinvariants (and, dually, of conformal blocks) of ordinary modules in the category $\O_k$ forms a vector bundle over the moduli space $\overline{\M}_{0,n}$ of stable $n$-pointed genus-zero curves. The proof relies on establishing the finite-dimensionality of genus-zero coinvariants, a vanishing theorem that isolates ordinary admissible modules at the boundary, and a smoothing theorem enabled by partial strong identity elements in the mode transition algebra. Consequently, admissible affine vertex operator algebras supply a rich class of non-rational, non-$C_2$-cofinite examples in which vector-bundle behavior is preserved for a well-behaved subcategory of modules.
    \end{abstract}
	\maketitle
\section{Introduction}\label{sec:intro}

Vertex operator algebras (VOAs) provide an algebraic framework for two-dimensional conformal field theory, and their modules give rise to spaces of coinvariants and conformal blocks attached to pointed algebraic curves. When studied in families, these spaces form sheaves over the moduli stack of pointed curves. A central question is to determine when these sheaves are vector bundles. This property has strong consequences: these vector bundles possess well-defined ranks, degrees, and Chern classes that can be computed explicitly, thereby bridging the geometry of $\overline{\M}_{g,n}$ and representation theory. This framework feeds directly into cohomological field theories and modular functors.

Sheaves of coinvariants for admissible modules over a strongly rational VOA (that is, a rational, $C_2$-cofinite, and self-contragredient VOA) are known to satisfy the factorization theorem and form vector bundles over the moduli stack $\overline{\M}_{g,n}$~\cite{DGT2}.

Affine VOAs at non-integral admissible levels are prominent examples outside of the classical framework because they are neither $C_2$-cofinite nor rational. Far from being pathological, these VOAs possess highly regular structures: they are quasi-lisse~\cite{A15, AvEM22}, they have a semisimple category $\O_k$~\cite{A16}, and the subcategory $\O_{k,\Ord}$ of ordinary modules carries a braided tensor structure~\cite{creutzig2018braided} and is rigid in types $ADE$~\cite{Creutzig19}. 
Therefore, they are ideal candidates to test which geometric properties survive. As our main theorem shows, after restricting to the category $\O_{k,\Ord}$ of ordinary modules, the sheaves of coinvariants and conformal blocks form vector bundles in genus zero.

\begin{thmalph}[Theorem~\ref{thm:mainA}]\label{thm:A}
Let $\g$ be a complex simple Lie algebra, let $L_k(\g)$ be a simple affine VOA at an admissible level $k$, and let $n\geq 3$. For any $n$-tuple $M^\bullet = (M^1,\dots, M^n)$ of $L_k(\g)$-modules in the category $\O_{k,\Ord}$, the sheaf of coinvariants $\VV_{0,n}(L_k(\g),M^\bullet)$ is coherent, locally free, and globally generated over $\overline{\M}_{0,n}$. In particular, the dual sheaf of conformal blocks is a vector bundle on $\overline{\M}_{0,n}$.
\end{thmalph}

The proof replaces strong rationality with three structural pillars. First, degree-one strong generation ensures the coherence of ordinary-module coinvariants in genus zero~\cite[Corollary A]{DG}. Second, the central character decomposition of the Zhu algebra~\cite[Theorem 3.4]{AvE} and the complete reducibility of the modules in $\O_k$~\cite[Main theorem]{A16} are used to show that the factorization theorem descends to the subcategory $\O_{k,\Ord}$ of ordinary modules. Finally, we prove that the smoothing morphism is well-defined on the subcategory $\O_{k,\Ord}$ via partial strong identity elements, even though strong identity elements need not exist in the full mode transition algebra $\fA$.

\subsection{Outline of the proof}

We first fix the notation. Let $\g$ be a complex simple Lie algebra and let $L_k(\g)$ be a simple affine VOA at an admissible level $k$. The category of admissible $L_k(\g)$-modules is denoted by $\Adm(L_k(\g))$, $\O_k\subset \Adm(L_k(\g))$ is the associated BGG category $\O$ \cite{A16} and $\O_{k,\Ord}\subset \O_k$ is the subcategory of ordinary modules. We write $(C,P_\bullet,t_\bullet)$ for a curve $C$ together with marked points $P_\bullet = (P_1, \dots, P_n)$ and formal coordinates $t_\bullet = (t_1, \dots, t_n)$ at those points. Attached to $C\setminus P_\bullet$ is a chiral Lie algebra acting on the tensor product $M^1\otimes\cdots\otimes M^n$, which we also denote by $M^\bullet$, and the space of coinvariants $\VV(V,M^\bullet)_{(C,P_\bullet,t_\bullet)}$ is the cokernel of that action.

The logical structure of the proof of Theorem~\ref{thm:A} is summarized in Figure~\ref{fig:proof_columns}, organized around the three pillars announced above: \emph{coherence}, \emph{factorization}, and \emph{smoothing}.

\textbf{Coherence.} Since $L_k(\g)$ is strongly generated in degree one and its degree-one subspace is the finite-dimensional simple Lie algebra $\g$, the sheaf of coinvariants attached to modules in $\O_{k,\Ord}$ is coherent over $\overline{\M}_{0,n}$ by~\cite[Corollary A]{DG}. Coherence is used at two points of the argument: it converts constancy of the fiber dimension into local freeness (Proposition~\ref{prop:flatness}), and it allows Nakayama's lemma to lift a factorization isomorphism on the central fiber of a smoothing family to an isomorphism over the whole family (Proposition~\ref{prop:smoothing}).

\textbf{Factorization.} The rationality of $L_k(\g)$ in the category $\O_k$~\cite[Main theorem]{A16} lets us identify $\O_{k,\Ord}$ with a subcategory of modules over the Zhu algebra $\A$ (Corollary~\ref{coro:equiofcats}), reducing statements about infinite-dimensional $L_k(\g)$-modules to statements about the finite-dimensional $\g$-modules $L(\la)$, $\la \in \Pr^k \cap P$. Combined with the central character decomposition of $\A$~\cite[Theorem 3.4]{AvE}, this equivalence drives our vanishing theorem (Theorem~\ref{thm:fusion_closedness}): genus-zero coinvariants vanish whenever a non-ordinary module is inserted alongside ordinary ones, because the corresponding central characters cannot match. Iterating this vanishing statement over the nodes of a stable curve produces the restricted factorization theorem (Theorem~\ref{thm:restricted_factorization}), which shows that the usual factorization formula for coinvariants---a sum over \emph{all} simple modules of $L_k(\g)$---collapses to a sum over the ordinary weights $\Pr^k \cap P$ alone once every inserted module is ordinary.

\textbf{Smoothing.} The remaining obstruction is that the classical smoothing morphism, which relates coinvariants on a family of curves to coinvariants on its central fiber, was only constructed when the VOA admits strong identity elements in the full mode transition algebra $\fA$. We circumvent this requirement by working with the smaller subcategory $\O_{k,\Ord}\subset \O_k$ of ordinary modules. Using the equivalence of categories for $\O_{k,\Ord}$, we construct \emph{partial} strong identity elements in the mode transition algebra $\fA' \subset \fA$ (Proposition~\ref{prop:partial_strong_identity}). By Definition/Lemma~\ref{def:strong_identity}, these satisfy the strong identity equations, which is all that the criterion of~\cite[Proposition 5.1.2]{DGK2} requires, so the smoothing morphism is already well-defined (Proposition~\ref{prop:smoothing_well_defined}).

Theorem~\ref{thm:A} is then proved as follows. On the central fiber of a one-parameter smoothing of a nodal curve, the restricted factorization theorem identifies the coinvariants of the nodal curve with those of its normalization, with the ordinary bimodule $\fA'$ inserted at the two preimages of the node. The partial strong identity elements define a smoothing morphism over the family whose restriction to the central fiber is precisely this factorization isomorphism; by Proposition~\ref{prop:smoothing}, which rests on coherence, the smoothing morphism is therefore an isomorphism of sheaves over the whole family. This is the smoothing theorem (Theorem~\ref{thm:main_smoothing}). Since $\overline{\M}_{0,n}$ is stratified by the number of nodes, induction on this stratification shows that the dimension of the space of coinvariants $[M^\bullet]_{(C,P_\bullet)}$ is constant across $\overline{\M}_{0,n}$, which by Proposition~\ref{prop:flatness} is equivalent to local freeness of the sheaf of coinvariants---establishing Theorem~\ref{thm:A}.

\begin{figure}[htbp]
\centering
\begin{tikzpicture}[
    lane/.style={fill=black!3, rounded corners=8pt},
    input/.style={
        draw=cyan!60!black,
        fill=cyan!4,
        rectangle,
        rounded corners=5pt,
        minimum width=4.6cm,
        minimum height=1.3cm,
        text width=4.3cm,
        align=center,
        font=\small,
        line width=0.7pt
    },
    ours/.style={
        draw=violet!60!black,
        fill=violet!4,
        rectangle,
        rounded corners=5pt,
        minimum width=4.6cm,
        minimum height=1.3cm,
        text width=4.3cm,
        align=center,
        font=\small,
        line width=0.7pt
    },
    mainthm/.style={
        draw=orange!80!black,
        fill=yellow!15,
        rectangle,
        rounded corners=6pt,
        minimum width=4.6cm,
        minimum height=1.5cm,
        text width=4.3cm,
        align=center,
        font=\small\bfseries,
        line width=1.3pt
    },
    arrow/.style={-{Stealth[scale=1.15]}, thick, draw=black!55, rounded corners=5pt}
]

    \fill[lane] (-2.45, 2.3) rectangle (2.45, -5.0);
    \fill[lane] (2.9, 2.3) rectangle (7.8, -11.1);
    \fill[lane] (8.25, 2.3) rectangle (13.15, -5.0);

    \node[font=\footnotesize\bfseries, text=black!60, align=center] at (0, 1.7) {COHERENCE};
    \node[font=\footnotesize\bfseries, text=black!60, align=center] at (5.35, 1.7) {FACTORIZATION};
    \node[font=\footnotesize\bfseries, text=black!60, align=center] at (10.7, 1.7) {SMOOTHING};

    \node[input] (c1_1) at (0, 0.7) {Degree 1 Strong\\Generation~\cite{FZ92}};
    \node[input] (c1_2) at (0, -3.5) {Coherence on\\$\overline{\M}_{0,n}$~\cite{DG}};

    \node[input] (c2_1) at (5.35, 0.7) {Semisimplicity of\\Category $\O_k$~\cite{A16}};
    \node[ours] (c2_2) at (5.35, -1.4) {Equivalence of Categories\\ for $\O_{k}$ (Thm~\ref{thm:equiofcats})};
    \node[ours] (c2_3) at (5.35, -3.5) {Vanishing Theorem\\for $\O_{k,\Ord}$ (Thm~\ref{thm:fusion_closedness})};
    \node[ours] (c2_4) at (5.35, -5.6) {Restricted Factorization\\ in Genus Zero (Thm~\ref{thm:restricted_factorization})};
    \node[ours] (c2_5) at (5.35, -7.7) {Smoothing Theorem\\in Genus Zero (Thm~\ref{thm:main_smoothing})};
    \node[mainthm] (c2_6) at (5.35, -10.0) {THEOREM~\ref{thm:A}\\[3pt]\small $\VV_{0,n}$ \textmd{is a globally}\\ \textmd{generated vector bundle}};

    \node[input] (c3_1) at (10.7, 0.7) {Zhu Algebra\\Decomposition~\cite{AvE}};
    \node[ours] (c3_2) at (10.7, -1.4) {Existence of the\\Partial Strong Identity\\Elements (Prop~\ref{prop:partial_strong_identity})};
    \node[ours] (c3_3) at (10.7, -3.5) {Well-Defined Smoothing\\Map~\cite{DGK2}\\(Prop~\ref{prop:smoothing_well_defined})};


    \draw[arrow] (c1_1) -- (c1_2);
    \draw[arrow] (c1_1.south) |- (c2_2.west);
    \draw[arrow] (c1_2.south) |- (c2_5.west);
    \draw[arrow] (c1_2.south) |- (c2_6.west);

    \draw[arrow] (c2_1) -- (c2_2);
    \draw[arrow] (c2_2) -- (c2_3);
    \draw[arrow] (c2_3) -- (c2_4);
    \draw[arrow] (c2_4) -- (c2_5);
    \draw[arrow] (c2_5) -- (c2_6);

    \draw[arrow] (c2_2.east) -- (c3_2.west);
    \draw[arrow] (c3_1) -- (c3_2);
    \draw[arrow] (c3_1.west) -- ++(-0.4,0) |- (c2_3.east);
    \draw[arrow] (c3_2) -- (c3_3);
    \draw[arrow] (c3_3.south) |- (c2_5.east);
\end{tikzpicture}
\caption{Logical structure of the proof of Theorem~\ref{thm:A}; arrows indicate logical dependence. Cyan boxes are results established in the literature, violet boxes are original results of this paper, and the highlighted bottom box is the main theorem.}
\label{fig:proof_columns}
\end{figure}
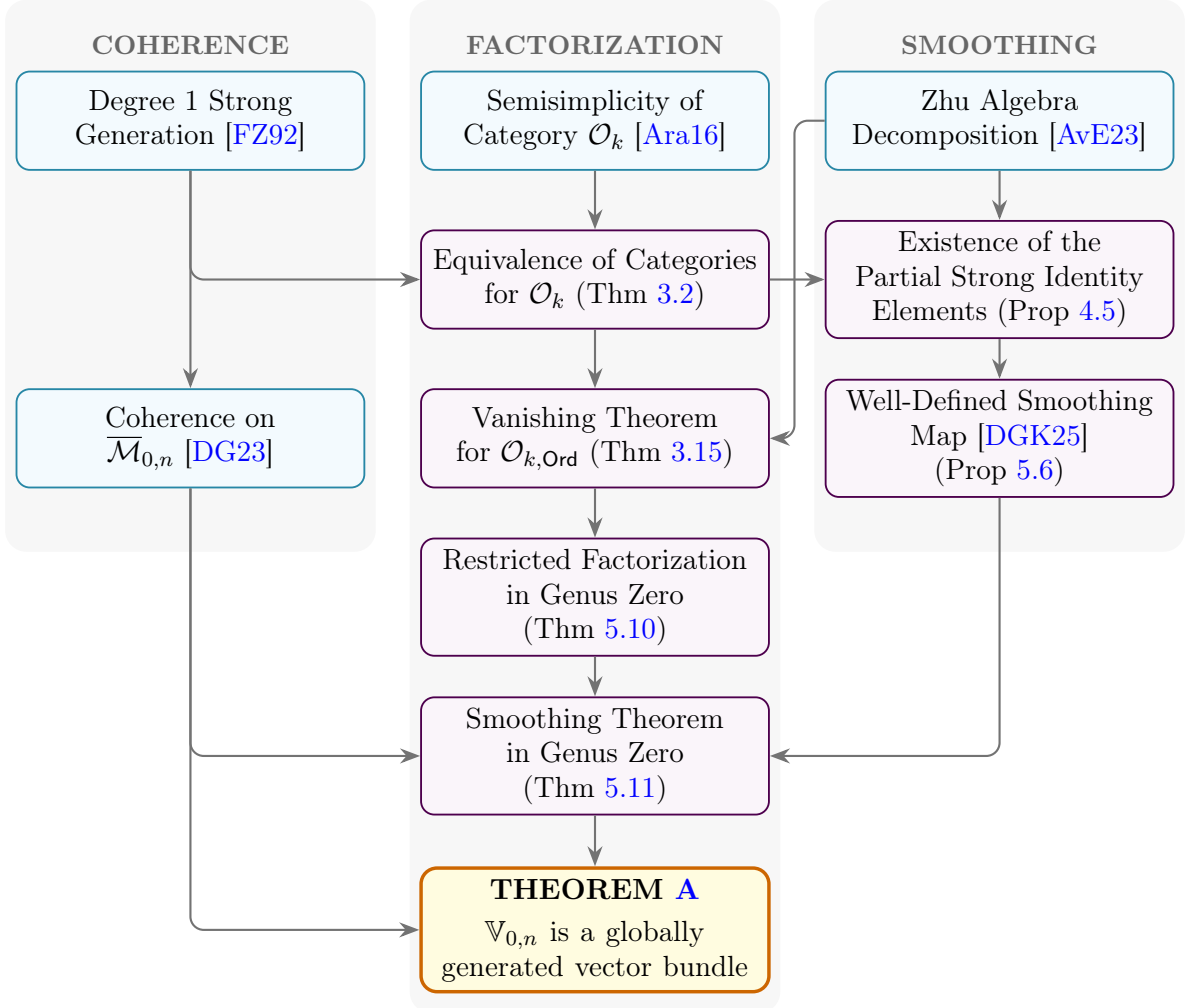

\subsection{Consequences}

Because Theorem~\ref{thm:A} produces an honest vector bundle rather than merely a coherent sheaf, the invariants attached to $\VV_{0,n}(L_k(\g),M^\bullet)$---its rank and its first Chern class---are well-defined; in particular the rank is constant across all of $\overline{\M}_{0,n}$, including the boundary strata where the curve degenerates. In Section~\ref{sec:fusion_rules} we pursue this explicitly: the restricted factorization theorem (Theorem~\ref{thm:restricted_factorization}) expresses the rank as a sum, over labelings of a trivalent tree by ordinary admissible weights, of products of three-point fusion numbers (Theorem~\ref{thm:rank_formula}), and the partial strong identity elements make the boundary residues of the projectively flat connection computable, showing that the classical degree formula, previously available only for strongly rational VOAs, remains valid in this setting (Proposition~\ref{prop:degree_formula}). For simply-laced $\g$ we show that the resulting classes are proportional to the classical ones at the integrable level $k'=p-h^\vee$, the conformal dimensions being uniformly rescaled by $q$ (Proposition~\ref{prop:qscaling}). This gives explicit rank and degree formulas for coinvariant bundles attached to a non-rational, non-$C_2$-cofinite VOA, a setting in which such formulas were previously out of reach.

More broadly, Theorem~\ref{thm:A} shows that the vector-bundle behavior underlying cohomological field theories and modular functor constructions does not require the full strength of strong rationality: it suffices that the category of inserted modules be closed under fusion in the sense of the vanishing theorem (Theorem~\ref{thm:fusion_closedness}), namely that genus-zero coinvariants vanish as soon as a module from outside the category is inserted alongside modules from inside it. We expect the same strategy---coherence, a central-character vanishing theorem for restricted factorization, and partial strong identity elements for smoothing---to apply to other VOAs whose ordinary modules form a well-behaved subcategory; we return to this direction in the Discussion (Section~\ref{sec:discussion}).

\tableofcontents

\section{Preliminaries}\label{sec:prelim}
In this section, we recall the basic definitions and constructions for vertex operator algebras~\cite[Sections 3 and 4]{LL04}.
\subsection{Vertex operator algebras}

\begin{definition}
    A \textbf{(CFT-type) vertex operator algebra (VOA)} is a quadruple $(V,\vac,\omega,Y(\cdot, z))$, where
    \begin{itemize}
        \item $V = \bigoplus_{n\in \N} V_n$ is a non-negatively graded vector space with $\dim V_n < \infty$;
        \item $ V_0=\C\vac$, where $\vac$ is the vacuum vector;
        \item $\omega$ is an element of $V_2$ called the conformal vector;
        \item $Y(\cdot, z): V \rightarrow \End(V)\lb z, z^{-1} \rb $ is a linear map $a\mapsto Y(a,z) := \sum_{m \in \Z} a_{(m)}z^{-m-1}$, where $Y(a,z)$ is called the vertex operator associated with $a$,
    \end{itemize}
    satisfying the following axioms:
    \begin{itemize}
        \item For all $a,b\in V$, $a_{(m)}b=0$ for $m\gg0$;
        \item $Y(\vac,z)=\id_V \in \End V \lb z,z^{-1} \rb$ and $Y(a,z)\vac \in a + zV\lb z \rb$;
        \item For all $a,b \in V$, the corresponding vertex operators are local, that is, there is $N \in \N$ such that
        \begin{equation}
            (z_1-z_2)^N[Y(a,z_1), Y(b,z_2)] = 0 \in \End V \lb z_1,z_1^{-1}, z_2, z_2^{-1} \rb.
        \end{equation}
        \item For all $a,c\in V$ there is $N\in\N$ such that for any $b\in V$ we have the weak associativity of vertex operators
        \begin{equation}
            (z_1+z_2)^N Y(Y(a,z_1)b,z_2)c = (z_1+z_2)^N Y(a,z_1+z_2)Y(b,z_2)c.
        \end{equation}
        \item For $Y(\omega,z) = \sum_{m \in \Z} \omega_{(m)}z^{-m-1}=\sum_{n \in \Z} L(n)z^{-n-2}$, with $\om_{(n+1)}=L(n)$, the Virasoro algebra relations hold
        \begin{equation}
            [L(m),L(n)] = (m-n) L(m+n) + \frac{c}{12} \delta_{m+n,0} (m^3-m) \id_V,
        \end{equation}
        where $c\in \C$ is the central charge of the VOA $V$. Moreover, for all homogeneous $a\in V$,
        \begin{equation}
            \begin{split}
              L(0) a &= (\wt a)a,\\
                Y(L(-1)a,z) &= \frac{d}{dz} Y(a,z).
            \end{split}
        \end{equation}
    \end{itemize}
\end{definition}
\begin{remark}
Weak associativity together with locality (weak commutativity) is equivalent to the formal Jacobi identity~\cite[Proposition 3.4.3]{LL04}:
    \begin{equation}
        \begin{split}
            &z_0^{-1} \delta\left( \frac{z_1-z_2}{z_0}\right) Y(a,z_1)Y(b,z_2) - z_0^{-1} \delta\left( \frac{z_2-z_1}{-z_0}\right) Y(b,z_2)Y(a,z_1) \\&= z_2^{-1} \delta\left(\frac{z_1-z_0}{z_2}\right) Y(Y(a,z_0)b,z_2),
        \end{split}
    \end{equation}
    where $\delta(z):= \sum_{n\in \Z} z^n$. The formal Jacobi identity also has the equivalent component form
    \begin{equation}\label{eq:componentJacobi}
    \sum_{i \ge 0} (-1)^i \binom{l}{i} \left( a_{(m+l-i)} b_{(n+i)} - (-1)^l b_{(n+l-i)} a_{(m+i)} \right) = \sum_{i \ge 0} \binom{m}{i} (a_{(l+i)}b)_{(m+n-i)},
\end{equation}
for all $a, b \in V$ and integers $l, m, n\in\Z$.
\end{remark}
\begin{definition}[{\cite[p.2]{DLM2}}]
    A \textbf{weak $V$-module} is a $\C$-vector space $W$ together with a linear map
    \begin{equation}
        \begin{split}
            Y^W(\cdot, z): V \rightarrow \End W\lb z, z^{-1} \rb,\\
            a \mapsto Y^W(a,z):=\sum_{m \in \Z} a^W_{(m)}z^{-m-1},
        \end{split}
    \end{equation}
    such that $a^W_{(m)}w=0$ for $m\gg 0$ (for all $a\in V$, $w\in W$), $Y^W(\vac,z)=\id_W$, and the formal Jacobi identity of the preceding remark holds with $Y(a,z_1)$, $Y(b,z_2)$ and $Y(Y(a,z_0)b,z_2)$ replaced by $Y^W(a,z_1)$, $Y^W(b,z_2)$ and $Y^W(Y(a,z_0)b,z_2)$, respectively.
\end{definition}

We briefly recall the following notation and constructions associated with a VOA $V$; see~\cite{borcherds86,FZ92,FHL93,DLM,zhu1996modular} for more details.
\begin{enumerate}
\item The category $\Adm(V)$ of \textbf{admissible} $V$-modules~\cite[Definition 2.3]{DLM2} consists of weak $V$-modules $W$ that admit a grading $W = \bigoplus_{n \in \N} W(n)$ such that $a^W_{(m)} W(n) \subset W(n+\wt a - m -1)$ for every homogeneous $a\in V$. In particular, the operator $a^W_{(n)}\in \End(W)$ has degree $\deg(a_{(n)})=\wt a-n-1$; $V$ is called \textbf{rational} if $\Adm(V)$ is semisimple.
\item The category $\Ord(V)$ of ordinary $V$-modules~\cite[Definition 2.1]{DLM2} consists of weak $V$-modules $M$ that admit a direct sum decomposition $M=\bigoplus_{\la\in \C} M_\la$ such that $M_\la$ is a finite-dimensional eigenspace for $L(0)$ with eigenvalue $\la\in \C$ and, for any $\la\in \C$, $M_{\la+n}=0$ for $n\in \Z$ and $n\ll0$; in particular, any ordinary module is admissible~\cite[Remark 2.4]{DLM2}. If there is a complex number $\la_M$ such that $L(0)v = (\la_M + \deg (v)) v$ for every homogeneous $v \in M$, then $\la_M$ is called the \textbf{conformal dimension} of $M$. Moreover, any simple ordinary $V$-module has a conformal dimension~\cite[p.128]{FZ92}.
\item For $n\geq 2$, $V$ is called \textbf{$C_n$-cofinite}~\cite[p.10]{Li1999} if $C_n(V):=\operatorname{span}_\C\{a_{(-n)}b\mid a,b\in V\}$ has finite codimension in $V$. For $n=1$ the same span is all of $V$, since $\vac_{(-1)}b=b$; following~\cite{KarelLi99}, one sets instead $C_1(V):=\operatorname{span}_\C\{a_{(-1)}b,\ L(-1)a\mid a,b\in V_+\}$, where $V_+=\bigoplus_{d>0}V_d$, and $V$ is called \textbf{$C_1$-cofinite} if $C_1(V)$ has finite codimension in $V$. If $V$ is of CFT-type and strongly generated by $V_1$, then $V$ is $C_1$-cofinite: a monomial $a^1_{(-n_1)}\cdots a^r_{(-n_r)}\vac$ with $a^s\in V_1$, $n_s\geq 1$, $r\geq 1$, lies in $V_1$ if $r=1$ and $n_1=1$, and otherwise in $C_1(V)$, since $a_{(-n)}b=\frac{1}{(n-1)!}\big(L(-1)^{n-1}a\big)_{(-1)}b$ with $L(-1)^{n-1}a\in V_+$; hence $V=\C\vac+V_1+C_1(V)$.
\item The \textbf{Borcherds Lie algebra} of $V$~\cite[Section 8]{borcherds86} is given by
\[	L(V)=(V\o \C[t,t^{-1}])/\nabla (V\o \C[t,t^{-1}]),\]
	where $\nabla =L(-1)\o \Id+\Id\o (d/dt)$. Denote $a_{[n]}=a\o t^n+\im(\nabla)$. Then the Lie bracket of $L(V)$ is given by
    \[
    	[a_{[m]},b_{[n]}]=\sum_{j\geq0} \binom{m}{j} (a_{(j)}b)_{[m+n-j]},\quad a,b\in V,\ m,n\in \Z.
    \]
We write $\deg(a_{[n]})=\wt a-n-1$ for all $a_{[n]}\in L(V)$.
\item The \textbf{Zhu algebra $\A=A(V)$} of $V$~\cite[Theorem 2.1.1]{zhu1996modular} is the quotient space
$\A=V/O(V),$ where
\[	O(V)=\spn\left\{ a\circ b=\Res_{z=0}Y(a,z)b\frac{(1+z)^{\wt a}}{z^2}: a,b\in V\right\}.
		\]
Denote $[a]=a+O(V)$ in $\A$. Then $\A$ is an associative algebra with respect to the product
		\begin{equation}\label{eq:prodAV}
			[a]\ast [b]:=\Res_{z=0} [Y(a,z)b]\frac{(1+z)^{\wt a}}{z},\quad a,b\in V.
		\end{equation}
		Denote the category of left $\A$-modules by $\mathsf{Mod}(\A)$.
\end{enumerate}

\begin{remark}
It is convenient to use the notation
\[
J_n(a):=a_{[\wt a+n-1]},\quad a\in V,\ n\in \Z,
\]
since $\deg(J_n(a))=-n$. In the rest of the paper, we use the two notations $J_n(a)$ and $a_{[m]}$ for the mode operators interchangeably.  
\end{remark}

\subsection{Mode transition algebras associated with a VOA}
In this subsection, we recall the construction of the universal enveloping algebra $\scrU=\scrU(V)$ associated with a VOA $V$; see~\cite{FZ92,DGK2}. We mainly adopt the construction and notation of~\cite[Section 2]{DGK2}.
\begin{definition}
Let $V$ be a vertex algebra. The \textbf{left, right, and finite ancillary Lie algebras}---denoted $\mathcal{L}(V)^\L$, $\mathcal{L}(V)^\R$, and $\mathcal{L}(V)^f$, respectively---are defined as the quotient spaces
\begin{equation}
    \mathcal{L}(V)^{\Box} = \frac{V \otimes S}{\im \nabla}, \qquad \text{where } \nabla = \omega_{(0)} \otimes \id_S + \id_V \otimes \frac{d}{dt},
\end{equation}
and the coefficient space $S$ is chosen to be $\mathbb{C}((t))$, $\mathbb{C}((t^{-1}))$, and $\mathbb{C}[t,t^{-1}]$, for $\Box = \L,\R,f$, respectively. Denoting the equivalence class of an element $a \otimes t^n \in V \otimes S$ by $a_{[n]}$, the Lie bracket is defined as
\begin{equation}
    [a_{[n]}, b_{[m]}] := \sum_{k \ge 0} \binom{n}{k} (a_{(k)}b)_{[n+m-k]},
\end{equation}
for all $a, b \in V$ and $n, m\in\Z$. By construction, $\mathcal{L}(V)^f$ is isomorphic to the Borcherds Lie algebra $L(V)$.
\end{definition}

\begin{remark}
    There are natural inclusions $\LL(V)^\L \hookleftarrow \LL(V)^f \hookrightarrow \LL(V)^\R$. We write $U^\L,U^\R,U^f$ for the universal enveloping algebras of $\LL(V)^\L,\LL(V)^\R, \LL(V)^f$. These are \emph{not} what we call the universal enveloping algebra of $V$: the latter is the completion $\scrU(V)$ constructed below.
\end{remark}

	Write $\U:=U^f=U(L(V))$ for short. The degrees of the elements in the Borcherds Lie algebra $L(V)$ give rise to a graded algebra structure on $\U$:
	\[
	\U=\bigoplus_{d\in \Z} \U_d,\quad \U_d=\spn\left\{a^1_{[n_1]}\cdots a^r_{[n_r]}\in \U: r\geq 0,\ \sum\limits_{i=1}^{r}\left(\wt a^i-n_i-1\right)=d \right\}.
	\]
	Let $\U_{\leq -n}=\sum_{d\leq -n} \U_d$, which makes $\U$ a split-filtered associative algebra $\U=\bigcup_{n\in \Z} \U_{\leq -n}$. Define
	\begin{equation}\label{eq:nei}
		\rN^n_\L \U=\U\cdot \U_{\leq -n}=\U\cdot L(V)_{\leq -n},\quad 	\rN^n_\R \U=\U_{\geq n}\cdot \U=L(V)_{\geq n}\cdot \U,
	\end{equation}
	where $L(V)_{\leq -n}=\spn\{a_{[k]}\in L(V): \deg(a_{[k]})\leq -n\}$; see~\cite[Lemma 2.4.2]{DGK2}. The subspace $L(V)_{\geq n}$ is defined similarly. Since the identity $1\in \U_0$ is contained in $\U_{\leq 0}$ and $\U_{\geq 0}$, we have $\rN^n_\L \U=\U=\rN^n_\R \U	$ if $n\leq 0$.

	The left ideals $\{\rN^n_\L \U:n\in \Z_{\geq 0}\}$ form a system of neighborhoods of $0$ in $\U$, which gives a canonical seminorm on $\U$; see~\cite[Definition A.6.1]{DGK2}. One can restrict these seminorms to the graded parts $\U_d$ of $\U$:
	\begin{equation}\label{eq:completion}
		\rN^n_\L \U_d:=(\U\cdot \U_{\leq -n})_d=\sum_{j\leq -n} \U_{d-j}\cdot \U_j,\quad  	\rN^n_\R \U_d=(\U_{\geq n}\cdot \U)_d=\sum_{i\geq n} \U_i\cdot \U_{d-i}.
	\end{equation}
	In particular, $\rN^n_\L \U_d=\rN^{n+d}_\R \U_d$ for any $d\in \Z$. Define the completion
	\[
	\widehat{\U}_d:=\varprojlim_n\frac{\U_d}{\rN^n_\L \U_d}=\varprojlim_n\frac{\U_d}{\rN^{n+d}_\R \U_d}\quad \text{and}\quad \widehat{\U}:=\bigoplus_{d\in \Z} \widehat{\U}_d.
	\]


	Let $J\subset \widehat{\U}$ be the graded ideal generated by the component form~\eqref{eq:componentJacobi} of the Jacobi identity for the Lie algebra elements $a_{[n]}$, and let $\bar{J}\ssq \widehat{\U}$ be the closure of $J$ with respect to the seminorm defined by the image of the neighborhoods~\eqref{eq:nei} in $\widehat{\U}$. Define
	\[\scrU=\scrU(V):=\widehat{\U}/\bar{J}=\bigoplus_{d\in \Z} \scrU_d.\]
	Then $\scrU$ is a graded complete seminormed associative algebra with respect to the canonical seminorm induced by the image of neighborhoods~\eqref{eq:nei}. $\scrU=\scrU^f$ is called the \textbf{(finite) universal enveloping algebra of the VOA $V$}. The left and right neighborhoods at $0$ of $\scrU$ are given by
	\begin{equation}\label{eq:nebors}
		\rN^n_\L \scrU=\scrU\cdot \scrU_{\leq -n}\quad \text{and}\quad  	\rN^n_\R \scrU=\scrU_{\geq n}\cdot \scrU,
	\end{equation}
	with $	\rN^{n+1}_\L \scrU_0=\sum_{j \geq n+1} \scrU_j\cdot \scrU_{-j}=	\rN^{n+1}_\R \scrU_0$, for any $n\geq 0$.

	The contragredient module action map $\theta$~\cite{FHL93} induces an anti-involution of topological associative algebras, defined for homogeneous $a\in V$ by
	\begin{equation}\label{eq:involution}
		\theta: \scrU\ra \scrU^{\mathrm{op}},\quad \theta(a_{[n]})=\sum_{j\geq 0}\frac{(-1)^{\wt a}}{j!} (L(1)^ja)_{[2\wt a-n-j-2]},
	\end{equation}
	such that $\theta(\scrU_n)=\scrU_{-n}$ and $\theta(\rN^n_\L \scrU)=\rN^n_\R \scrU$, for any $n\in \N$.

	\begin{lemma}[{\cite[Section 1.4]{FZ92}, \cite[Lemma 2.5.2]{DGK2}}]\label{lm5.1}
		The Zhu algebra $\A=A(V)$ can be identified with $\scrU_0/\rN^1_\L\scrU_0$ as an associative algebra. Moreover, the anti-involution $\theta:\scrU\ra \scrU^{\mathrm{op}}$~\eqref{eq:involution} induces an anti-involution of $\A$~\cite[(2.1.9)]{zhu1996modular}
        \[
        \theta: \A\ra \A^{\mathrm{op}},\quad \theta([a])=[e^{L(1)}(-1)^{L(0)}a],
        \]
        for any $[a]=a+O(V)\in \A$.
	\end{lemma}

	Note that the neighborhood $\rN^1_\L \scrU$ (resp. $\rN^1_\R \scrU$) contains all the negatively (resp. positively) graded subspaces of $\scrU$; see~\eqref{eq:completion} and~\eqref{eq:nebors}. Hence
	\begin{equation}\label{eq:gradationleftright}
		\scrU/	\rN^1_\L \scrU=\bigoplus_{n=0}^\infty(\scrU/	\rN^1_\L \scrU)_n,\quad \scrU/\rN^1_\R \scrU=\bigoplus_{m=0}^\infty (\scrU/\rN^1_\R \scrU)_{-m}.
	\end{equation}
	In particular, for any $m,n\geq 0$, we have
	\begin{align*}
		&(\scrU/	\rN^1_\L \scrU)_n=\scrU_n/	\rN^1_\L \scrU_n=\scrU_n \big/\sum_{j\leq -1} \scrU_{n-j}\scrU_j\\
		&=\widehat{\spn}\left\{a^1_{[n_1]}\cdots a^r_{[n_r]}+\rN^1_\L \scrU: \deg(a^i_{[n_i]})\geq 0\ \forall i,\ \sum_{i=1}^r \deg(a^i_{[n_i]})=n \right\},\\
		&(\scrU/	\rN^1_\R \scrU)_{-m}=\scrU_{-m}/	\rN^1_\R \scrU_{-m}=\scrU_{-m}\big/\sum_{j\geq 1} \scrU_{j}\scrU_{-m-j}\\
		&=\widehat{\spn}\left\{b^1_{[m_1]}\cdots b^s_{[m_s]}+\rN^1_\R \scrU: \deg(b^j_{[m_j]})\leq 0\ \forall j,\ \sum_{j=1}^s \deg(b^j_{[m_j]})=-m \right\},
	\end{align*}
    where $\widehat{\spn}$ means the closure of the subspace spanned by the bracketed elements.
	Then $(\scrU/	\rN^1_\L \scrU)_0=\scrU_0/\rN^1_\L \scrU_0\cong \A\cong \scrU_0/\rN^1_\R \scrU_0=(\scrU/	\rN^1_\R \scrU)_{0}$, in view of Lemma~\ref{lm5.1}.

One can similarly define the left and right universal enveloping algebras $\scrU^\L$ and $\scrU^\R$ by replacing the Borcherds Lie algebra $L(V)=\mathcal{L}(V)^f$ by $\mathcal{L}(V)^\L$ and $\mathcal{L}(V)^\R$, respectively; see~\cite[Section 2]{DGK2} for more details.

\begin{definition}[{\cite[Section 6]{DLM}, \cite[Definition 3.1.1]{DGK2}}]\label{def:Verma}
Let $V$ be a VOA, let $\A=A(V)$ be its Zhu algebra, and let $\scrU$ (resp. $\scrU^\L$, $\scrU^\R$) be its finite (resp. left, right) universal enveloping algebra.

Given a left $\A$-module $S$, the \textbf{left generalized Verma module associated with $S$} $\Phi^\L(S)$ is defined as the induced $\scrU$-module; it is an admissible $V$-module:
\begin{equation}\label{eq:PhiL}
    \Phi^\L(S) = \left( \frac{\scrU}{\rN_\L^1 \scrU} \right) \otimes_{\scrU_0} S=\bigoplus_{p=0}^\infty \Phi^\L(S)(p).
\end{equation}
Similarly, for a right $\A$-module $Z$, the \textbf{right generalized Verma module associated with $Z$} $\Phi^\R(Z)$ is defined as the graded right $\scrU$-module
\begin{equation}
    \Phi^\R(Z) = Z \otimes_{\scrU_0} \left( \frac{\scrU}{\rN_\R^1 \scrU} \right)=\bigoplus_{p=0}^\infty \Phi^\R(Z)(-p).
\end{equation}
\end{definition}
The functor $\Phi^\L(-): \Mod(\A)\ra \Adm(V)$ is the left adjoint of the ``highest-weight subspace'' functor $\Om(-):\Adm(V)\ra \Mod(\A)$, where \[\Om(W)=\{w\in W: a_{(n)}w=0\ \text{for all }a\in V,\ n\in \Z\ \text{with}\ \deg a_{(n)}<0\}.\]
In other words, there exists a natural isomorphism
\[
\Hom_{\Adm(V)}(\Phi^\L(S),W)\cong \Hom_{\A}(S,\Om(W)).
\]
\begin{remark}
    We note that a $\scrU$-module $W$ is a $V$-module if for each $w\in W$ and $a\in V$ one has $a^W_{(n)}w = 0$ for $n \gg 0$. For a left generalized Verma module, this condition is automatic since $\rN^1_\L\scrU$ contains all $a_{[n]}$ with $\deg a_{[n]} < 0$.
\end{remark}

\begin{definition}[{\cite[Definition 3.2.1]{DGK2}}]\label{def:mta}
    Let $V$ be a VOA and let $\A$ be its Zhu algebra. The \textbf{mode transition algebra} is defined as a two-sided induced module:
    \[
    \fA:=\Phi^\L(\Phi^\R(\A))=(\scrU/\rN^1_\L\scrU)\o_{\scrU_0}\A\o_{\scrU_0} (\scrU/\rN^1_\R\scrU).
    \]
  	\begin{enumerate}
			\item   The gradings~\eqref{eq:gradationleftright} give rise to a bigrading	$\fA=\bigoplus_{n,m\in \N} \fA_{n,-m}$, where
			\begin{equation}\label{eq:defAn-m}
				\fA_{n,-m}=(\scrU/	\rN^1_\L \scrU)_{n}\o _{\scrU_0} \A\o _{\scrU_0} (\scrU/\rN^1_\R \scrU)_{-m}.
			\end{equation}
			\item 	The associative product $\star$ on $\fA$ satisfies
			\[\fA_{i,-j}\star \fA_{k,-l}\ssq \delta_{j,k}\fA_{i,-l},\quad i,j,k,l\in \N.\]
			More precisely, for $\b^i\o a\o \al^{-j}\in \fA_{i,-j}$ and $\b^k\o b\o \al^{-l}\in \fA_{k,-l}$, we have
			\begin{equation}\label{eq:defofprodonmta}
				\left(\b^i\o a\o \al^{-j}\right) \star \left(\b^k\o b\o \al^{-l}\right)= \b^i\o a\cdot [\al^{-j}\b^k]_0\cdot b\o \al^{-l}\in \delta_{j,k}\fA_{i,-l},
			\end{equation}
				where $[\al^{-j}\b^k]_0$ is the image of $\al^{-j}\b^k\in \scrU_0$ under the epimorphism \[\scrU_0\twoheadrightarrow \A\cong \scrU_0/\rN^1_\L\scrU_0,\quad u\mapsto [u]_0.\] We identify $\A\o_{\scrU_0} \A\o_{\scrU_0} \A\cong \A$ via $a\o  [\al^{-j}\b^k]_0\o b\mapsto a\cdot [\al^{-j}\b^k]_0\cdot b$.
			\item 	In particular,
			$\fA_d=\fA_{d,-d}$ is an associative algebra under $\star$, called the \textbf{$d$-th mode transition algebra} associated with $V$.
			\item 	 $\fA_0=(\scrU/	\rN^1_\L \scrU)_0\o_{\scrU_0} \A\o_{\scrU_0} (\scrU/	\rN^1_\R \scrU)_{0}\cong \A$ as associative algebras, and $\fA_{i,-j}$ is an $(\fA_{i},\fA_j)$-bimodule under $\star$.
		\end{enumerate}
\end{definition}

\subsection{Sheaves of coinvariants and conformal blocks}

To investigate the deformation of coinvariants on families of nodal curves, we first recall the definition of coinvariants and conformal blocks. We adopt the notation in~\cite{DGT2,frenkel_ben_zvi_book,NT} for the definition of coinvariants and conformal blocks.

The group scheme $\Aut \O$ is defined through its functor of points. For a $\C$-algebra $R$, its $R$-points form the group
\begin{equation}
    \Aut \O(R):= \left\{z \mapsto \rho(z) = \sum_{n=1}^{\infty} a_nz^n\mid a_{n} \in R,\ a_1 \in R^\times \right\},
\end{equation}
each such $\rho$ being an $R$-algebra automorphism of $R\lb z\rb$ preserving the ideal $z R\lb z\rb$. Composition defines the group structure.

Let $\M_{g,n}$ be the moduli stack of smooth $n$-pointed curves of genus $g$, and let $\overline{\M}_{g,n}\supset \M_{g,n}$ be its compactification by stable $n$-pointed curves; it is a smooth and proper Deligne--Mumford stack, and its coarse moduli space is projective~\cite{Knudsen1983a,Knudsen1983b}. In the genus-zero case of interest to us, $\overline{\M}_{0,n}$ is a fine moduli space and a smooth projective variety~\cite[Theorem 6.1]{Knudsen1983b}. Let $\widetriangle{\M}_{g,n}$ be the moduli stack of stable $n$-pointed coordinatized curves, that is, of triples $(C, P_\bullet, t_\bullet)$ in which $P_\bullet= (P_1, \dots, P_n)$ is a tuple of points $P_i \in C$ and $t_\bullet = (t_1, \dots, t_n)$ is a tuple of formal coordinates $t_i \in \mathfrak{m}_{P_i} \setminus \mathfrak{m}^2_{P_i}$ at the $P_i$. The group $(\Aut \O)^n$ acts on $\widetriangle{\M}_{g,n}$ by
\begin{equation}
    \widetriangle{\M}_{g,n} \times \left(\Aut \O(\C)\right)^n \rightarrow \widetriangle{\M}_{g,n}, \qquad \big((C,P_\bullet,t_\bullet), \rho_\bullet\big) \mapsto \big(C,P_\bullet,(\rho_1(t_1), \dots, \rho_n(t_n))\big).
\end{equation}
This action makes $\widetriangle{\M}_{g,n}$ an $(\Aut \O)^n$-torsor over $\overline{\M}_{g,n}$, the projection being the morphism that forgets the formal coordinates
\begin{equation}
    \widetriangle{\M}_{g,n} \stackrel{(\Aut \O)^n}{\longrightarrow} \overline{\M}_{g,n}.
\end{equation}
\begin{definition}
    Let $\CCC \rightarrow \S$ be a family of curves. The \textbf{$\Aut \O$-torsor} $\AAut_\CCC \rightarrow \CCC$ is the torsor of formal coordinates along the fibers: for a $\C$-algebra $R$ and an $R$-point $x$ of $\CCC$, the fiber of $\AAut_\CCC$ over $x$ consists of the formal coordinates at $x$ in the fiber of $\CCC \rightarrow \S$ through $x$, and $\Aut \O$ acts on it simply transitively by reparametrization; see~\cite{frenkel_ben_zvi_book} and~\cite[Section 2]{DGT2}.
\end{definition}

\begin{definition}
    Let $V$ be a VOA and let $\CCC \rightarrow \S$ be a family of smooth curves. Then the \textbf{vertex algebra bundle} $V_\CCC$ is defined as the vector bundle associated with $\pi: \AAut_\CCC \rightarrow \CCC$
    \begin{equation}
        V_{\mathscr{C}} := V \times^{\Aut \O} \AAut_\mathscr{C},
    \end{equation}
    and the corresponding sheaf of sections $\mathscr{V}_\CCC$ is called the \textbf{sheaf of vertex algebras}
    \begin{equation}
        \mathscr{V}_\CCC := (V \otimes \pi_* \O_{\AAut_\CCC})^{\Aut \O}.
    \end{equation}
\end{definition}
The sheaf of vertex algebras can be constructed on singular curves as well; the construction is more involved, and we refer to~\cite[Section 2.5]{DGT2} for details. Here we only recall that, given a normalization $\eta: \widetilde{\CCC} \rightarrow \CCC$ of a curve $\CCC$ with a node at $Q$, with $Q_\pm = \eta^{-1}(Q)$, we have a short exact sequence of sheaves on $\CCC$
\begin{equation}
    0 \rightarrow \eta_* \left(\O_{\widetilde{\CCC}}(-Q_+-Q_-) \otimes \mathscr{V}_{\widetilde{\CCC}}\right) \rightarrow \eta_*\mathscr{V}_{\widetilde{\CCC}} \stackrel{\mathrm{ev}}{\longrightarrow} V_{Q_-} \oplus V_{Q_+} \rightarrow 0,
\end{equation}
where each $V_{Q_\pm}$ is a skyscraper sheaf at $Q_\pm$ with fiber $V$. Let $\pi_{\pm}$ be the natural projections of the last term onto $V_{Q_\pm}$, and let $\gamma: V_{Q_+}\ra V_{Q_-}$ be the isomorphism of fibers determined by the choice of formal coordinates at $Q_\pm$; see~\cite[Section 2.5]{DGT2}. We define the sheaf of vertex algebras as the subsheaf of $\eta_*\mathscr{V}_{\widetilde{\CCC}}$ given by
\begin{equation}
    \mathscr{V}_\CCC := \ker \left( \gamma \circ \pi_+ - \pi_- \right).
\end{equation}

\begin{definition}\label{df:chiralLie}
    The \textbf{sheaf of chiral Lie algebras} $\LL_{\mathscr{C}\setminus P_\bullet}(V)$  of a VOA $V$ over a family of stable curves $\CCC\rightarrow \mathcal{S}$ is defined as
    \begin{equation}
        \LL_{\CCC\setminus P_\bullet}(V) := H^0\left( \CCC \setminus P_\bullet, \frac{\mathscr{V}_\CCC \otimes \omega_\CCC}{\im \nabla}\right),
    \end{equation}
    where $\nabla: \mathscr{V}_\CCC \rightarrow \mathscr{V}_\CCC \otimes \omega_\CCC$ is the flat logarithmic connection defined in~\cite[Proposition 3.2]{DGT1}.
\end{definition}

Let $V$ be a vertex operator algebra, let $M^1, \dots, M^n$ be an $n$-tuple of admissible $V$-modules, and let $(\mathscr{C}, P_\bullet, t_\bullet)$ be a coordinatized family of curves over $\S = \operatorname{Spec}R$. The sheaf of chiral Lie algebras $\mathcal{L}_{\mathscr{C} \setminus P_\bullet}(V)$ acts naturally on the sheaf $M^\bullet_\S := \left(\bigotimes_{i=1}^n M^i\right) \otimes \O_\S$ as follows. For any $1 \le i \le n$, there is an inclusion of the formal punctured disk $\mathbb{D}^\times_\S \hookrightarrow \mathscr{C}$ centered at $P_i$, which induces a morphism
\[
\iota: \mathcal{L}_{\mathscr{C} \setminus P_\bullet}(V) \longrightarrow \bigoplus_{i=1}^n \mathcal{L}(V)_{\S}, \qquad \sigma \mapsto (\sigma_{P_1}, \dots, \sigma_{P_n}),
\]
where $\LL(V)_{\S}:=\LL(V)^\L\, \hat{\otimes}\, \O_{\S}$ and $\sigma_{P_i}$ is the restriction of $\sigma$ to the formal punctured disk $\mathbb{D}^\times_{\S}$ at $P_i$.
Since each component $\sigma_{P_i} \in \mathcal{L}(V)_{\S}$ admits a decomposition $\sigma_{P_i} = \sum_{j} (A_{i,j})_{[j]} \otimes f_{i,j}$ for $A_{i,j}\in V$ and $f_{i,j}\in \O_{\S}$, we can define the action of $\mathcal{L}_{\mathscr{C} \setminus P_\bullet}(V)$ on $M^\bullet_\S$ as
\begin{equation}\label{chiralLieact}
\begin{split}
    \sigma \cdot (w_1 \otimes\cdots\otimes w_n \otimes g) = \sum_{i=1}^n w_1 \otimes\cdots\otimes (\sigma_{P_i} \cdot w_i) \otimes\cdots\otimes w_n \otimes g \\
    = \sum_{i=1}^n \sum_{j} w_1 \otimes\cdots\otimes \left(A_{i,j}\right)^{M^i}_{[j]}\cdot w_i \otimes\cdots\otimes w_n \otimes f_{i,j}\cdot g
\end{split}
\end{equation}
for $w_1 \otimes\cdots\otimes w_n \otimes g \in M^\bullet_\S$.

\begin{definition}
     Let $V$ be a vertex operator algebra, let $M^1,\dots, M^n$ be an $n$-tuple of admissible $V$-modules, and let $(\mathscr{C}, P_\bullet, t_\bullet)$ be a coordinatized family of curves over $\S$. Denote $M^\bullet := \bigotimes_{i=1}^n M^i$ and $M^\bullet_\S:=M^\bullet \otimes \O_\S$.

     The \textbf{sheaf of coinvariants} $\VV(V,M^\bullet)_{(\CCC,P_\bullet,t_\bullet)}$ (sometimes denoted by $[M^\bullet]_{(\mathscr{C},P_\bullet,t_\bullet)}$) is defined as the cokernel of the action map $a$
\begin{equation}
    \mathcal{L}_{\mathscr{C} \setminus P_\bullet}(V) \otimes_{\O_{\mathcal{S}}} M^\bullet_\S \stackrel{a}{\longrightarrow} M^\bullet_\S \longrightarrow \VV(V,M^\bullet)_{(\mathscr{C},P_\bullet,t_\bullet)} \longrightarrow 0.
\end{equation}
    The \textbf{sheaf of conformal blocks} is the dual of the sheaf of coinvariants.
\end{definition}
\begin{remark}
    When $\mathcal{S} = \operatorname{Spec} \mathbb{C}$, this definition recovers the classical space of coinvariants $[M^\bullet]_{(C, P_\bullet, t_\bullet)}$ associated with a single coordinatized stable curve $C$.
\end{remark}
\begin{defproposition}[{\cite[Section 5.2]{DGT1}}]
    The sheaf of coinvariants, defined above over an affine base $\S$, extends to a sheaf $\widetriangle{\VV}_{g,n}(V,M^\bullet)$ over $\widetriangle{\M}_{g,n}$.
\end{defproposition}
\begin{defproposition}[{\cite[Section 8.7]{DGT2}}]\label{prop:isocoinvfiber}
    Let $V$ be a VOA and let $M^1,\dots,M^n$ be simple admissible modules with rational conformal dimensions. Then the sheaf $\widetriangle{\VV}_{g,n}(V,M^\bullet)$ over $\widetriangle{\M}_{g,n}$ descends to a sheaf $\VV_{g,n}(V,M^\bullet)$ over $\overline{\M}_{g,n}$.
\end{defproposition}
\begin{remark}\label{remark:coordinatefree}
    In particular, the last proposition implies that for any automorphism $\varphi: C\xrightarrow{\simeq} C$ of the smooth curve $C$ such that $\varphi(P_i)=Q_i$ for $1\leq i\leq n$, and for any choice of local coordinates $s_\bullet$ around $Q_\bullet$, there is a vector space isomorphism
\[
[M^\bullet]_{(C,P_\bullet,t_\bullet)}\cong [M^\bullet]_{(C,Q_\bullet,s_\bullet)}.
\]
\end{remark}

\subsection{Affine VOA $L_k(\g)$ at an admissible level}\label{subsec:affineVOA}
In this subsection, we recall the facts from the representation theory of affine Lie algebras that will be needed. The key fact used in the later sections is the Zhu algebra decomposition (Theorem~\ref{thm:AvE_result}).

Let $\g$ be a complex simple finite-dimensional Lie algebra. It has a triangular decomposition $\g = \n_- \oplus \h \oplus \n_+$.

Let $\Delta$ be the set of roots of $\g$ and let $\Delta_+$ and $\Delta_-$ be the sets of positive and negative roots, respectively. Let $\Pi=\{\al_1,\ds, \al_l\}$ be the set of simple roots. The Weyl group $W$ acts on $\h^*$ via the standard linear action denoted by $w(\la)$, as well as the \textbf{shifted (dot) action} defined by
\begin{equation}
    w \circ \la = w(\la + \rho) - \rho
\end{equation}
for $w \in W$ and $\la \in \h^*$. The coroot lattice is $\check{Q} = \sum_{\alpha \in \Delta} \Z \alpha^\vee$, where $\alpha^\vee := \frac{2 \alpha}{(\alpha,\alpha)}$ and $(\cdot,\cdot)$ is the invariant bilinear form on $\h^\ast$ normalized so that the highest root $\theta$ satisfies $(\theta, \theta)=2$. The highest short root $\theta_s$ satisfies $(\theta_s, \theta_s)=2/r^\vee$, where $r^\vee$ is the lacing number of $\g$. Let $P$ be the weight lattice of $\g$ and let $\check{P}$ be the coweight lattice. The basis of $\check{Q}$ formed by simple coroots is dual to the basis of $P$ formed by fundamental weights $\omega_i$. Analogously, fundamental coweights of $\check{P}$ are defined as the dual basis to the simple roots. Denote by $\rho$ the half sum of the positive roots of $\g$ and by $\rho^\vee$ the half sum of the positive coroots. The set of dominant weights of $\g$ is $P_+ = \sum_i \Z_{\ge0} \omega_i$ and the set of dominant coweights is $\check{P}_+ = \sum_i \Z_{\ge0} \omega_i^\vee$.

Denote by $L(\la)$ the simple $\g$-module of highest weight $\la \in \h^\ast$.
A primitive ideal in the universal enveloping algebra $U(\g)$ is by definition the annihilator of some irreducible $\g$-module. Denote $J_\la := \operatorname{Ann}_{U(\g)} L(\la)$.

The center of $U(\g)$ is denoted by $Z(\g)$ and the character $\gamma_\la: Z(\g) \rightarrow \C$ is defined by $z v_\la = \gamma_\la(z) v_\la$, where $v_\la \in L(\la)$ is the highest-weight vector.

Let $\hat{\g} = \g[t,t^{-1}] \oplus \mathbb{C}K$ be the affine Kac--Moody algebra associated with $\g$ and let $\hat{\mathfrak{h}} = \mathfrak{h} \oplus \mathbb{C}K$ be its standard Cartan subalgebra. The extended Cartan subalgebra is denoted by $\tilde{\h} = \h \oplus \mathbb{C}K \oplus \mathbb{C}d$, and their duals are denoted by $\hat{\mathfrak{h}}^* = \mathfrak{h}^* \oplus \mathbb{C}\Lambda_0 $ and $\tilde{\mathfrak{h}}^* = \mathfrak{h}^* \oplus \mathbb{C}\Lambda_0 \oplus \mathbb{C}\delta$, where $d$ is the degree operator, $\Lambda_0$ is the fundamental affine weight and $\delta$ is the imaginary root. The affine Lie algebra $\hat{\g}$ has a triangular decomposition $\hat{\g} = \hat{\n}_- \oplus \hat{\h} \oplus \hat{\n}_+$, where $ \hat{\n}_\pm = t^{\pm 1}\g[t^{\pm1}] \oplus \n_\pm$. For any $k\in \C$ one can construct a vertex algebra $V^k(\g)$~\cite[Section 2]{FZ92} as a (generalized) Verma $\hat{\g}$-module
\begin{equation}
    V^k(\g):= U(\hat{\g})\otimes_{U(\g[t] \oplus \C K)} \C_k,
\end{equation}
where $\C_k$ is the one-dimensional $\g[t] \oplus \C K$-module.
Let $h^\vee$ be the dual Coxeter number of $\g$. If $k \neq -h^\vee$, $V^k(\g)$ has a VOA structure with $Y(a(-1)\vac,z)=a(z)=\sum_{n\in \Z} a(n) z^{-n-1}$, where $a(n)=a\o t^n$ for $a\in \g$ and $n\in \Z$.
Moreover, $V^k(\g)$ has a unique simple quotient $L_k(\g)$, called the \textbf{simple affine VOA at level $k$}. 

For any weight $\la \in \mathfrak{h}^*$, let $\hat{\la} = \la + k\Lambda_0 \in \hat{\mathfrak{h}}^*$ denote its corresponding affine weight. Let $\widehat{W} = W \ltimes \check{Q}$ be the affine Weyl group and let $\overline{W} = W \ltimes \check{P}$ be the extended affine Weyl group of $\hat{\g}$.

The sets of real roots and positive real roots are defined as
\[
\begin{split}
    \hat{\Delta}^{re} &= \{\alpha + n \delta \mid \alpha \in \Delta, n \in \Z\}, \\
    \hat{\Delta}^{re}_+ &= \{\alpha + n \delta \mid \alpha \in \Delta_+, n \in \Z_{\ge0}\} \cup \{-\alpha + n \delta \mid \alpha \in \Delta_+, n \in \Z_{\ge1}\}.
\end{split}
\]
The integral root system $\hat{\Delta}(\hat{\la})$ associated with an affine weight $\hat{\la} \in \hat{\mathfrak{h}}^*$ is defined as
\[
\hat{\Delta}(\hat{\la}) = \{\alpha \in \hat{\Delta}^{re} \mid \langle \hat{\la}+\hat{\rho}, \alpha^{\vee} \rangle \in \mathbb{Z}\},
\]
where $\hat{\rho} = \rho + h^{\vee}\Lambda_0$.

An affine weight $\hat{\la} \in \hat{\mathfrak{h}}^*$ is called \textbf{admissible}\label{def:admissible} if it satisfies the following two conditions:
\begin{enumerate}
    \item \textbf{Regular dominance:}\label{cond:regdom} $\langle \hat{\la}+\hat{\rho}, \alpha^{\vee} \rangle \notin \{0, -1, -2, \dots \}$ for all positive real roots $\alpha \in \hat{\Delta}_{+}^{re}$.
    \item \textbf{Maximal rank condition:} $\mathbb{Q}\hat{\Delta}(\hat{\la}) = \mathbb{Q}\hat{\Delta}^{re}$.
\end{enumerate}

The set $\Pr^k \subset \mathfrak{h}^*$ of \textbf{level $k$ admissible weights} is defined by
\[
\Pr^k = \{ \la \in \mathfrak{h}^* \mid \hat{\la} = \la + k\Lambda_0 \text{ is admissible and } \exists y \in \overline{W} \text{ such that } \hat{\Delta}(\hat{\la}) = y ( \hat{\Delta}(k\Lambda_0)) \}.
\]
For $\la \in \h^*$ and $k\in \C$, denote by $\hat{L}_k(\la)$ the unique simple quotient of the Verma module of $\hat{\g}$ with highest weight $\hat\la=\la + k \Lambda_0 \in \hat{\h}^*$. The module $\hat{L}_k(\la)$ is called admissible if $\hat{\la}$ is admissible. A complex number $k$ is called admissible if $k \Lambda_0$ is admissible. The conformal dimension $a_\la$ of $\hat{L}_k(\la)$ is given by
\begin{equation}\label{eq:confdim_prelim}
a_\la = \frac{\langle \la, \la + 2\rho \rangle}{2(k + h^\vee)}.
\end{equation}
\begin{proposition}[{\cite[Theorems 2.1 and 2.2]{kac1989classification}, \cite[Proposition 1.1]{Kac2008}}]
    The number $k$ is admissible if and only if
    \begin{equation}
        k + h^\vee = \frac{p}{q} \text{ and } p \ge \begin{cases}
            h^\vee \quad \text{if } (r^\vee,q)=1,\\
            h \quad \ \,\text{if } (r^\vee, q) = r^\vee,
        \end{cases}
    \end{equation}
     for $p,q \in \Z_{\ge 1}$ coprime.
\end{proposition}
Following~\cite{kac1989classification,Kac2008}, an admissible level $k$ with $k+h^\vee=p/q$ is called \textbf{principal} if $(r^\vee,q)=1$ and \textbf{coprincipal} if $(r^\vee,q)=r^\vee$; since $r^\vee=1$ in types $A$, $D$, $E$, the coprincipal case occurs only for non-simply-laced $\g$.

Since every $\la\in \Pr^k$ is a rational weight, i.e., $\la\in \Q\o_\Z P$, and $k\in\Q$, the conformal dimension of $\hat{L}_k(\la)$ is rational for any $\la \in \Pr^k$.
\begin{theorem}[{\cite[Main theorem]{A16}}]\label{thm:Aramain}
    Let $k$ be an admissible number for $\hat{\g}$. Then a $\hat{\g}$-module in the category $\O_k$ is an $L_k(\g)$-module if and only if it is a direct sum of modules $\hat{L}_k(\la)$ with $\la\in \Pr^k$.
\end{theorem}

\begin{proposition}[{\cite[Proposition 2.4]{arakawa2017weight}}]\label{prop:arakawa_primitive}
    Let $\la,\mu \in \Pr^k$. Then:
    \begin{itemize}
        \item $J_\la$ is the unique maximal two-sided ideal containing $U(\g)\ker \gamma_\la$. In particular, $U(\g)/J_\la$ is a simple algebra.
        \item We have $J_\la = J_\mu$ if and only if $\la \in W \circ \mu$.
    \end{itemize}
\end{proposition}
For a fixed admissible level $k$, the set $\Pr^k$ of admissible weights of level $k$ is finite~\cite{kac1989classification}.
Let $[\Pr^k]:= \Pr^k/\sim$, where the equivalence classes are the orbits of the dot action of the finite Weyl group: $\mu \sim \la$ if $\mu \in W\circ\la$.
\begin{theorem}[{\cite[Theorem 3.4]{AvE}}]\label{thm:AvE_result}
Let $k\in\C$ be an admissible number.
Then the Zhu algebra $  \A = A(L_k(\g)) $ admits a decomposition:
\begin{equation}\label{eq:admAV}
\A \cong \prod_{[\la] \in [\Pr^k]} \A_\la, \qquad \A_\la:=U(\g)/J_\la,
\end{equation}
which depends only on the class of $\la$ (Proposition~\ref{prop:arakawa_primitive}). Moreover, $\A_\la$ is finite-dimensional if and
only if $[\la]\cap P_+\neq\emptyset$, in which case $\A_\la\cong L(\mu)\otimes_\C L(\mu)^\ast$
for $\mu\in[\la]\cap P_+$.
\end{theorem}

\section{Coinvariants of $L_k(\g)$-modules in the category $\O_k$}\label{sec:fusion}
\subsection{Categories of $L_k(\g)$-modules}
The category $\O_k$ is the full subcategory of the category of finitely generated $\hat{\g}$-modules of level $k$ whose objects satisfy the following conditions:
\begin{enumerate}
    \item $L(0)$ acts locally finitely;
    \item $\hat{\n}_+=\hat{\g}_+\op \mathfrak{n}_+$ acts locally nilpotently, where $\hat{\g}_+:=\g\o t\C[t]$;
    \item $\hat{\mathfrak{h}}$ acts semisimply.
\end{enumerate}

A $\hat{\g}$-module $M$ is called smooth if, for any $a\in \g$ and $v\in M$, one has $a(n)v=0$ for sufficiently large $n$. Recall that any smooth $\hat{\g}$-module of level $k$ is a weak $V^k(\g)$-module and vice versa~\cite[p.5]{A16}. Since $\hat{\n}_+$ acts locally nilpotently on the modules of $\O_k$, these modules are $\N$-gradable, and we regard the category $\O_k$ as a subcategory of $\Adm(V^k(\g))$. However, not every $V^k(\g)$-module is an $L_k(\g)$-module, which motivates the following definition.

\begin{definition}\label{def:cats}
We introduce the following categories of modules over the VOA $L_k(\g)$ and the associative algebra $\A = A(L_k(\g))$:
\begin{enumerate}
    \item Let $\O_k(L_k(\g))$ be the full subcategory of $\Adm(L_k(\g))$ whose objects lie in the category $\O_k$ as $\hat{\g}$-modules.
Then it follows immediately from Theorem~\ref{thm:Aramain} that the category $\O_k(L_k(\g))$ is a semisimple category with simple objects $\{\hat{L}_k(\la):\la\in \Pr^k\}$.
\item Let $\O^\fin_k$ be the full subcategory of $\Mod(\A)$ whose objects lie in the Bernstein--Gelfand--Gelfand category $\O$ of $\g$, that is, the category of finitely generated $\g$-modules on which $\mathfrak{n}_+$ acts locally nilpotently and $\h$ acts semisimply.

It was proved in~\cite[Corollary 3.2]{AvE} that $\O^\fin_k$ is a semisimple category with simple objects $\{L(\la): \la\in \Pr^k\}$.
\item Let $\O_{k,\Ord}$ be the full subcategory of $\O_k(L_k(\g))$ whose objects are ordinary $L_k(\g)$-modules, i.e., any object $M\in \O_{k,\Ord}$ has a direct sum decomposition $M=\bigoplus_{\la\in \C} M_\la$, where $M_\la$ is a finite-dimensional eigenspace for $L(0)$ with eigenvalue $\la\in \C$. Moreover, given $\la\in \C$, $M_{\la+n}=0$ for $n\in \Z$ and $n\ll0$.

In particular, $\O_{k,\Ord}$ is a semisimple category with simple objects $\{\hat{L}_k(\la):\la\in \Pr^k\cap P\}$: the module $\hat{L}_k(\la)$ is ordinary if and only if its bottom degree $L(\la)$ is finite-dimensional, that is, if and only if $\la\in P_+$, and $\Pr^k\cap P_+=\Pr^k\cap P$ by the description of $\Pr^k\cap P$ recalled in~\eqref{eq:PrkZ} below.

\item Let $\O^{\fin}_{k,\Ord}$ be  the essential image of the restricted functor \[\Omega\big|_{\O_{k,\Ord}}\colon\mathcal O_{k,\Ord}\longrightarrow \Mod(\A),\]
which is the full subcategory of $\O^\fin_k$ whose objects are isomorphic to  $\Om(M)$ for some $M\in \O_{k,\Ord}$. Then $\O^{\fin}_{k,\Ord}$ is also a semisimple category with simple objects $\{L(\la): \la\in \Pr^k\cap P\} $.

\end{enumerate}
The relations between these categories are given by the following diagram:
\begin{equation}
  \begin{tikzcd}[row sep=large, column sep=large]
\O_{k,\Ord}\arrow[r,hook]\arrow[d,"\Om", dashed, shift left =2pt]&\arrow[d, dashed, shift left =2pt,"\Om"]\O_{k}(L_k(\g))\arrow[r,hook]&\mathsf{Adm}(L_k(\g))\arrow[d,"\Om", shift left =2pt]\\
\O^{\fin}_{k,\Ord}\arrow[r,hook]\arrow[u,dashed,"\Phi^\L",shift left = 2pt]& \O^{\fin}_{k} \arrow[r,hook]\arrow[u,dashed,"\Phi^\L",shift left = 2pt]& \Mod(\A),\arrow[u,"\Phi^\L",shift left = 2pt]
\end{tikzcd}
\end{equation}

where $\Phi^\L\dashv \Om$ is the adjoint pair introduced after Definition~\ref{def:Verma}.
\end{definition}

Note that in general the generalized Verma module functor $\Phi^\L(-)$ is \textbf{not} the irreducible quotient module functor $L(-):\Mod(\A)\ra \Adm(V)$ in~\cite[Theorem 6.2]{DLM}.

\begin{theorem}\label{thm:equiofcats}
     Retaining the notation from Definition~\ref{def:cats}, the adjoint functors $\Phi^\L\dashv \Om$ between $\Adm(L_k(\g))$ and $\Mod(\A)$ restrict to well-defined adjoint functors
\begin{equation}\label{eq:equicat}
   (\Phi^\L\dashv \Om):\O_k^\fin\leftrightarrows \O_k(L_k(\g)).
\end{equation}
Moreover, this adjoint pair~\eqref{eq:equicat} is an adjoint equivalence between the categories $\O_k^\fin$ and $\O_k(L_k(\g))$.

In particular, any irreducible $L_k(\g)$-module $\hat{L}_k(\la)$, with $\la\in \Pr^k$, in the category $\O_k(L_k(\g))$ is a generalized Verma module associated with its bottom degree $L(\la)$:
\begin{equation}\label{eq:LklambdaVerma}
\hat{L}_k(\la)\cong \Phi^\L(L(\la)).
\end{equation}
\end{theorem}
\begin{proof}
First, we show that the essential image $\Om(\O_k(L_k(\g)))$ is contained in $\O_k^\fin$. Indeed, since $\O_k(L_k(\g))$ is semisimple, we only need to show $\Om(W)\in \O_k^\fin$ for any irreducible module $W=\hat{L}_k(\la)$, with $\la\in \Pr^k$. It follows from~\cite[Proposition 5.4]{DLM} that $\Om(W)=W(0)=L(\la)$, which is in $\O_k^\fin$; see Definition~\ref{def:cats}(2).

Conversely, given an irreducible module $S=L(\la)\in \O_k^\fin$, by Definition~\ref{def:Verma}, $\Phi^\L(S)$ is an admissible $L_k(\g)$-module. We claim that it is in the category $\O_k$.

Recall that $L_k(\g)=V^k(\g)/N_k$, where $N_k$ is the maximal proper ideal in the universal affine VOA $V^k(\g)$. 
The enveloping algebra $\scrU(V^k(\g))$ is isomorphic to the completed universal enveloping algebra $\tilde{U}_k(\hat{\g})$~\cite[Section 3.1]{FZ92}, \cite[Section 4.4.3]{frenkel_ben_zvi_book}. Then
 $\scrU(L_k(\g))\cong \tilde{U}_k(\hat{\g})/\tilde{I}_k$, where $\tilde{I}_k$ is the closed two-sided ideal generated by the images of the modes $a_{[n]}$, $a\in N_k$, $n\in\Z$, under $\scrU(V^k(\g))\xrightarrow{\simeq} \tilde{U}_k(\hat{\g})$.
 Then, by Definition~\ref{def:Verma},
\begin{align*}
 \Phi^\L(S)&=(\scrU/\rN^1_\L\scrU)\o_{\scrU_0} S= \bigoplus_{p=0}^\infty \Phi^\L(S)(p),\\
\Phi^\L(S)(p)&=\spn\{a^1(-n_1)\cdots a^r(-n_r)\o w:a^i\in \g,\ \\ & n_1\geq \dots\geq n_r\geq 1,\ \sum_{i=1}^r n_i=p, \ w\in L(\la)\}.
\end{align*}
Since $L(0)$ acts on the bottom degree $\Phi^\L(S)(0)=L(\la)$ by the scalar $\frac{\<\la,\la+2\rho\>}{2(k+h^\vee)}$, and $\h$ acts semisimply there, it is clear that $L(0)$ acts locally finitely on $\Phi^\L(S)$, and $\hat{\h}=\h\op \C K$ acts semisimply on $\Phi^\L(S)$. Moreover, $\hat{\g}_+$ acts locally nilpotently on $\Phi^\L(S)$, since for any $X(n)\in \hat{\g}_+$ with $n\geq 1$, $X(n).\Phi^\L(S)(p)\subset \Phi^\L(S)(p-n)$. It remains to show that $\hat{\n}_+$ acts locally nilpotently on $\Phi^\L(S)$.

The $\hat{\g}$-module $\Phi^\L(S)$ is generated by $S=\Phi^\L(S)(0)$, on which $\hat{\g}_+$ acts by zero (for degree reasons) and $\g$ acts through the $\A$-module structure of $S$. By the universal property of induced modules, $\Phi^\L(S)$ is therefore a quotient of
\[
\widetilde{V}(S):=U(\hat{\g})\o_{U(\g[t]\op\C K)} S\cong U(\hat{\g}_{<0})\o_\C S,\qquad \hat{\g}_{<0}:=\g\o t^{-1}\C[t^{-1}],
\]
where $\g[t]$ acts on $S$ through the evaluation at zero, and $K$ acts by $k$. The module $\widetilde{V}(S)$ is $\N$-graded by $\deg\big(a^1(-n_1)\cdots a^r(-n_r)\o w\big)=\sum_i n_i$, the quotient map is graded, and local nilpotency of an operator passes to quotients. It therefore suffices to prove that $\hat{\n}_+$ acts locally nilpotently on $\widetilde{V}(S)$. Since $\widetilde{V}(S)=U(\hat{\g}_{<0})\o_\C S$ is a tensor product of vector spaces, operators may be defined factorwise.

Let $X\in\n_+$. Since $[X(0),a(-n)]=(\ad X(a))(-n)$ for $a\in\g$ and $n\geq 1$, the derivation $\ad X(0)$ of $U(\hat{\g})$ preserves the subalgebra $U(\hat{\g}_{<0})$. Hence
\[
A(u\o w):=[X(0),u]\o w,
\qquad
B(u\o w):=u\o Xw,\qquad u\in U(\hat{\g}_{<0}),\ w\in S,
\]
are well-defined linear operators on $\widetilde{V}(S)$, and for $u=a^1(-n_1)\cdots a^r(-n_r)$ and $w\in L(\la)$ we have
\[
X(0).(u\o w)
=
[X(0),u]\o w + u\o (X.w)=(A+B)(u\o w).
\]
Thus
$
X(0)=A+B$ as operators on $\widetilde{V}(S)$, and $ A B= B A$, since $AB(u\o w)=[X(0),u]\o Xw=BA(u\o w)$.
It follows that
\[
X(0)^N.(u\o w)=(A+B)^N(u\o w)=\sum_{j=0}^N \binom{N}{j} A^j B^{N-j}(u\o w).
\]
Since $X\in\n_+$ acts locally nilpotently on $L(\la)$, it is clear that $B^m(u\o w)=0$ for $m$ sufficiently large.

Next, we show that $A^M(u\o w)=0$ for large enough $M$. Indeed, since $\g$ is finite-dimensional semisimple and $X\in\n_+$, the operator $\ad X$ is nilpotent on $\g$. Therefore for each $a\in\g$ there exists $N(a)$ such that
\[
(\ad X)^{N(a)}(a)=0.
\]
Moreover, for $m>0$ we have
$
A.(a(-m)\o w)
=
[X(0),a(-m)]\o w=\ad X(a)(-m)\o w
$. Hence 
\[
A^{N(a)}(a(-m)\o w)=\left((\ad X)^{N(a)}(a)\right)(-m)\o w=0.
\]
Since $A$ acts as $\ad X$, which is a derivation on the product $u=a^1(-n_1)\cdots a^r(-n_r)$, repeated applications of $A$ distribute among the $r$ factors. Therefore, if $M>N(a^1)+\cdots+N(a^r)$, then $(\ad X)^M(u)=0$. This shows $A^M(u\o w)=0$ for sufficiently large $M$.

Now let $N_1$ and $N_2$ be such that $A^{N_1}(u\o w)=0$ and $B^{N_2}(u\o w)=0$. For $N>N_1+N_2$, we have
$X(0)^N.(u\o w)=0$. Hence every $X(0)$, with $X\in \n_+$, acts locally nilpotently on $\widetilde{V}(S)$.

Next, let $\nu=X(0)+Y\in \hat{\n}_+$ be an arbitrary element, with $X\in\n_+$ and $Y\in\hat{\g}_+$ a finite sum of homogeneous terms $Z(m)$, $Z\in\g$, $m\geq 1$. Fix $v\in\bigoplus_{e\leq p}\widetilde{V}(S)(e)$. Expand $\nu^N$ as a sum of words in the letters $X(0)$ and the homogeneous components of $Y$. Any word containing more than $p$ letters from $\hat{\g}_+$ annihilates $v$, since each such letter strictly lowers the degree and $\widetilde{V}(S)$ is graded by $\N$. In a word with $t\leq p$ letters from $\hat{\g}_+$, push every letter $X(0)$ to the right using $[X(0),Z(m)]=(\ad X(Z))(m)$: each move either transposes the pair or replaces it by the single letter $(\ad X(Z))(m)\in\hat{\g}_+$. Since $\ad X$ is nilpotent on $\g$, say $(\ad X)^{h}=0$, each letter of $\hat{\g}_+$ can absorb at most $h-1$ letters $X(0)$ before vanishing, so at least $N-t-t(h-1)\geq N-ph$ letters $X(0)$ reach the right end. Every surviving word is therefore of the form (a word in $\hat{\g}_+$ of length $\leq t$)$\,\cdot\, X(0)^b v$ with $b\geq N-ph$, and choosing $N>ph+N_0$, where $X(0)^{N_0}v=0$ by the preceding argument, gives $\nu^N v=0$. Hence $\hat{\n}_+$ acts locally nilpotently on $\widetilde{V}(S)$, and therefore on its quotient $\Phi^\L(S)$.

Lastly, $\Phi^\L(S)$ is a finitely generated $\hat{\g}$-module, being generated by the cyclic $\g$-module $S=L(\la)$ (indeed by a single highest-weight vector of $L(\la)$). Together with the properties established above, this shows $\Phi^\L(S)\in \O_k(L_k(\g))$, and hence that $\Phi^\L\dashv \Om$ restricts to a well-defined adjoint pair $\O_k^\fin\leftrightarrows \O_k(L_k(\g))$.

Finally, we show that $\Phi^\L\dashv \Om$ is an adjoint equivalence between $\O_k^\fin$ and $\O_k(L_k(\g))$.
Given an irreducible module $S=L(\la)\in \O_k^\fin$,
we first show that $\Phi^\L(S)$ is an irreducible $L_k(\g)$-module. Indeed, suppose $0\neq \mathcal{K}\subsetneq \Phi^\L(S)$ is a proper submodule.
Since $\Phi^\L(S)=\bigoplus_{p=0}^\infty \Phi^\L(S)(p)$ is a direct sum of $L(0)$-eigenspaces with distinct eigenvalues, we have $\mathcal{K}=\bigoplus_{p=0}^\infty \mathcal{K}(p)$, where $\mathcal{K}(p)=\mathcal{K}\cap \Phi^\L(S)(p)$. Now $\mathcal{K}(0)$ is an $\A$-submodule of the simple $\A$-module $S$, so $\mathcal{K}(0)=0$ or $\mathcal{K}(0)=S$; the latter is impossible, since $S$ generates $\Phi^\L(S)$ and would force $\mathcal{K}=\Phi^\L(S)$, contradicting properness. Hence $\mathcal{K}(0)=0$. Since the category $\O_k(L_k(\g))$ is semisimple, we have $\Phi^\L(S)\cong \mathcal{K}\op (\Phi^\L(S)/\mathcal{K})$ as an $L_k(\g)$-module, and $\mathcal{K}(0)=0$ gives $S\hookrightarrow \Phi^\L(S)/\mathcal{K}$. Realizing $\Phi^\L(S)/\mathcal{K}$ as a direct summand of $\Phi^\L(S)$ containing $S$, and using again that $S$ generates $\Phi^\L(S)$, we conclude that this summand is all of $\Phi^\L(S)$, so $\mathcal{K}=0$, contrary to our assumption. Hence $\Phi^\L(S)$ is an irreducible $L_k(\g)$-module in $\O_k$. Applying Theorem~\ref{thm:Aramain} again, we see that $\Phi^\L(S)$ is a direct sum of $\hat{L}_k(\mu)$, with $\mu\in \Pr^k$. But it is irreducible with bottom degree $L(\la)$. Hence $\Phi^\L(S)\cong\hat{L}_k(\la)$.


Therefore, given a simple object $W=\hat{L}_k(\la)\in \O_k(L_k(\g))$, we have $\Phi^\L(\Om(W))=\Phi^\L(L(\la))\cong \hat{L}_k(\la)=W$. Conversely, given a simple object $S=L(\la)\in \O_k^\fin$, we have $\Om(\Phi^\L(S))=\Om(\hat{L}_k(\la))=\hat{L}_k(\la)(0)=S$. Since $\O_k(L_k(\g))$ and $\O_k^\fin$ are both semisimple, we have $\Phi^\L\circ \Om\cong \Id$ and $\Om\circ \Phi^\L\cong \Id$. 
\end{proof}

\begin{corollary}\label{coro:equiofcats}
The adjoint pair~\eqref{eq:equicat} further restricts to an equivalence between the subcategories of ordinary modules:
\[(\Phi^\L\dashv \Om):\O_{k,\Ord}^\fin\leftrightarrows \O_{k,\Ord}.\]
\end{corollary}
\begin{proof}
Both categories are semisimple, so it suffices to check that the equivalence~\eqref{eq:equicat} matches their simple objects. By Definition~\ref{def:cats}, the simple objects of $\O_{k,\Ord}$ are the $\hat{L}_k(\la)$ with $\la\in \Pr^k\cap P$, and those of $\O^\fin_{k,\Ord}$ are the $L(\la)$ with $\la\in \Pr^k\cap P$; by~\eqref{eq:LklambdaVerma} the functors $\Phi^\L$ and $\Om$ interchange them.
\end{proof}

Following the notation of~\cite{A16,AvEM22}, we write $\Pr^k_\Z=\Pr^k\cap P$ for the set of ordinary admissible weights. By~\cite[eq.~(10)]{A16}, under the correspondence $\hat{\la}=\la+k\Lambda_0\in \hat{\h}^\ast$ for $\la\in \h^\ast$,
\begin{equation}\label{eq:PrkZ}
\Pr^k_\Z=\begin{cases}
\left\{\la\in \h^\ast: \<\la,\al^\vee_i\>\in \Z_{\geq 0}\ \forall i,\ \<\la,\theta^\vee\>\leq p-h^\vee \right\}& \text{if }(r^\vee,q)=1,\\
\left\{\la\in\h^\ast: \<\la,\al^\vee_i\>\in \Z_{\geq 0}\ \forall i,\ \<\la,\theta_s^\vee\>\leq p-h \right\}& \text{if }(r^\vee,q)=r^\vee.
\end{cases}
\end{equation}
We record the properties of this set that will be used throughout. Let $w_0\in W$ be the longest Weyl group element, so that $w_0(\Delta_+)=\Delta_-$.

\begin{lemma}\label{lm:w0closed}
\begin{enumerate}
\item $\Pr^k_\Z\subset P_+$. Consequently $\Pr^k\cap P=\Pr^k\cap P_+$, and the quotient map restricts to a bijection
\begin{equation}\label{eq:PrkZclasses}
\Pr^k_\Z\xrightarrow{\ \simeq\ }\left\{[\la]\in[\Pr^k]:\la\in \Pr^k_\Z\right\},
\end{equation}
so that the two index sets may be used interchangeably.
\item $\Pr^k_\Z$ is stable under $\la\mapsto -w_0(\la)$, and for every $\la\in \Pr^k_\Z$ the dual module $L(\la)^\ast\cong L(-w_0(\la))$ is again a simple object of $\O^\fin_{k,\Ord}$.
\end{enumerate}
\end{lemma}
\begin{proof}
(1) The first condition in either case of~\eqref{eq:PrkZ} says exactly that $\la\in P_+$. Distinct elements of $P_+$ lie in distinct $W\circ$-orbits, since $\la+\rho$ is then strictly dominant, which gives~\eqref{eq:PrkZclasses}.

(2) The involution $-w_0$ permutes $\Delta_+$, hence permutes the simple roots, say $-w_0(\al_i)=\al_{\si(i)}$ for a permutation $\si$ of the index set, and fixes the highest root and the highest short root; the same holds for the dual action on the coroots. Since $w_0^{-1}=w_0$, for $\la\in \Pr^k_\Z$ in the case $(r^\vee,q)=1$,
\[
\<-w_0(\la),\al^\vee_i\>=\<\la,-w_0(\al^\vee_i)\>=\<\la,\al_{\si(i)}^\vee\>\in \Z_{\geq 0},\qquad
\<-w_0(\la),\theta^\vee\>=\<\la,\theta^\vee\>\leq p-h^\vee,
\]
so $-w_0(\la)\in \Pr^k_\Z$; the coprincipal case is identical, with $\theta_s$ in place of $\theta$. Finally $L(\la)^\ast\cong L(-w_0(\la))$ for $\la\in P_+$, and $\dim L(\la)<\infty$, so this is a simple object of $\O^\fin_{k,\Ord}$ by Definition~\ref{def:cats}.
\end{proof}

\subsection{$(n+1)$-point coinvariants of $L_k(\g)$ on the projective line}


\subsubsection{Coinvariants on $(n+1)$-pointed $\P^1$}
		If $C\bs P_\bullet$ is affine, then by~\cite[Section 5.1]{DG}, \cite[Lemma 2.6.1]{DGT2}, and~\cite[Section 5.1]{NT}, there is an identification
	$
	\mathcal{L}_{C\bs P_\bullet}(V)\cong H^0(C\bs P_\bullet, \mathscr{V}_C\o \om_C)/\im\nabla,
	$
	where
	\[
	H^0(C\bs P_\bullet, \mathscr{V}_C\o \om_C)\cong \bigoplus_{m=0}^\infty V_m\o H^0(C\bs P_\bullet, \om_C^{1-m}),
	\]
    as vector spaces.
Now we consider the case when $C$ is a smooth projective curve of genus zero (i.e.\ $\P^1$) with $n+1$ marked points $P_\bullet=(P_0,P_1,\ds,P_n)$. In view of Remark~\ref{remark:coordinatefree} we may assume, without loss of generality, that $P_0=\infty$; we fix the affine coordinate $z$ on $\P^1\bs\{P_0\}$ and identify each of the remaining marked points $P_1,\ds,P_n$ with its coordinate in $\C$. Then
    \[
H^0(C\bs P_\bullet, \om_C^{1-m})\cong \C[z, (z-P_1)^{-1},\ds, (z-P_n)^{-1}] (dz)^{1-m}.
    \]
Moreover, $\im\nabla$ can be identified with the image of \[\nabla=L(-1)\o 1+1\o \frac{d}{dz}:\ \ \bigoplus_{m=0}^\infty V_m\o H^0(C\bs P_\bullet,\om_C^{-m})\ra \bigoplus_{m=0}^\infty V_m\o H^0(C\bs P_\bullet, \om_C^{1-m}). \]
Hence we may represent sections in the chiral Lie algebra $ \mathcal{L}_{\P^1\bs \{P_0,P_1,\ds ,P_n\}}(V)$ by
\begin{equation}\label{eq:si}
    \si=a\o f(z) (dz)^{1-m}=	a\o \frac{z^r}{(z-P_1)^{m_1}\ds (z-P_n)^{m_n}} (dz)^{1-m},
\end{equation}
where $ a\in V_m$ and $r,m_1,\ds,m_n\in \N$,
subject to the relations
\[L(-1)a\o f(z) (dz)^{-m}=-a\o \frac{d}{dz}(f(z)) (dz)^{1-m},\]
for all $a\in V_m$ and $f(z)\in \C[z, (z-P_1)^{-1},\ds, (z-P_n)^{-1}]$.

Let $M^0,M^1,\ds, M^n$ be admissible $V$-modules attached to $P_0,P_1,\ds,P_n$. In this subsection we write $M^\bullet:=M^1\o\cdots\o M^n$ for the tensor product of the modules attached to the finite points and $M^\bullet(0):=M^1(0)\o\cdots\o M^n(0)$ for the tensor product of their degree-zero subspaces, so that the module attached to all $n+1$ points is $M^0\o M^\bullet$. The action of the chiral Lie algebra element $\si=a \o  f(z) (dz)^{1-m}$~\eqref{eq:si} in Definition~\ref{df:chiralLie} has an explicit description with the choice of local coordinates $t_\bullet=(1/z,z-P_1,\ds ,z-P_n)$; see~\cite[Definition 5.1.2]{NT}.

First, for the point $P_i$, with $1\leq i\leq n$, the attached $V$-module $M^i$ is a module over the chiral Lie algebra via
\[
\begin{split}
   & \rho_{P_i}: \mathcal{L}_{\P^1\bs \{P_0,P_1,\ds ,P_n\}}(V)\xrightarrow{\iota_{z=P_i}} \mathcal{L}(V)^\L \stackrel{\rho}{\longrightarrow} \gl(M^i),\\
 \si &\mapsto a\o \iota_{z=P_i}(f(z))\mapsto \Res_{z=P_i}Y_{M^i}(a,z-P_i)\iota_{z=P_i}(f(z))=\si_{P_i},
\end{split}
\]
where $\iota_{z=P_i}(f(z))$ is a Laurent series expansion of $f(z)$ near $P_i$, and $\rho$ is the action of $\mathcal{L}(V)^\L$ on $M^i$.
More precisely, for any $u\in M^i$, we have
\begin{equation}\label{eq:Piaction}
    \begin{aligned}
\si_{P_i}.u&=  \sum_{j,j_1,\ds ,j_n\geq 0} \binom{r}{j}\binom{-m_1}{j_1}\cdots \widehat{\binom{-m_i}{j_i}}\cdots \binom{-m_n}{j_n} \\
 &\cdot P_i^{r-j} (P_i-P_1)^{-m_1-j_1}\cdots \widehat{(P_i-P_{i})}\cdots(P_i-P_n)^{-m_n-j_n} a_{(j+j_1+\ds\hat{j_i}\ds +j_n-m_i)}u.
    \end{aligned}
\end{equation}
Second, for the point $P_0=\infty$, the attached $V$-module $M^0$ is a module over the chiral Lie algebra via
\[
\rho_{P_0}: \mathcal{L}_{\P^1\bs \{P_0,P_1,\ds ,P_n\}}(V)\ra \gl(M^0),\quad \si\mapsto \si_{P_0}=\Res_{z=P_0}Y_{M^0}(\vartheta(a),z^{-1})\iota_{z=P_0}(f(z)),
\]
where $\vartheta(a)=-e^{zL(1)}(-z^{-2})^{L(0)}(a)$.
More precisely, for any $v\in M^0$, we have
\begin{equation}\label{eq:P0action}
    \begin{aligned}
\si_{P_0}.v&=  -\sum_{j_1,\ds ,j_n\geq 0}\binom{-m_1}{j_1}\cdots \binom{-m_n}{j_n} \\
 &\cdot  (-P_1)^{j_1}\cdots (-P_n)^{j_n} a'_{(r-m_1-j_1-\ds -m_n-j_n)}v,
    \end{aligned}
\end{equation}
where  $a'_{(m)}=\sum_{i\geq 0} \frac{(-1)^{\wt a}}{i!} (L(1)^ia)_{(2\wt a-m-i-2)}$~\cite[Section 5.2]{FHL93}.

Finally, it follows from~\eqref{chiralLieact} that $M^0\o M^1\o\cdots\o M^n$ is a tensor product module over the chiral Lie algebra
\[
\rho: \mathcal{L}_{\P^1\bs \{P_0,P_1,\ds ,P_n\}}(V)\ra \gl(M^0\o M^1\o\cdots\o M^n),
\]
where the action of $\si\in \mathcal{L}_{\P^1\bs \{P_0,P_1,\ds ,P_n\}}(V)$ is given by~\eqref{eq:Piaction} and~\eqref{eq:P0action}:
\[
\si.(v\o u^1\o\cdots\o u^n)=(\si_{P_0}.v)\o u^1\o\cdots\o u^n+\sum_{i=1}^n v\o u^1\o\cdots\o (\si_{P_i}.u^i)\o\cdots\o u^n.
\]
For a smooth genus-zero curve $C$, the restriction of sections gives an embedding
\[H^0(C,\om_C^{1-m})\ssq H^0(C\bs P_\bullet,\om_C^{1-m}),\]
where the sections in $H^0(C,\om_C^{1-m})$ are regular at all marked points $P_\bullet$.
In particular, for $C\bs P_\bullet=\P^1\bs\{P_0,P_1,\ds ,P_n\}$, we have $H^0(C,\om_C^{1-m})=\spn\{ z^l (dz)^{1-m}: 0\leq l\leq 2m-2\}$. If $m=1$, then $H^0(C,\om_C^{1-m})=H^0(C,\O_C)\cong \C$.

Introduce the following subspaces of $\mathcal{L}_{C\bs P_\bullet}(V)$:
\begin{equation}\label{eq:gvreg}
\begin{aligned}
\mathcal{L}_{C\bs P_\bullet}(V)_{\reg}&:=\left(\bigoplus_{m=0}^\infty V_m\o H^0(C,\om_C^{1-m})+\im\nabla\right)/\im\nabla,  \\
\g(V)_\reg&:=\left(V_1\o H^0(C,\O_C)+\im\nabla\right)/\im\nabla.
\end{aligned}
\end{equation}
Note that $\g(V)_\reg=0$ when $V_1=0$.
\begin{lemma}\label{lm:injgv}
Let $V$ be a VOA of CFT-type. Then $L(-1)$ acts injectively on $V_+:=\bigoplus_{m\geq 1}V_m$, and the subspace $\g(V)_\reg$ can be identified with $V_1\o H^0(C,\O_C)\cong V_1$.
\end{lemma}
\begin{proof}
For the first claim, let $a\in V$ with $L(-1)a=0$. Then $\frac{d}{dz}Y(a,z)=Y(L(-1)a,z)=0$, so $Y(a,z)=a_{(-1)}$ is constant in $z$. By skew-symmetry~\cite[Section 3.1]{LL04}, for every $b\in V$,
\[
Y(b,-z)a=e^{-zL(-1)}Y(a,z)b=e^{-zL(-1)}a_{(-1)}b\in V\lb z\rb,
\]
so $b_{(n)}a=0$ for all $b\in V$ and $n\geq 0$. Taking $b=\om$ and $n=1$ gives $L(0)a=0$, hence $a\in V_0$. Thus $\ker L(-1)=V_0=\C\vac$, and $L(-1)$ is injective on $V_+$.

For the second claim, it suffices to show $(V_1\o H^0(C,\O_C))\cap \im\nabla=0$. For any $m\geq 0$ and $b\o g(z) (dz)^{-m}\in V_m\o H^0(C\bs P_\bullet, \om_C^{-m})$, we have
\begin{equation}\label{eq:nablaformula}
\nabla(b\o g(z) (dz)^{-m})=L(-1)b\o g(z) (dz)^{-m}+b\o g'(z) (dz)^{1-m},
\end{equation}
so $\nabla$ maps $V_m\o H^0(C\bs P_\bullet,\om_C^{-m})$ into $V_{m+1}\o H^0(C\bs P_\bullet,\om_C^{-m})\op V_m\o H^0(C\bs P_\bullet,\om_C^{1-m})$.
Suppose $a\o 1=\nabla(\xi)$ with $a\in V_1$ and $\xi=\sum_{m=0}^{M}\xi_m$, where $\xi_m\in V_m\o H^0(C\bs P_\bullet,\om_C^{-m})$ and $\xi_M\neq 0$. By~\eqref{eq:nablaformula}, the component of $\nabla(\xi)$ in $V_{M+1}\o H^0(C\bs P_\bullet,\om_C^{-M})$ is $(L(-1)\o 1)(\xi_M)$, the summands $\xi_m$ with $m<M$ contributing nothing to it. If $M\geq 1$, then $M+1\geq 2$ and the corresponding component of $a\o 1\in V_1\o H^0(C,\O_C)$ is zero, so $(L(-1)\o 1)(\xi_M)=0$; since $L(-1)$ is injective on $V_M$, this forces $\xi_M=0$, a contradiction. Hence $M=0$ and $\xi=\vac\o g(z)$ for some $g\in H^0(C\bs P_\bullet,\O_C)$, so that $\nabla(\xi)=\vac\o g'(z)\,dz\in V_0\o H^0(C\bs P_\bullet,\om_C)$ because $L(-1)\vac=0$. Comparing with $a\o 1\in V_1\o H^0(C,\O_C)$ gives $a\o 1=0$. This shows $(V_1\o H^0(C,\O_C))\cap \im\nabla=0$, and it follows from~\eqref{eq:gvreg} that $\g(V)_\reg\cong V_1\o H^0(C,\O_C)\cong V_1$.
\end{proof}

\begin{lemma}\label{lm:regdegree}
Let $C\bs P_\bullet=\P^1\bs\{P_0,P_1,\ds ,P_n\}$ and let $t_\bullet=(1/z,z-P_1,\ds ,z-P_n)$. Let $M^0,M^1,\ds, M^n$ be admissible $V$-modules attached to $P_0,P_1,\ds,P_n$, with the chiral Lie algebra action described by~\eqref{eq:Piaction} and~\eqref{eq:P0action}. Then for a given section
\[
\si=a\o z^l (dz)^{1-m}\in \mathcal{L}_{\P^1\bs \{P_0,P_1,\ds ,P_n\}}(V)_\reg,\quad 0\leq l\leq 2m-2,
\]
and homogeneous subspaces $M^0(p_0),M^1(p_1),\ds, M^n(p_n)$, with $p_0,p_1,\ds,p_n\geq 0$, of the admissible $V$-modules, we have
\[
\si_{P_0}.M^0(p_0)\ssq M^0(p_0-m+l+1),\quad  \si_{P_i}.M^i(p_i)\ssq \sum_{j\geq 0}M^i(p_i+m-j-1).
\]

\end{lemma}
\begin{proof}
For $\si=a\o z^l (dz)^{1-m}$, it follows from~\eqref{eq:P0action} that
\[
\si_{P_0}.M^0(p_0)=-a'_{(l)}M^0(p_0)=-\sum_{i\geq 0} \frac{(-1)^m}{i!} (L(1)^ia)_{(2m-l-i-2)}M^0(p_0),
\]
with $\deg((L(1)^ia)_{(2m-l-i-2)})=-m+l+1$. Hence $\si_{P_0}.M^0(p_0)\ssq M^0(p_0-m+l+1)$. On the other hand, given any $u^i\in M^i(p_i)$, by~\eqref{eq:Piaction} we have
\[
\si_{P_i}.u^i=\sum_{j\geq 0} \binom{l}{j} P_i^{l-j} a_{(j)}u^i\in \sum_{j\geq 0} M^i(p_i+m-j-1).
\]
\end{proof}

\begin{corollary}\label{coro:surjmap}
Let
$\tau=a\o 1\in \g(V)_\reg,$ with $a\in V_1$.
Then the action of $\tau$ preserves $M^0(0)\o M^1(0)\o\cdots\o M^n(0)$. Moreover, the canonical embedding map $M^0(0)\o M^\bullet(0)\ra M^0\o M^\bullet$, $v\o u^\bullet\mapsto v\o u^\bullet$, induces a well-defined linear map to the space of coinvariants:
\begin{equation}\label{eq:0coinvtocoinv}
\frac{M^0(0)\o M^\bullet(0)}{\g(V)_\reg.(M^0(0)\o M^\bullet(0))}\ra [M^0\o M^\bullet]_{(C,P_\bullet,t_\bullet)}.
\end{equation}
\end{corollary}
\begin{proof}
    Let $\tau=a\o 1\in \g(V)_\reg$, with $m=1$ and $l=0$. Then we have
\[
\tau_{P_0}. M^0(0)\ssq M^0(0),\quad \tau_{P_i}.M^i(0)\ssq \sum_{j\geq 0} M^i(0-j)=M^i(0).
\]
Thus, $\g(V)_\reg. (M^0(0)\o M^\bullet(0))\ssq M^0(0)\o M^\bullet(0)$. Moreover, since $\g(V)_\reg$ by definition is contained in $\mathcal{L}_{\P^1\bs \{P_0,P_1,\ds ,P_n\}}(V)$, we have a well-defined linear map~\eqref{eq:0coinvtocoinv}.
\end{proof}
\subsubsection{Affine VOA coinvariants}
From now on we assume that $V=L_k(\g)$ is the simple affine VOA at an admissible level $k$ with $k+h^\vee=p/q$.
We exclude the trivial case $k=0$, where $L_0(\g)=\C$ and all statements below hold trivially. For $k\neq 0$, $L_k(\g)$ is strongly generated by its degree-one subspace $L_k(\g)_1$, which is the Lie algebra $\g$; in particular $L_k(\g)$ is $C_1$-cofinite. In this paper we do not use the shifted Virasoro element for the VOA $L_k(\g)$ as in~\cite{DLM96}.
\begin{definition}\label{def:bottomgen}
An admissible $V$-module $M=\bigoplus_{d\in \N}M(d)$ is \textbf{generated in degree zero} if
\[
M=\spn_\C\left\{a^1_{(m_1)}\cdots a^r_{(m_r)}u\ :\ r\geq 0,\ a^s\in V,\ m_s\in \Z,\ u\in M(0)\right\}.
\]
\end{definition}

\begin{lemma}\label{lm:bottomgen}
Let $V$ be a VOA strongly generated by $V_1$ and let $M$ be an admissible
$V$-module generated in degree zero. Then for every $d\geq 1$,
\begin{equation}\label{eq:bottomgen}
M(d)=\spn_\C\left\{a^1_{(-j_1)}\cdots a^r_{(-j_r)}u\ :\ a^s\in V_1,\ j_s\geq 1,\ \sum\nolimits_{s=1}^r j_s=d,\ u\in M(0)\right\}.
\end{equation}
\end{lemma}

\begin{proof}
Since $\wt a^s=1$ we have $\deg\big(a^s_{(m)}\big)=-m$, so the right-hand side of
\eqref{eq:bottomgen} is contained in $M(d)$ and it suffices to prove the reverse
inclusion. By hypothesis $M$ is spanned by elements $b^1_{(m_1)}\cdots b^\ell_{(m_\ell)}u$
with $b^t\in V$ and $u\in M(0)$. As $V$ is strongly generated by $V_1$, each $b^t$ is a
linear combination of vectors $c^1_{(-k_1)}\cdots c^{s}_{(-k_{s})}\vac$ with $c^i\in V_1$
and $k_i\geq 1$; substituting these and expanding by repeated use of the associator
formula---the case $m=0$ of~\eqref{eq:componentJacobi}, which expresses $(a_{(l)}b)_{(n)}$ through products of modes of $a$ and $b$---rewrites each such
element as a linear combination of elements $a^1_{(m_1)}\cdots a^r_{(m_r)}u$ with all
$a^s\in V_1$. This is the argument of~\cite[Lemma 4.0.1]{DG}, whose first step is
precisely the hypothesis of Definition~\ref{def:bottomgen}.

It remains to arrange $j_s:=-m_s\geq 1$. Since $M(e)=0$ for $e<0$, a mode of negative
degree annihilates $M(0)$, and a mode of degree zero preserves $M(0)$; so the rightmost
factor may be assumed of positive degree, after replacing $u$ by another element of
$M(0)$. For the remaining factors we induct on $r$ and use the commutator formula, the
case $l=0$ of~\eqref{eq:componentJacobi}: for $a,b\in V_1$,
\[
\big[a_{(m)},b_{(n)}\big]=\sum_{i\geq 0}\binom{m}{i}\big(a_{(i)}b\big)_{(m+n-i)}
=\big(a_{(0)}b\big)_{(m+n)}+m\,\big(a_{(1)}b\big)_{(m+n-1)},
\]
because $a_{(i)}b\in V_{1-i}$ vanishes for $i\geq 2$. Here $a_{(0)}b\in V_1$, while
$a_{(1)}b\in V_0=\C\vac$ contributes only through $\vac_{(m+n-1)}$, which is zero unless
$m+n=0$. Hence transposing two adjacent factors produces only terms with a shorter
string of $V_1$-modes, and the induction closes.
\end{proof}

\begin{lemma}\label{lm:DG1}
Let $V$ be a VOA of CFT-type strongly generated by $V_1$, let $(C,P_\bullet,t_\bullet)$
be a smooth $(n+1)$-pointed coordinatized curve of genus zero, and let
$M^0,M^1,\ds,M^n$ be admissible $V$-modules, each generated in degree zero
(Definition~\ref{def:bottomgen}). Then the canonical map
\[
M^0(0)\o M^\bullet(0)\longrightarrow [M^0\o M^\bullet]_{(C,P_\bullet,t_\bullet)},\qquad
u^\bullet\mapsto [u^\bullet],
\]
is surjective. No finiteness assumption is made on the spaces $M^i(0)$.
\end{lemma}

\begin{proof}
Write $N:=M^0\o M^\bullet$ and filter it by total degree,
\[
\mathcal F_d N:=\bigoplus_{d_0+\ds+d_n\leq d} M^0(d_0)\o\cdots\o M^n(d_n),\qquad d\in \N,
\]
so that $\mathcal F_0 N=M^0(0)\o M^\bullet(0)$ and $N=\bigcup_{d\geq 0}\mathcal F_d N$. It
suffices to prove
\begin{equation}\label{eq:filtrationstep}
\mathcal F_d N\ssq \mathcal F_{d-1}N+\LL_{C\bs P_\bullet}(V).N \qquad \text{for all } d\geq 1,
\end{equation}
since iterating~\eqref{eq:filtrationstep} gives
$\mathcal F_dN\ssq \mathcal F_0N+\LL_{C\bs P_\bullet}(V).N$ for every $d$, which is the assertion.

Let $d\geq 1$. By Lemma~\ref{lm:bottomgen} the space $\mathcal F_dN$ is spanned by pure
tensors $w^\bullet=w^0\o\cdots\o w^n$ with $w^k\in M^k(d_k)$ homogeneous,
$\sum_k d_k\leq d$, and such that for at least one index $i$ with $d_i\geq 1$ we may write
\[
w^i=a_{(-j)}u^i,\qquad a\in V_1,\ j\geq 1,\ u^i\in M^i(d_i-j)\ \text{homogeneous}.
\]
(If all $d_k=0$ then $w^\bullet\in\mathcal F_0N$ and there is nothing to prove.)

\emph{Construction of a section.} By Remark~\ref{remark:coordinatefree} we may take
\[
(C,P_\bullet,t_\bullet)=\big(\P^1,\ (P_0,P_1,\ds,P_n),\ (1/z,z-P_1,\ds,z-P_n)\big),
\qquad P_0=\infty.
\]
Because $a\in V_1$, a section of the form $\si=a\o f(z)(dz)^{1-1}=a\o f(z)$,
with $f$ a rational function whose poles lie in $P_\bullet$, is an element of
$\LL_{\P^1\bs P_\bullet}(V)$. Set
\[
\si:=a\o\mu,\qquad
\mu:=\begin{cases}
(z-P_i)^{-j}, & 1\leq i\leq n,\\[2pt]
z^{\,j}, & i=0,
\end{cases}
\]
so that in either case $\mu$ has a pole of order exactly $j$ at $P_i$ and no other pole.
This is the only step that uses the genus-zero hypothesis (together with strong generation in degree one):
for $a\in V_1$ the coefficient of $a$ is an ordinary rational \emph{function} rather than a
section of a power of $\om_C$, and on a rational curve a function with one prescribed pole
and no other pole exists for every pole order $j\geq 1$.

\emph{Expansions.} We record the action of $\si$ at each marked point using
\eqref{eq:Piaction} and~\eqref{eq:P0action}, in which $f(z)=z^r\prod_{k=1}^n(z-P_k)^{-m_k}$.
Note first that for $a\in V_1$ formula~\eqref{eq:P0action} simplifies: as in the proof of
Lemma~\ref{lm:regdegree},
\[
-a'_{(l)}=\sum_{s\geq 0}\tfrac{1}{s!}\big(L(1)^sa\big)_{(-l-s)}=a_{(-l)}+\delta_{l,0}\cdot(\text{scalar}),
\]
since $L(1)a\in V_0=\C\vac$, $L(1)^2a=0$ and $\vac_{(-l-1)}=\delta_{l,0}\id$. The
$L(1)$-correction therefore occurs only in the component $l=0$, where it is a scalar, and it
affects none of the degree bounds below. 

\emph{Case $1\leq i\leq n$}: here $\mu=(z-P_i)^{-j}$, i.e.\ $r=0$, $m_i=j$ and $m_k=0$ for
$k\neq i$. Since $\binom{0}{s}=\delta_{s,0}$, every binomial factor attached to an index
$\neq i$ collapses, and
\begin{align*}
\si_{P_i}&=a_{(-j)},\\
\si_{P_k}&=\sum_{s\geq 0}\binom{-j}{s}(P_k-P_i)^{-j-s}a_{(s)}\qquad (k\neq i,\ k\geq 1),\\
\si_{P_0}&=\sum_{s\geq 0}\binom{-j}{s}(-P_i)^{s}a_{(j+s)}.
\end{align*}

\emph{Case $i=0$}: here $\mu=z^{j}$, i.e.\ $r=j$ and all $m_k=0$, so
\[
\si_{P_0}=a_{(-j)},\qquad
\si_{P_k}=\sum_{s\geq 0}\binom{j}{s}P_k^{\,j-s}a_{(s)}\quad (1\leq k\leq n).
\]

In both cases $\si$ acts at the distinguished point $P_i$ by the \emph{single} mode
$a_{(-j)}$, raising degree by exactly $j$, and at every other marked point by a combination
of modes $a_{(s)}$ with $s\geq 0$ if that point is finite, resp.\ $a_{(j+s)}$ with $s\geq 0$
if that point is $P_0$. Since $\deg\big(a_{(s)}\big)=-s$ for $a\in V_1$, in all cases
\begin{equation}\label{eq:nonincreasing}
\si_{P_k}.M^k(d_k)\ssq \bigoplus_{e\leq d_k}M^k(e)\qquad (k\neq i);
\end{equation}
that is, $\si$ raises degree at no marked point other than $P_i$.

\emph{Conclusion.} Put $v^\bullet:=w^0\o\cdots\o w^{i-1}\o u^i\o w^{i+1}\o\cdots\o w^n$, an
element of $\mathcal F_{d-j}N$. Because $\LL_{C\bs P_\bullet}(V)$ acts on $N$ one tensor
factor at a time through the expansions $\si_{P_k}$, see~\eqref{chiralLieact}, and because
$\si_{P_i}.u^i=a_{(-j)}u^i=w^i$, we get
\[
\si.v^\bullet-w^\bullet
=\sum_{k\neq i}\left(w^0\o\cdots\o \si_{P_k}.w^k\o\cdots\o u^i\o\cdots\o w^n\right),
\]
and every summand on the right lies in $\mathcal F_{d-j}N\ssq \mathcal F_{d-1}N$ by
\eqref{eq:nonincreasing}, using $j\geq 1$. As
$\si.v^\bullet\in \LL_{C\bs P_\bullet}(V).N$, this proves~\eqref{eq:filtrationstep}.
(Should $\si$ vanish in
$\LL_{C\bs P_\bullet}(V)=H^0(C\bs P_\bullet,\mathscr{V}_C\o\om_C)/\im\nabla$, the same identity
gives $w^\bullet\in\mathcal F_{d-1}N$ directly.)
\end{proof}

\begin{remark}\label{rmk:bottomgen_necessary}
Some hypothesis of this kind is necessary: an admissible grading may be shifted. If
$M=\bigoplus_{d\in\N}M(d)$ is admissible and $M'$ denotes the same weak module with the
grading $M'(d):=M(d-1)$, then $M'$ is again admissible, since
$a_{(m)}M'(d)=a_{(m)}M(d-1)\ssq M(d-1+\wt a-m-1)=M'(d+\wt a-m-1)$, but $M'(0)=0$ while
$[M'^\bullet]_{(C,P_\bullet,t_\bullet)}=[M^\bullet]_{(C,P_\bullet,t_\bullet)}$, the space of
coinvariants depending only on the underlying weak modules. Being generated in degree zero
forces $M(0)\neq 0$ and rules this out.
\end{remark}

\begin{remark}\label{rmk:bottomgen_holds}
All modules to which Lemma~\ref{lm:DG1} is applied below are generated in degree zero.
\begin{enumerate}
\item For any $\A$-module $S$ the generalized Verma module $\Phi^\L(S)$ is generated in
degree zero by construction, and $\Phi^\L(S)(0)=S$; see Definition~\ref{def:Verma}. This
covers $M^0=\Phi^\L(\A_\mu)$ for $\mu\in \Pr^k\bs P$, \emph{including} the case
$\dim_\C \A_\mu=\dim_\C U(\g)/J_\mu=\infty$, which is the relevant one in
Lemma~\ref{lm:coinv0} and Theorem~\ref{thm:fusion_closedness}.
\item If $M$ is a simple admissible $V$-module then $\scrU.M(0)$ is a nonzero submodule of
$M$, hence equals $M$. This covers the simple modules $\hat{L}_k(\la)$, $\la\in \Pr^k$.
\item A finite direct sum $\bigoplus_{j}M_j$ of admissible modules generated in degree
zero is generated in degree zero for the grading with $\left(\bigoplus_j M_j\right)(d)=\bigoplus_j M_j(d)$,
i.e.\ the grading in which the bottom spaces are aligned. This covers arbitrary objects of
$\O_{k,\Ord}$: by Definition~\ref{def:cats}(3) the category $\O_{k,\Ord}$ is semisimple with
simple objects $\hat{L}_k(\la)$, $\la\in \Pr^k\cap P$, and each of its objects is a
\emph{finite} direct sum of these, as in the first paragraph of the proof of
Corollary~\ref{coro:fusion_closedness_general}.
\end{enumerate}
\end{remark}

Let $(C,P_\bullet)=(\P^1,(P_0,P_1,\ds,P_n))$, and let $M^0,M^1,\ds ,M^n\in \Adm(L_k(\g))$ be generated in degree zero (Definition~\ref{def:bottomgen}; see Remark~\ref{rmk:bottomgen_holds}). 
By Corollary~\ref{coro:surjmap} the map below is well-defined, and it is surjective by Lemma~\ref{lm:DG1}:
\[
    \frac{M^0(0)\o M^\bullet(0)}{\g(V)_\reg.(M^0(0)\o M^\bullet(0))}\twoheadrightarrow [M^0\o M^\bullet]_{(C,P_\bullet,t_\bullet)},\quad [v\o u^\bullet]\mapsto [v\o u^\bullet],
\]
 where $v\in M^0(0)$, $u^i\in M^i(0)$, for $1\leq i\leq n$.

Since $V=L_k(\g)$ is of CFT-type, it follows from Lemma~\ref{lm:injgv} that $\g(V)_\reg\cong V_1\cong \g,\ \tau=a(-1)\vac\o 1\mapsto a(-1)\vac\mapsto a$. Since we also have $L(1)a=0$ for any $a\in \g=V_1$---indeed $L(1)a(-1)\vac=[L(1),a(-1)]\vac=a(0)\vac=0$, using $[L(m),a(n)]=-n\,a(m+n)$---it follows from~\eqref{eq:Piaction} and~\eqref{eq:P0action} that the action of $\tau=a(-1)\vac\o 1\in \g(V)_\reg$ is given by
\[
\tau_{P_0}. v=a(0)v,\quad \tau_{P_i}.u^i=a(0)u^i,\quad 1\leq i\leq n.
\]
In other words, the action of $\g(V)_\reg$ on $M^0(0)$ and $M^i(0)$ also agrees with the action of the finite Lie algebra $\g$ on these modules.
Hence
\[
  \frac{M^0(0)\o M^\bullet(0)}{\g(V)_\reg.(M^0(0)\o M^\bullet(0))}\cong M^0(0)\o_{U(\g)} M^\bullet (0),
\]
where $M^0(0)$ is naturally viewed as a right $U(\g)$-module via the anti-involution $\theta_\g: U(\g)\ra U(\g)$, $a\mapsto -a$ for $a\in \g$, and $M^\bullet(0)=M^1(0)\o\cdots\o M^n(0)$ is the usual tensor product of $\g$-modules.
In other words, we have the following result:
\begin{proposition}
Let $V=L_k(\g)$ be an admissible-level affine VOA, and let $M^0,M^1,\ds,M^n\in \Adm(V)$ be generated in degree zero (Definition~\ref{def:bottomgen}). Then 
there exists a canonical surjection:
\begin{equation}\label{eq:gsurj}
M^0(0)\o_{U(\g)} M^\bullet (0)\twoheadrightarrow [M^0\o M^\bullet]_{(C,P_\bullet,t_\bullet)},
\end{equation}
where $C$ is a smooth genus-zero curve, and $M^0(0)$ is regarded as a right $U(\g)$-module via the anti-involution $\theta_\g: U(\g) \rightarrow U(\g)$, $a \mapsto -a$, for $a\in \g$.
\end{proposition}

Note that the weight lattice $P$ is invariant under the action of the finite Weyl group $W$, and that $\rho\in P$. Hence, for any $\la,\mu\in \Pr^k$ with $\la=\sigma\circ\mu$ for some $\sigma\in W$, we have $\la\in P$ if and only if $\mu\in P$; that is, membership in $P$ is a property of the class $[\mu]\in[\Pr^k]$.
Hence we may write
\begin{equation}\label{eq:decofPrk}
[\Pr^k]=\Pr^k_\Z\sqcup [\Pr^k\bs P].
\end{equation}
Recall from~\cite[Theorem 3.4]{AvE} that $\A=A(L_k(\g))=U(\g)/I_k$, where $I_k:=\bigcap_{\la\in[\Pr^k]}J_\la$; since $[\Pr^k]$ is finite and the $J_\la$ are pairwise distinct maximal two-sided ideals, the Chinese remainder theorem identifies this quotient with the product~\eqref{eq:admAV}.
\begin{lemma}\label{lm:coinv0}
Let $M^0(0)$ be the left $\A$-module $\A_\mu=U(\g)/J_\mu$ for some $\mu\in \Pr^k\bs P$, and let $M^i(0)=L(\la_i)$ for some $\la_i\in \Pr^k\cap P$, $1\leq i\leq n$. Then we have
\[
 M^0(0)\o_{U(\g)} M^\bullet (0) =0.
\]
\end{lemma}
\begin{proof}
By assumption, $M^i(0)=L(\la_i)$ is finite-dimensional for all $i$. Then, by Weyl's complete reducibility theorem, $M^\bullet(0)\cong \bigoplus_l L(\nu_l)$ as a $U(\g)$-module, where $\nu_l\in P_+$. It follows from~\cite[eq.~(3.3)]{AvE} that
\[
M^0(0)=U(\g)/J_\mu\cong \A\o_{Z} \C_{\gamma_\mu},\qquad \C_{\gamma_\mu}=\C 1_{\gamma_\mu},
\]
where $Z=Z(U(\g))$ and $\gamma_\mu:Z\ra \C$ is the central character associated with $\mu\in \Pr^k\bs P$. Given any $\nu\in P_+$, $z\in Z$, and $\al=([a]\o_Z 1_{\gamma_\mu})\o_{U(\g)} v\in M^0(0)\o_{U(\g)}L(\nu)$, we have
\begin{align*}
\gamma_\mu(\theta_\g(z))\cdot \al&=([a]\o_Z 1_{\gamma_\mu}).z\o_{U(\g)} v=([a]\o_Z 1_{\gamma_\mu})\o_{U(\g)} z.v=\gamma_{\nu}(z)\cdot \al.
\end{align*}
It therefore suffices to produce $z\in Z$ with $\gamma_\mu(\theta_\g(z))\neq \gamma_\nu(z)$. Now $\theta_\g$ is the antipode of $U(\g)$, and under the Harish-Chandra isomorphism it corresponds to $\mu+\rho\mapsto -(\mu+\rho)$; since $-(\mu+\rho)$ and $-w_0(\mu+\rho)=-w_0(\mu)+\rho$ lie in the same $W$-orbit, this gives
\[
\gamma_\mu \circ \theta_\g = \gamma_{-w_0(\mu)}.
\]
As $W$ preserves $P$ and $-P=P$, the map $-w_0$ preserves $P$, so $\mu\notin P$ forces $-w_0(\mu)\notin P$. Were $\gamma_{-w_0(\mu)}=\gamma_\nu$, we would get $-w_0(\mu)\in W\circ \nu\subset P$, using $\rho\in P$, a contradiction. Hence $\al=0$, and $ M^0(0)\o_{U(\g)} M^\bullet (0)\cong  \bigoplus_l M^0(0)\o_{U(\g)} L(\nu_l)=0$.
\end{proof}

Now we prove the main theorem of this section.

\begin{theorem}\label{thm:fusion_closedness}
Let $M^0=\Phi^\L(\A_\mu)$ for some $\mu\in \Pr^k\bs P$, and let $M^1,\ds,M^n$ be simple ordinary modules in $\O_{k,\Ord}$. Then the space of coinvariants $[M^0\o M^\bullet]_{(C,P_\bullet,t_\bullet)}$ on a smooth $(n+1)$-pointed genus-zero curve $C$ is zero.
\end{theorem}
\begin{proof}
By Remark~\ref{remark:coordinatefree} we may assume that $(C,P_\bullet,t_\bullet)=(\P^1,(P_0,P_1,\ds, P_n),(1/z,z-P_1,\ds,z-P_n))$. The simple objects of $\O_{k,\Ord}$ are the $\hat{L}_k(\la)$ with $\la\in \Pr^k\cap P$, so $M^i(0)=L(\la_i)$ with $\la_i\in\Pr^k\cap P$, and~\eqref{eq:gsurj} together with Lemma~\ref{lm:coinv0} gives
$
0=M^0(0)\o_{U(\g)} M^\bullet(0)\twoheadrightarrow [M^0\o M^\bullet]_{(C,P_\bullet,t_\bullet)}.
$
\end{proof}

\begin{corollary}\label{coro:fusion_closedness_general}
Let $M^0=\Phi^\L(\A_\mu)$ for some $\mu\in \Pr^k\bs P$, and let $M^1,\ds,M^n$ be arbitrary (not necessarily simple) modules in $\O_{k,\Ord}$. Then the space of coinvariants $[M^0\o M^\bullet]_{(C,P_\bullet,t_\bullet)}$ on a smooth $(n+1)$-pointed genus-zero curve $C$ is zero.
\end{corollary}
\begin{proof}
The category $\O_{k,\Ord}$ is semisimple with simple objects $\{\hat{L}_k(\la):\la\in \Pr^k\cap P\}$, and this set is finite since the set $\Pr^k$ of admissible weights of level $k$ is finite~\cite[Theorems 2.1 and 2.2]{kac1989classification}. Consequently, every $M^i\in \O_{k,\Ord}$ decomposes as a \emph{finite} direct sum $M^i\cong \bigoplus_{j=1}^{r_i} \hat{L}_k(\la_{i,j})$: if some isomorphism class occurred with infinite multiplicity in $M^i$, the corresponding $L(0)$-eigenspace of $M^i$ at the bottom weight of that class would be an infinite direct sum of copies of a nonzero finite-dimensional space, contradicting that $M^i$ is ordinary.

By~\eqref{chiralLieact}, the chiral Lie algebra $\mathcal{L}_{C\bs P_\bullet}(V)$ acts on $M^0\o M^\bullet$ diagonally, one tensor factor at a time, and the vertex operators of a direct sum of modules act block-diagonally with respect to the direct sum decomposition. Expanding $M^\bullet=M^1\o\cdots\o M^n$ via the finite decompositions above, the action map
\[
\mathcal{L}_{C\bs P_\bullet}(V)\o_{\C} (M^0\o M^\bullet) \longrightarrow M^0\o M^\bullet
\]
therefore splits as a finite direct sum, indexed by $(j_1,\ds,j_n)\in \{1,\ds,r_1\}\times\ds\times\{1,\ds,r_n\}$, of the action maps for the tuples $(M^0,\hat{L}_k(\la_{1,j_1}),\ds,\hat{L}_k(\la_{n,j_n}))$. Since taking cokernels commutes with finite direct sums of block-diagonal maps, we obtain
\[
[M^0\o M^\bullet]_{(C,P_\bullet,t_\bullet)}\cong \bigoplus_{j_1,\ds,j_n} [M^0\o \hat{L}_k(\la_{1,j_1})\o\cdots\o \hat{L}_k(\la_{n,j_n})]_{(C,P_\bullet,t_\bullet)}.
\]
Each summand on the right vanishes by Theorem~\ref{thm:fusion_closedness}, since each $\hat{L}_k(\la_{i,j_i})$ is a simple ordinary module. Hence $[M^0\o M^\bullet]_{(C,P_\bullet,t_\bullet)}=0$.
\end{proof}

\section{The partial strong identity property}\label{sec:strong_identity}

\subsection{Splitting of the mode transition algebras}

Throughout this subsection we assume that the Zhu algebra $\A=A(V)$ of a VOA $V$ has a decomposition as a direct product of ideals, each of which is an algebra with its own identity element:
\begin{equation}
    \A \cong \A' \times \A'',
\end{equation}
with the product on $\A$ given by $(a',a'')\cdot(b',b'')=(a'\cdot b',a''\cdot b'')$. 
\begin{example}\label{ex:Aadm}
By Theorem~\ref{thm:AvE_result}, the Zhu algebra of the admissible-level affine VOA $L_k(\g)$ decomposes as a direct product of algebras. Collecting the finite-dimensional factors in $\A'$, we obtain
\begin{equation}\label{eq:Aadm}
    \A'=\prod_{\la\in \Pr^k_\Z} L(\la)\otimes_\C L(\la)^\ast\cong \prod_{\la\in \Pr^k_\Z}\End_\C(L(\la)),\quad \text{and}\quad \A''=\prod_{\la\in [\Pr^k\bs P]} U(\g)/J_{\la},
\end{equation}
indexed via~\eqref{eq:PrkZclasses}, so $\A'$ is finite-dimensional semisimple and $\A''$ is a product of simple infinite-dimensional algebras. On each factor of $\A'$ the product is
\begin{equation}\label{eq:alambda}
(u\o f)\cdot (v\o g)=\<f,v\>\cdot u\o g,\quad u,v\in L(\la),\ f,g\in L(\la)^\ast,
\end{equation}
so the identity element is $1_{\A'}=\prod_{\la\in \Pr^k_\Z} \Id_{L(\la)}$. We write $a'=(a'_\la)_{\la\in \Pr^k_\Z}$, with $a'_\la\in L(\la)\o_\C L(\la)^\ast$, for a general element of $\A'$.
\end{example}

We analyze the strong identity property of the mode transition algebras $\fA_d$ with respect to the decomposition of $\A$. Note that the natural projections
\[
p':\A\ra\A',\ a=(a',a'')\mapsto a',\quad  p'':\A\ra \A'',\ a=(a',a'')\mapsto a'',
\]
are algebra homomorphisms. In particular, $\A'$, and similarly $\A''$, is a module over the degree-zero enveloping algebra $\scrU_0$ via the homomorphism:
\begin{equation}\label{eq:projhomo}
\scrU_0\twoheadrightarrow \A\xrightarrow{p'}\A',\quad \al\mapsto [\al]_0\mapsto p'([\al]_0),
\end{equation}
for all $\al\in \scrU_0$.  Let
\begin{align*}
\fA'&=\Phi(\A')=(\scrU/\rN^1_\L\scrU)\o _{\scrU_0} \A'\o_{\scrU_0} (\scrU/\rN^1_\R\scrU),\\
\fA''&=\Phi(\A'')=(\scrU/\rN^1_\L\scrU)\o _{\scrU_0} \A''\o_{\scrU_0} (\scrU/\rN^1_\R\scrU),
\end{align*}
where $\Phi:=\Phi^\L\circ\Phi^\R$, and $\A'$ (resp. $\A''$) is a $\scrU_0$-module via~\eqref{eq:projhomo}. Since the product $\A\cong \A'\times\A''$ is finite, $\Phi$ commutes with it and $\fA\cong \fA'\op\fA''$ as $\scrU$--$\scrU$-bimodules. In particular, $\fA'$ (resp. $\fA''$) has the bigrading inherited from $\fA$:
\[
\fA'=\bigoplus_{m,n\in \N} \fA'_{n,-m}=\bigoplus_{m,n\in \N} (\scrU/\rN^1_\L\scrU)_n\o _{\scrU_0} \A'\o_{\scrU_0} (\scrU/\rN^1_\R\scrU)_{-m}.
\]
Let $\fA'_d=\fA'_{d,-d}$ for all $d\in \N$. We observe the following basic facts:
\begin{enumerate}
    \item  $\fA'$ is a $\scrU$--$\scrU$-bimodule, since $(\scrU/\rN^1_\L\scrU)_n$ and $(\scrU/\rN^1_\R\scrU)_{-m}$ are naturally left and right $\scrU$-modules, respectively. Moreover, the $\scrU$-actions match the degrees:  
\begin{align*}
    &\scrU_r\times \fA'_{n,-m}\ra \fA'_{r+n,-m},\quad (u,\fa')\mapsto u\cdot \fa',\\
    & \fA'_{n,-m}\times \scrU_{r}\ra \fA'_{n,-m+r},\quad (\fb',u)\mapsto \fb'\cdot u,
\end{align*}
for any $m,n\in \N$, $r\in \Z$, $\fa',\fb'\in \fA'_{n,-m}$, and $u\in \scrU_r$.
\item In particular, when $d=0$, it follows from Lemma~\ref{lm5.1} and~\eqref{eq:projhomo} that
\begin{equation}\label{eq:fA0identification}
    \begin{aligned}
    \fA'_0&=(\scrU_0/\rN^1_\L\scrU_0)\o _{\scrU_0} \A'\o_{\scrU_0} (\scrU_0/\rN^1_\R\scrU_0)\cong \A\o _{\scrU_0}\A'\o_{\scrU_0}\A\\
&\cong \A\o _{\A}\A'\o_{\A}\A \cong \A'.
    \end{aligned}
\end{equation}
\item There is a natural pairing map defined by~\eqref{eq:projhomo}:
\begin{equation}\label{eq:pairing}
\begin{aligned}
&(\scrU/\rN^1_\R\scrU)_{-j}\times (\scrU/\rN^1_\L\scrU)_k\ra \A\xrightarrow{p'}\A',\\
& (\al^{-j},\b^k)\mapsto \begin{cases}
p'([\al^{-j}\b^k]_0),&\text{if }j=k,\\
0,& \text{otherwise},
\end{cases}
\end{aligned}
\end{equation}
where $\al^{-j}=b^1_{[m_1]}\cdots b^s_{[m_s]}+\rN^1_\R \scrU $ and $\b^{k}=a^1_{[n_1]}\cdots a^r_{[n_r]}+\rN^1_\L \scrU$, with degree conditions $\sum_{i=1}^s \deg(b^i_{[m_i]})=-j$ and $\sum_{i=1}^r \deg(a^i_{[n_i]})=k$, respectively, and their product is given by $[\al^{-j}\b^k]_0=b^1_{[m_1]}\cdots b^s_{[m_s]}a^1_{[n_1]}\cdots a^r_{[n_r]}+\rN^1_\L\scrU_0\in \A$ when $j=k$, in view of Lemma~\ref{lm5.1}. 
\end{enumerate}

The following lemma is an immediate consequence of Definition~\ref{def:mta} and the definitions above.

\begin{lemma}
With the notation as above, $\fA'$ (resp. $\fA''$) is an associative algebra with respect to the $\star$-product defined using~\eqref{eq:defofprodonmta}, \eqref{eq:projhomo}, and~\eqref{eq:pairing}. More precisely, given $\b^i\o a'\o \al^{-j}\in \fA'_{i,-j}$ and $\b^k\o b'\o \al^{-l}\in \fA'_{k,-l}$, we have
\begin{equation}\label{eq:defofprodonmtaA'}
				\left(\b^i\o a'\o \al^{-j}\right) \star \left(\b^k\o b'\o \al^{-l}\right)= \b^i\o a'\cdot p'([\al^{-j}\b^k]_0)\cdot b'\o \al^{-l},
			\end{equation}
where $[\al^{-j}\b^k]_0:=0$ if $j\neq k$.
Moreover, the $\star$-product is compatible with the $\scrU$--$\scrU$-bimodule action
\begin{equation}\label{eq:starprodcompatibility}
    (\fa'\cdot u)\star\fb'=\fa'\star(u\cdot \fb'),
\end{equation}
for all $\fa'\in \fA'_{i,-j}$, $\fb'\in \fA'_{k,-l}$, and $u\in \scrU_r$.
\end{lemma}
\begin{proof}
The product~\eqref{eq:defofprodonmtaA'} is clearly associative. To show~\eqref{eq:starprodcompatibility}, we let $\fa'=\b^i\o a'\o \al^{-j}$ and $\fb'=\b^k\o b'\o \al^{-l}$. Then
\begin{align*}
(\fa'\cdot u)\star\fb'&=\left(\b^i\o a'\o \al^{-j}u\right)\star \left(\b^k\o b'\o \al^{-l}\right)\\
&=\b^i\o a'\cdot p'([(\al^{-j}u)\b^k]_0)\cdot b'\o \al^{-l}\\
&=\fa'\star(u\cdot \fb'),
\end{align*}
where we used the fact that $(\al^{-j}u)\b^k=\al^{-j}(u\b^k)$ in $\scrU$.
\end{proof}

The following lemma is~\cite[Lemma 5.1.5]{DGK2} in our situation.
\begin{deflemma}\label{def:strong_identity}
    With the notation as above, a sequence of elements $\I'_d\in \fA'_d$, with $d\in \N$ and $\I'_0=1_{\A'}\in \fA'_0\cong \A'$~\eqref{eq:fA0identification}, is said to satisfy the \textbf{(partial) strong identity property} if
\begin{equation}\label{eq:strongidentitydf'}
\fa'\star\I'_p=\fa'=\I'_q\star \fa'
\end{equation}
for all $p,q\in \N$ and all $\fa'\in \fA'_{q,-p}$. This holds if and only if the sequence $\{\I'_d\}_{d\in \N}$ satisfies the \textbf{strong identity equations}:
\begin{equation}\label{eq:5.18}
J_n(a)\cdot \I'_d=\I'_{d-n}\cdot J_{n}(a),
\end{equation}
for all $n\in \Z$, $d\in \N$ and all $J_n(a)=a_{[\wt a +n-1]}\in L(V)_{-n}\subset \scrU_{-n}$, with the convention $\I'_j:=0$ for $j<0$. (For $n>d$ equation~\eqref{eq:5.18} reads $J_n(a)\cdot \I'_d=0$, which holds automatically: the left-hand side lies in $\fA'_{d-n,-d}$, and $\fA'=\bigoplus_{m,n\in\N}\fA'_{n,-m}$ has no components of negative left degree.)
\end{deflemma}
\begin{proof}
Assume the sequence satisfies~\eqref{eq:strongidentitydf'}. Let $\fa'=J_n(a)\cdot\I'_d\in \fA'_{d-n,-d}$. Applying~\eqref{eq:starprodcompatibility} and~\eqref{eq:strongidentitydf'}, we have
\[
J_n(a)\cdot \I'_d=\I'_{d-n}\star(J_n(a)\cdot \I'_d)=(\I'_{d-n}\cdot J_n(a))\star \I'_d=\I'_{d-n}\cdot J_n(a),
\]
where the first and last equalities follow from~\eqref{eq:strongidentitydf'}, with $q=d-n$ and $p=d$, respectively, and the middle equality is~\eqref{eq:starprodcompatibility}.

Conversely, assume the sequence satisfies~\eqref{eq:5.18}. Then~\eqref{eq:strongidentitydf'} follows by the same induction as in~\cite[Lemma 5.1.5]{DGK2}; the base case is the identity $\fa'\star 1_{\A'}=\fa'=1_{\A'}\star \fa'$ for all $\fa'\in \fA'_0\cong \A'$, which holds because $\I'_0=1_{\A'}$.
 \end{proof}

\begin{remark}
As in~\cite[Lemma 5.1.5]{DGK2}, we can show that the partial strong identity property for the sequence $\I'_d\in \fA'_d$ is equivalent to 
\begin{equation}\label{eq:strongidentitydf''}
\fa'\star\I'_p=\fa'\quad \text{and}\quad \I'_q\star \fb'=\fb',
\end{equation}
for all $p,q\in \N$, and all $\fa'\in \fA'_{0,-p}$ and $\fb'\in \fA'_{q,0}$.
\end{remark}

\subsection{Partial strong identity elements of the mode transition algebras of $L_k(\g)$}
We now specialize to $\A=A(L_k(\g))$, with the splitting $\A=\A'\times \A''$ of Example~\ref{ex:Aadm}.

\begin{proposition}\label{prop:partial_strong_identity}
  With the notation as above, there exists a sequence of elements $\{\I'_d\in \fA'_d\}_{d\in \N}$ satisfying the partial strong identity property
  \eqref{eq:strongidentitydf'}.
\end{proposition}
\begin{proof}
Denote by $\scrU$ the universal enveloping algebra of the VOA $L_k(\g)$ and let $\la\in \Pr^k_\Z$. The right $\scrU$-module $\Phi^\R(L(\la)^\ast)$ becomes an $\N$-graded left $\scrU$-module---that is, an admissible $L_k(\g)$-module---via the anti-involution $\theta:\scrU\ra\scrU^{\mathrm{op}}$ of~\eqref{eq:involution}. We write ${}^\theta\Phi^\R(L(\la)^\ast)$ for this left $\scrU$-module.

We claim that
${}^\theta\Phi^\R(L(\la)^\ast)$ is isomorphic to the VOA-contragredient module $\Phi^\L(L(\la))^\vee$. Its bottom degree is, as a left $\A$-module, $L(\la)^\ast$, with action
\[
\<[a].f,v\>=\<f,\theta([a])v\>,\quad [a]\in \A,\ f\in L(\la)^\ast,\ v\in L(\la),
\]
which is also the $\A$-action on $\Phi^\L(L(\la))^\vee(0)=L(\la)^\ast$. Since ${}^\theta\Phi^\R(L(\la)^\ast)$ is an admissible module generated by its bottom degree, the adjunction of Definition~\ref{def:Verma} yields a surjection $\Phi^\L(L(\la)^\ast)\twoheadrightarrow{}^\theta\Phi^\R(L(\la)^\ast)$. By Lemma~\ref{lm:w0closed}(2) we have $L(\la)^\ast\cong L(-w_0(\la))$ with $-w_0(\la)\in \Pr^k_\Z$, so Theorem~\ref{thm:equiofcats} identifies $\Phi^\L(L(\la)^\ast)\cong \hat{L}_k(-w_0(\la))$, which is simple; a surjection from a simple module onto a nonzero one is an isomorphism. The same argument applied to $\Phi^\L(L(\la))^\vee$, which is likewise admissible and generated by its bottom degree $L(\la)^\ast$, gives
\begin{equation}\label{eq:contragredient}
    {}^\theta\Phi^\R(L(\la)^\ast)\cong \Phi^\L(L(\la)^\ast)\cong  \Phi^\L(L(\la))^\vee\cong \hat{L}_k(-w_0(\la)).
\end{equation}

Note that~\eqref{eq:contragredient} is an isomorphism of $L_k(\g)$-modules, and that the $\N$-grading on the admissible $L_k(\g)$-module $\Phi^\L(L(\la))^\vee$ is given by the graded dual spaces $\Phi^\L(L(\la))^\vee=\bigoplus_{d=0}^\infty \Phi^\L(L(\la))(d)^\ast$.
It follows that
\begin{align*}
\fA'_{n,-m}&=\Phi^\L(\Phi^\R(\A'))_{n,-m}\\
&=\prod_{\la\in \Pr^k_\Z}\left((\scrU/\rN^1_\L\scrU)_n\o_{\scrU_0}L(\la)\right)\o_\C \left(L(\la)^\ast\o_{\scrU_0}(\scrU/\rN^1_\R\scrU)_{-m}\right)\\
&=\prod_{\la\in \Pr^k_\Z} \Phi^\L(L(\la))(n)\otimes_\C \Phi^\R(L(\la)^\ast)(-m)\\
&\cong \prod_{\la\in \Pr^k_\Z} \Phi^\L(L(\la))(n)\otimes_\C \Phi^\L(L(\la))(m)^\ast,
\end{align*}
where the last isomorphism is~\eqref{eq:dualgraded} below.
The natural pairing between the quotients of $\scrU$ by the right and left neighborhoods~\eqref{eq:pairing} descends to a pairing between $\Phi^\R(L(\la)^\ast)(-j)$ and $\Phi^\L(L(\la))(j)$:
\[
\begin{tikzcd}
L(\la)^\ast\o_\C (\scrU/\rN^1_\R\scrU)_{-j}\times (\scrU/\rN^1_\L\scrU)_j\o _\C L(\la)\arrow[r,"\eqref{eq:pairing}"]\arrow[d,two heads] & L(\la)^\ast\o_\C \A'\o_\C L(\la)\arrow[d]\\
\Phi^\R(L(\la)^\ast)(-j)\times \Phi^\L(L(\la))(j)\arrow[r,dashed]& \C,
\end{tikzcd}
\]
where the left vertical arrow is the canonical surjection from $\o_\C$ onto $\o_{\scrU_0}$ and the right one is $f\o a\o v\mapsto \<f,a.v\>$. More precisely, given $f\o \al^{-j}\in \Phi^\R(L(\la)^\ast)(-j)$ and $\b^j\o v\in \Phi^\L(L(\la))(j)$, we have
\begin{equation}\label{eq:pairing2}
\<f\o \al^{-j}, \b^j\o v\>=\<f,p'([\al^{-j}\b^j]_0).v\>=\<f.p'([\al^{-j}\b^j]_0),v\>,
\end{equation}
where $f.p'([\al^{-j}\b^j]_0)$ is the natural right action of $p'([\al^{-j}\b^j]_0)\in \A'$ on $f\in L(\la)^\ast$. 
This pairing is perfect. Indeed, \eqref{eq:contragredient} identifies ${}^\theta\Phi^\R(L(\la)^\ast)$ with the graded dual $\Phi^\L(L(\la))^\vee=\bigoplus_{d\geq 0}\Phi^\L(L(\la))(d)^\ast$, and since the $\theta$-twist reverses the grading this gives, in each degree,
\begin{equation}\label{eq:dualgraded}
\Phi^\R(L(\la)^\ast)(-j)\cong\Phi^\L(L(\la))(j)^\ast.
\end{equation}
The contragredient module carries by definition the evaluation pairing with $\Phi^\L(L(\la))$, and transporting it through this identification returns~\eqref{eq:pairing2}; hence~\eqref{eq:pairing2} is a perfect pairing.
Moreover, the $\star$-product of the mode transition algebras $\fA'_{i,-j}\times \fA'_{k,-l}\ra \delta_{j,k}\fA'_{i,-l}$~\eqref{eq:defofprodonmtaA'} has an alternative description via~\eqref{eq:pairing2}
\begin{align*}
  & (\b^i\o (u_\la\o f_{\la})\o\al^{-j})_{\la\in \Pr^k_\Z}\star (\tilde{\b}^k\o (v_\mu\o g_{\mu})\o \tilde{\al}^{-l})_{\mu\in \Pr^k_\Z}\\
  &=\delta_{j,k}\<f_{\la},p'([\al^{-j}\tilde{\b}^k]_0).v_{\la}\>\cdot (\b^i\o (u_\la\o g_\la) \o \tilde{\al}^{-l})_{\la\in \Pr^k_\Z},
\end{align*}
which, under the identification ${}^\theta\Phi^\R(L(\la)^\ast)\cong  \Phi^\L(L(\la))^\vee$ again, can be written on elementary tensors as
\begin{equation}\label{eq:staralt}
    (U_\la\o F_{\la})_{\la\in \Pr^k_\Z}\star (V_\mu\o G_{\mu})_{\mu\in \Pr^k_\Z}=\left(\<F_\la,V_\la\> \cdot U_\la\o G_\la\right)_{\la\in \Pr^k_\Z},
\end{equation}
and extended bilinearly in general.
Since each $\hat{L}_k(\la)$, $\la\in\Pr^k_\Z$, is an ordinary module, its graded pieces $\Phi^\L(L(\la))(d)\cong \hat{L}_k(\la)(d)$ are \emph{finite-dimensional}; consequently
\[
\fA'_d\cong \prod_{\la\in\Pr^k_\Z}  \Phi^\L(L(\la))(d)\otimes_\C \Phi^\L(L(\la))(d)^\ast\cong  \prod_{\la\in\Pr^k_\Z} \End_\C\left(\Phi^\L(L(\la))(d)\right),
\]
where the second isomorphism uses finite-dimensionality. (This is precisely the point at which the restriction to the semisimple summand $\A'$ is used: the blocks of $\A''$ have infinite-dimensional graded pieces, and the identity operator on them does not lie in $\Phi^\L(-)(d)\o \Phi^\L(-)(d)^\ast$.)
Let $\I'_d\in \fA'_d$ be the identity element corresponding to $\prod_{\la\in \Pr^k_\Z} \Id_{\Phi^\L(L(\la))(d)}$. Since~\eqref{eq:staralt} holds in arbitrary bidegree, it follows that
\[
 (U_\la\o F_{\la})_{\la\in \Pr^k_\Z}\star \I'_p= (U_\la\o F_{\la})_{\la\in \Pr^k_\Z}=\I'_q\star  (U_\la\o F_{\la})_{\la\in \Pr^k_\Z},
\]
for all $p,q\in \N$ and all $ (U_\la\o F_{\la})_{\la\in \Pr^k_\Z}\in \fA'_{q,-p}$, where $U_\la\in  \Phi^\L(L(\la))(q)$ and $F_\la\in  \Phi^\L(L(\la))(p)^\ast$.
This shows the sequence $\{\I'_d\}_{d\in \N}$ satisfies the partial strong identity property~\eqref{eq:strongidentitydf'}.
\end{proof}




\section{Smoothing in genus zero} \label{sec:smoothing}

In this section we prove the main theorem. We first record the criterion that reduces local freeness of a coherent sheaf of coinvariants on $\overline{\M}_{0,n}$ to the constancy of its fiber dimension, and we recall from~\cite{DGK2} the smoothing morphism, which compares the coinvariants of a one-parameter family of curves acquiring a node with the coinvariants of the partial normalization of its central fiber, with the generalized Verma bimodule $\fA$ inserted at the two branches of the node. We then specialize to $V=L_k(\g)$ and to insertions from $\O_{k,\Ord}$: the vanishing theorem of Section~\ref{sec:fusion} yields a factorization theorem in which only the summand $\fA'$ contributes (Theorem~\ref{thm:restricted_factorization}). The partial strong identity elements of Section~\ref{sec:strong_identity} produce a smoothing morphism that restricts to this factorization isomorphism on the central fiber and hence is an isomorphism over the whole family (Theorem~\ref{thm:main_smoothing}); induction on the number of nodes then gives the constancy of the fiber dimension, hence Theorem~\ref{thm:mainA}.

\subsection{Smoothing setup}

\begin{proposition}\label{prop:flatness}
    Let $M^1,\dots, M^n$ be $V$-modules with rational conformal dimensions. Assume the sheaf of coinvariants $\VV_{0,n}(V,M^\bullet)$ is coherent over $\overline{\M}_{0,n}$. For a point $x\in \overline{\M}_{0,n}$, not necessarily closed, write $\VV_{0,n}|_x:=\VV_{0,n}(V,M^\bullet)\o\kappa(x)$ for the fiber at $x$; at a closed point $x=(C,P_\bullet)$ this is the space of coinvariants $[M^\bullet]_{(C,P_\bullet)}$. Then the following are equivalent:
    \begin{itemize}
        \item The dimension $\dim_{\kappa(x)}\VV_{0,n}|_x$ is the same for all points $x \in \overline{\M}_{0,n}$.
        \item The sheaf $\VV_{0,n}(V,M^\bullet)$ is locally free (that is, a vector bundle).
    \end{itemize}
\end{proposition}
\begin{proof}
    We follow the argument given in~\cite[Corollary 5.2.6]{DGK2}. 
    Recall that $\overline{\M}_{0,n}$ is a smooth projective variety and, in particular, an integral Noetherian scheme~\cite[Theorem 6.1]{Knudsen1983b}.    By hypothesis, $\VV_{0,n}(V,M^\bullet)$ is coherent, so~\cite[Exercise II.5.8]{Hartshorne} applies: a coherent sheaf on an integral Noetherian scheme is locally free if and only if the dimension of the fiber $\VV_{0,n}(V,M^\bullet)\o \kappa(x)$ is independent of the (not necessarily closed) point $x$. Since the formation of coinvariants is compatible with base change~\cite[Lemma 4.1.1]{DGK2}, at a closed point that fiber is the space of coinvariants $[M^\bullet]_{(C,P_\bullet)}$ of the corresponding curve, which justifies the description of the fibers given in the statement.
\end{proof}

We now relate the sheaf of coinvariants on a nodal curve to the one on its normalization, in the following situation.
\begin{setting}\label{setting}
    Let $(\CCC, P_\bullet, t_\bullet)$ be a family over $\mathcal{S}=\operatorname{Spec} \C\lb \varepsilon\rb $ whose central fiber $\CCC_0$ has a node $Q$, and which smooths this node: $\CCC$ is obtained by sewing the partial normalization of $\CCC_0$ at $Q$ with parameter $\varepsilon$, so that $\CCC$ has local equation $xy=\varepsilon$ at $Q$. Let $M^1,\dots,M^n$ be an $n$-tuple of $V$-modules attached to the points $P_1,\dots, P_n$, where $V$ is a VOA. Let $(\widetilde{\CCC},  P_\bullet \cup Q_\pm, t_\bullet \cup s_\pm) := (\widetilde{\CCC}_0 \times \S,  P_\bullet \cup Q_\pm, t_\bullet \cup s_\pm)$ be a trivial family extending the partial normalization $\eta: \widetilde{\CCC}_0 \rightarrow \CCC_0$, where $Q_\pm = \eta^{-1}(Q)$ and $s_\pm$ are formal coordinates at $Q_\pm$.
\end{setting}

The sheaf of chiral Lie algebras $\LL_{\CCC \setminus P_\bullet}(V)$ on a nodal curve is naturally identified with a subsheaf of chiral Lie algebras $\LL_{\widetilde{\CCC} \setminus (P_\bullet \cup Q_\pm)}(V)$ on its partial normalization by~\cite[Proposition 3.3.2]{DGT2}. In order to relate the sheaf of coinvariants on $(\CCC,P_\bullet, t_\bullet)$ to the one on $(\widetilde{\CCC},P_\bullet\cup Q_\pm, t_\bullet\cup s_\pm)$, we must specify $V$-modules at $Q_\pm$. The natural choice is to attach there the two sides of the generalized Verma bimodule $\fA=\Phi^\L(\Phi^\R(\A))$ of the Zhu algebra. This leads to the following definition.
\begin{definition}
    Assume Setting~\ref{setting}. Given $\I_0 \in \A \subset \fA$, a map
    \begin{equation}
         \alpha_0: M^\bullet \longrightarrow M^\bullet \otimes \mathfrak{A}, \quad w \mapsto w \otimes \I_0
    \end{equation}
    is called a \textbf{factorization morphism} if it is an $\LL_{\CCC_0 \setminus P_\bullet}(V)$-module homomorphism.
\end{definition}
\begin{proposition}[{\cite[Lemma 4.4.4]{DGK2}}]\label{prop:FM}
    Assume Setting~\ref{setting}. Then $1_\A \in \A \subset \fA$ defines a factorization morphism $\alpha_0$. Moreover, if $V$ is $C_1$-cofinite, it induces an isomorphism on the vector spaces of coinvariants
    \begin{equation}
         [\alpha_0]: [M^\bullet]_{(\CCC_0,P_\bullet,t_\bullet)} \stackrel{\cong}{\longrightarrow} [M^\bullet \otimes \fA]_{(\widetilde{\CCC}_0,P_\bullet \cup Q_\pm,t_\bullet \cup s_\pm)}, \quad w \mapsto w \otimes 1_{\A}.
    \end{equation}
\end{proposition}
\begin{definition}\label{def:smoothingmorphism}
    Assume Setting~\ref{setting}. Given $\I \in \fA\lb \varepsilon\rb $, let
    \begin{equation}
         \alpha: M^\bullet\lb \varepsilon\rb  \longrightarrow (M^\bullet \otimes \fA)\lb \varepsilon\rb, \quad w \mapsto w \otimes \I
    \end{equation}
    be a map induced by a morphism $M^\bullet \rightarrow M^\bullet \otimes \fA\lb \varepsilon\rb $, $w\mapsto w \otimes \I$ and extended by $\varepsilon$-adic continuity. The map $\alpha$ is called a \textbf{smoothing morphism} if it is an $\LL_{\CCC\setminus P_\bullet}(V)$-module homomorphism. (We do not require $\I_0=1_\A$, in contrast to the notion of smoothing map of~\cite[Section 5]{DGK2}.)
\end{definition}
\begin{proposition}\label{prop:smoothing_well_defined}
    Let $(\CCC, P_\bullet, t_\bullet)$ be a family of curves over $\operatorname{Spec}\C\lb \varepsilon\rb$, let $M^1,\dots,M^n$ be a collection of $V$-modules for a VOA $V$, let $\I=\sum_{d\geq 0}\I_d\varepsilon^d\in\fA\lb\varepsilon\rb$ with $\I_d\in \fA_d$ for every $d\in\N$, and let $\alpha$ be the associated map of Definition~\ref{def:smoothingmorphism}. The following are equivalent:
    \begin{itemize}
        \item $\alpha$ is a smoothing morphism;
        \item the sequence $(\I_d)_{d \in \N}$ of coefficients of $\I$ satisfies the strong identity equations~\eqref{eq:5.18}.
    \end{itemize}
\end{proposition}
\begin{proof}
This is~\cite[Proposition 5.1.2]{DGK2} and the proof translates verbatim. The one real difference is that
\cite[Definition 5.1.1]{DGK2} normalizes $\I_0=1_\A$, which we drop in
Definition~\ref{def:smoothingmorphism}; its proof does not use the normalization,
so the equivalence holds as stated, and this is what permits $\I_0=1_{\A'}$ in
Theorem~\ref{thm:main_smoothing}.
\end{proof}
\begin{remark}
    We emphasize that this proposition does not require the existence of strong identity elements in $\fA$, but only that the strong identity equations hold. For example, $\I=0$, that is, $\I_d\equiv 0$, satisfies those equations trivially, and the induced zero map is indeed a smoothing morphism, that is, a well-defined $\LL_{\CCC \setminus P_\bullet}(V)$-module homomorphism---though of course far from an isomorphism. It is the following proposition that supplies the injectivity of a smoothing map on coinvariants.
\end{remark}
\begin{proposition}\label{prop:smoothing}
    Assume Setting~\ref{setting}. Let $\fB\subseteq\fA$ be a $\scrU$--$\scrU$-sub-bimodule which is a direct summand of $\fA$ (e.g.\ $\fB=\fA$ or $\fB=\fA'$), and let $\alpha$ be a smoothing morphism with $\I\in\fB\lb\varepsilon\rb$. Then $\alpha$ takes values in $(M^\bullet\o\fB)\lb\varepsilon\rb$ and induces a well-defined morphism of sheaves of coinvariants
    \begin{equation}\label{eq:alphaB}
         [\alpha]: [M^\bullet\lb \varepsilon\rb]_{(\CCC,P_\bullet, t_\bullet)}  \longrightarrow [M^\bullet \otimes \fB \lb \varepsilon\rb ]_{(\widetilde{\CCC},P_\bullet \cup Q_\pm, t_\bullet \cup s_\pm)}, \quad w \mapsto w \otimes \I.
    \end{equation}
    Assume both sheaves in this equation are coherent. Then the following are equivalent:
    \begin{itemize}
        \item The morphism $[\alpha_0]$ obtained by restricting $[\alpha]$ to the central fiber is an isomorphism.
        \item $[\alpha]$ is an isomorphism.
    \end{itemize}
\end{proposition}
\begin{proof}
    Since $\fB$ is a $\scrU$--$\scrU$-sub-bimodule and a direct summand, $M^\bullet\o\fB$ is an $\LL_{\widetilde{\CCC}\setminus(P_\bullet\cup Q_\pm)}(V)$-submodule and direct summand of $M^\bullet\o\fA$, so the first assertion is clear. The only non-trivial direction of the equivalence is that an isomorphism $[\alpha_0]$ on the central fiber lifts to an isomorphism over the whole family; the argument is that of~\cite[Lemma 5.2.5]{DGK2}. Let $\mathcal{F}=[M^\bullet\lb \varepsilon\rb]_{(\CCC,P_\bullet, t_\bullet)}$ and $\mathcal{G}=[M^\bullet \otimes \fB \lb \varepsilon\rb ]_{(\widetilde{\CCC},P_\bullet \cup Q_\pm, t_\bullet \cup s_\pm)}$. By assumption both are finitely generated $\C\lb \varepsilon\rb$-modules, and since $\mathcal{G}$ is coherent and $\widetilde{\CCC}=\widetilde{\CCC}_0\times\S$ is the trivial family, \cite[Corollary 4.3.1]{DGK2} gives
    \[
    \mathcal{G}\cong[M^\bullet\o\fB]_{(\widetilde{\CCC}_0,P_\bullet\cup Q_\pm,t_\bullet\cup s_\pm)}\o_\C\C\lb \varepsilon\rb,
    \]
    a free $\C\lb \varepsilon\rb$-module; see also~\cite[Remark 5.2.4]{DGK2}. Since $[\alpha_0]=[\alpha]\otimes\C\lb\varepsilon\rb/(\varepsilon)$ is surjective, Nakayama's lemma~\cite[Theorem 2.2]{matsumura1989} shows that $[\alpha]$ is surjective. As $\mathcal{G}$ is free, the surjection $[\alpha]$ splits, $\mathcal{F}\cong\ker[\alpha]\oplus\mathcal{G}$, and reducing modulo $\varepsilon$ gives $\mathcal{F}/(\varepsilon)\cong\ker[\alpha]/(\varepsilon)\oplus\mathcal{G}/(\varepsilon)$ with $[\alpha_0]$ the projection onto the second summand. Since $[\alpha_0]$ is injective, $\ker[\alpha]/(\varepsilon)=0$. Finally, $\ker[\alpha]$ is a submodule of the finitely generated module $\mathcal{F}$ over the Noetherian ring $\C\lb\varepsilon\rb$, hence finitely generated, so Nakayama's lemma~\cite[Theorem 2.2]{matsumura1989} gives $\ker[\alpha]=0$.
\end{proof}

\subsection{Proof of the main theorem}
In this subsection, we restrict to the setting of the main theorem (Theorem~\ref{thm:mainA}). Its two components are the restricted factorization theorem (Theorem~\ref{thm:restricted_factorization}) and the smoothing theorem (Theorem~\ref{thm:main_smoothing}).

\begin{figure}[htbp]
    \centering
    
    \begin{subfigure}[b]{0.32\textwidth}
        \centering
        \includegraphics[width=\linewidth, page=1]{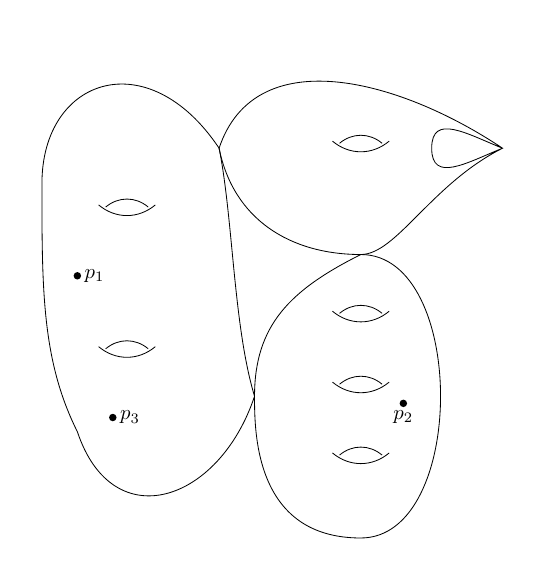}
        \caption{Topological surface}
        \label{fig:dual_panel1}
    \end{subfigure}\hfill
    \begin{subfigure}[b]{0.32\textwidth}
        \centering
        \includegraphics[width=\linewidth, page=2]{Graph.pdf}
        \caption{Algebraic curve}
        \label{fig:dual_panel2}
    \end{subfigure}\hfill
    \begin{subfigure}[b]{0.32\textwidth}
        \centering
        \includegraphics[width=\linewidth, page=3]{Graph.pdf}
        \caption{Dual graph}
        \label{fig:dual_panel3}
    \end{subfigure}
    
    \caption{Construction of the dual graph of a stable curve.}
    \label{fig:dual_graph}
\end{figure}

Let us recall some terminology from the moduli theory of curves. Following~\cite[Definition X.2.16]{arbarello2011geometry}, in the context of nodal pointed curves a graph $\Gamma$ consists of the following data:
\begin{itemize}
    \item A finite set of vertices $\mathrm{Vert}(\Gamma)$ and a finite set of half-edges $H(\Gamma)$;
    \item An assignment $H(\Gamma) \rightarrow \mathrm{Vert}(\Gamma)$ of a vertex to each half-edge, and an involution $\iota:H(\Gamma) \rightarrow H(\Gamma)$.
\end{itemize}
Orbits of $\iota$ of size two are called edges, $E_j:=\{h_j,\iota(h_j)\}$; fixed points of $\iota$ are called legs. (We write $\mathrm{Vert}$ and $H$ rather than $V$ and $L$, which are reserved for the VOA and for simple modules.)

To any nodal curve $C$ we associate its dual graph $\Gamma_C$: the vertices are the irreducible components of $C$, the edges correspond to the nodes, and the legs correspond to the marked points. In particular, self-intersecting components correspond to vertices carrying a loop; the construction is illustrated in Figure~\ref{fig:dual_graph}.

A leaf of a graph is by definition a vertex that has only one edge. In terms of the original curve, a leaf is an irreducible component meeting the rest of the curve in exactly one node.
Since a stable curve $(C,P_\bullet)$ of genus zero has arithmetic genus $p_a(C)=b_1(\Gamma_C)+\sum_i p_a(C_i)=0$~\cite[Chapter X, Section 2]{arbarello2011geometry}, where $\Gamma_C$ is the dual graph and $C_i$ are the irreducible components, every component is a $\P^1$ and $\Gamma_C$ is a tree. In particular every node of $C$ is separating, and if $C$ has at least one node then $\Gamma_C$ has at least two leaves. We first extend the vanishing theorem from smooth to stable curves.

\begin{lemma}\label{lm:nodal_vanishing}
Let $V=L_k(\g)$ for an admissible number $k$, let $(C,P_0\cup P_\bullet)\in\overline{\M}_{0,n+1}$, and let $M^0=\Phi^\L(\A_\mu)$ with $\mu\in[\Pr^k\bs P]$ be inserted at $P_0$ and $M^1,\dots,M^n\in\O_{k,\Ord}$ at $P_\bullet$. Then $[M^0\o M^\bullet]_{(C,P_0\cup P_\bullet)}=0$.
\end{lemma}
\begin{proof}
We induct on the number $e$ of nodes of $C$; for $e=0$ this is Corollary~\ref{coro:fusion_closedness_general}. Let $e\geq 1$ and choose a leaf component $C_+\cong\P^1$ not containing $P_0$, attached to the rest of the curve $C_-$ at a node $Q$; such a component exists because $\Gamma_C$ is a tree with at least two vertices, hence has at least two leaves, while $P_0$ lies on a single component; the construction is illustrated in Figure~\ref{fig:main_three_panels}. Let $\tilde{C}=C_+\sqcup\tilde{C}_-\to C$ be the partial normalization at $Q$, with $Q_\pm\in C_\pm$, and write $M^\bullet=M_+^\bullet\o M_-^\bullet$ according to the marked points on $C_\pm$. Applying the factorization isomorphism of Proposition~\ref{prop:FM} at $Q$, using that the chiral Lie algebra splits as $\LL_{\tilde{C}\setminus (P_0\cup P_\bullet\cup Q_\pm)}(V)\cong \LL_{C_+\setminus (P^+_\bullet\cup Q_+)}(V)\times \LL_{\tilde{C}_-\setminus (P_0\cup P^-_\bullet\cup Q_-)}(V)$ and that $\A\cong\prod_{\nu\in[\Pr^k]}\A_\nu$ is a finite product, we obtain
\begin{equation}\label{eq:leaf_split}
    [M^0\o M^\bullet]_{(C,P_0\cup P_\bullet)} \cong \bigoplus_{\nu \in [\Pr^k]} [M_+^\bullet \otimes \Phi^\L(\A_\nu)]_{(C_+,P_\bullet^+\cup Q_+)} \otimes_{\A_\nu} [M^0  \otimes M_-^\bullet  \otimes {}^\theta\Phi^\R(\A_\nu)]_{(\tilde{C}_-,P_0 \cup P_\bullet^-\cup Q_-)}.
\end{equation}
If $\nu\in[\Pr^k\bs P]$, the first factor vanishes by Corollary~\ref{coro:fusion_closedness_general}, since $C_+$ is smooth. If $\nu\in\Pr^k_\Z$, then $\A_\nu\cong L(\nu)\o_\C L(\nu)^\ast$ by Theorem~\ref{thm:AvE_result}, so ${}^\theta\Phi^\R(\A_\nu)$ is a finite direct sum of copies of ${}^\theta\Phi^\R(L(\nu)^\ast)\cong\hat{L}_k(\nu)^\vee\in\O_{k,\Ord}$~\eqref{eq:contragredient}, and the second factor vanishes by the inductive hypothesis, as $\tilde{C}_-$ has $e-1$ nodes and carries $M^0$ together with ordinary insertions only.
\end{proof}

\begin{theorem}[Restricted factorization theorem]\label{thm:restricted_factorization}
Let $V = L_k(\g)$ for an admissible number $k$ and let $M^1,\dots,M^n$ be $V$-modules in $\O_{k,\Ord}$. Let $(C,P_\bullet)\in \overline{\M}_{0,n} \setminus \M_{0,n}$ be a curve with at least one node $Q$, and let $\tilde{C} \rightarrow C$ be its partial normalization at $Q$. Then we have a factorization isomorphism
\begin{equation} \label{eq:ordinary_factorization}
    [M^\bullet]_{(C,P_\bullet)} \cong \bigoplus_{\la \in \Pr_\Z^k} [M_+^\bullet \otimes \hat{L}_k(\la)]_{(\tilde{C}_+,P_\bullet^+\cup Q_+)} \otimes [\hat{L}_k(\la)^\vee  \otimes M_-^\bullet ]_{(\tilde{C}_-,Q_- \cup P_\bullet^- )}
\end{equation}
induced by $M^\bullet \rightarrow M^\bullet \otimes \Phi(\A')$,  $w \mapsto w \otimes 1_{\A'}$.
\end{theorem}

\begin{figure}[htbp]
    \centering
    
    \begin{subfigure}[b]{0.32\textwidth}
        \centering
        \includegraphics[width=\linewidth, page=1]{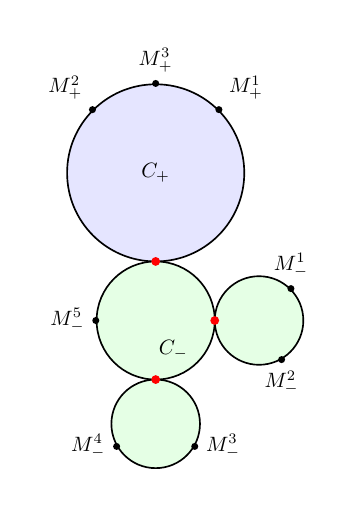}
        \caption{Singular curve}
        \label{fig:panel1}
    \end{subfigure}\hfill
    \begin{subfigure}[b]{0.32\textwidth}
        \centering
        \includegraphics[width=\linewidth, page=2]{Pic.pdf}
        \caption{Applied factorization}
        \label{fig:panel2}
    \end{subfigure}\hfill
    \begin{subfigure}[b]{0.32\textwidth}
        \centering
        \includegraphics[width=\linewidth, page=3]{Pic.pdf}
        \caption{Restriction to $\Pr^k_\Z$}
        \label{fig:panel3}
    \end{subfigure}
    
    \caption{Proof of Lemma~\ref{lm:nodal_vanishing} and Theorem~\ref{thm:restricted_factorization}. (a) The singular curve of genus zero. (b) Factorization applied to decouple one node. (c) Restriction to the ordinary modules decouples the module structure.}
    \label{fig:main_three_panels}
\end{figure}

\begin{proof}
Since $Q$ is separating, $\tilde{C}=\tilde{C}_+\sqcup\tilde{C}_-$ with $Q_\pm\in\tilde{C}_\pm$; write $M^\bullet=M_+^\bullet\o M_-^\bullet$ accordingly. Exactly as in~\eqref{eq:leaf_split}, the factorization isomorphism of Proposition~\ref{prop:FM} at $Q$, induced by $w\mapsto w\o 1_\A$, gives
\begin{equation}\label{eq:full_split}
    [M^\bullet]_{(C,P_\bullet)} \cong \bigoplus_{\la \in [\Pr^k]} [M_+^\bullet \otimes \Phi^\L(\A_\la)]_{(\tilde{C}_+,P_\bullet^+\cup Q_+)} \otimes_{\A_\la} [{}^\theta\Phi^\R(\A_\la) \otimes M_-^\bullet]_{(\tilde{C}_-,Q_- \cup P_\bullet^-)}.
\end{equation}
By Lemma~\ref{lm:nodal_vanishing}, the summands with $\la\in[\Pr^k\bs P]$ vanish. For $\la\in\Pr^k_\Z$, Theorem~\ref{thm:AvE_result} gives $\A_\la\cong L(\la)\o_\C L(\la)^\ast$, whence $\Phi^\L(\A_\la)\o_{\A_\la}{}^\theta\Phi^\R(\A_\la)\cong \Phi^\L(L(\la))\o{}^\theta\Phi^\R(L(\la)^\ast)\cong\hat{L}_k(\la)\o\hat{L}_k(\la)^\vee$ by~\eqref{eq:contragredient}. This is~\eqref{eq:ordinary_factorization}; and since $1_\A=(1_{\A'},1_{\A''})$ with the $\A''$-components of~\eqref{eq:full_split} equal to zero, the isomorphism is induced by $w\mapsto w\o 1_{\A'}$.
\end{proof}

\begin{theorem}[Smoothing theorem]\label{thm:main_smoothing}
    Assume Setting~\ref{setting}, where $\CCC_0$ is a stable curve of genus zero. Let $V = L_k(\g)$ be an admissible-level affine VOA, let $M^1,\dots, M^n$ be in $\O_{k,\Ord}$, and let $\I:=\sum_{d\geq 0}\I'_d\varepsilon^d\in\fA'\lb\varepsilon\rb\subset\fA\lb\varepsilon\rb$, where $\{\I'_d\}_{d\in\N}$ are the partial strong identity elements of Proposition~\ref{prop:partial_strong_identity}. Then $\alpha:w\mapsto w\o\I$ is a smoothing morphism, and it induces isomorphisms
    \begin{equation}\label{eq:smoothing_iso}
    [M^\bullet\lb\varepsilon\rb]_{(\CCC,P_\bullet,t_\bullet)}\ \xrightarrow{\ [\alpha]\ }\ [(M^\bullet\o\fA')\lb\varepsilon\rb]_{(\widetilde{\CCC},P_\bullet\cup Q_\pm,t_\bullet\cup s_\pm)}\ \cong\ [M^\bullet\o\fA']_{(\widetilde{\CCC}_0,P_\bullet\cup Q_\pm,t_\bullet\cup s_\pm)}\o_\C\C\lb\varepsilon\rb.
    \end{equation}
\end{theorem}

\begin{proof}
    \emph{$\alpha$ is a smoothing morphism.} By Definition/Lemma~\ref{def:strong_identity} the sequence $\{\I'_d\}$ satisfies the strong identity equations~\eqref{eq:5.18} in $\fA'$; since $\fA'$ is a $\scrU$--$\scrU$-sub-bimodule of $\fA$ with the inherited bigrading, these equations hold in $\fA$ and $\I'_d\in\fA'_d\subset\fA_d$. Hence Proposition~\ref{prop:smoothing_well_defined} applies.

    \emph{Coherence of the source.} Let $\varkappa:\S\ra\overline{\M}_{0,n}$ be the classifying morphism of $(\CCC,P_\bullet)$. Because $L_k(\g)$ is strongly generated in degree one, $\VV_{0,n}(L_k(\g),M^\bullet)$ is coherent on $\overline{\M}_{0,n}$ by~\cite[Corollary A]{DG}.
    Base change and the comparison of algebraic with formal coinvariants     \cite[Lemma 4.1.1, Propositions 4.2.2 and 4.2.3]{DGK2} then give
    \begin{equation}\label{eq:basechange}
    [M^\bullet\lb\varepsilon\rb]_{(\CCC,P_\bullet,t_\bullet)}\ \cong\
    \varkappa^\ast\,\VV_{0,n}(L_k(\g),M^\bullet),
    \end{equation}
    coherent over $\C\lb\varepsilon\rb$.

        \emph{Coherence of the target.} By~\eqref{eq:Aadm} and~\eqref{eq:contragredient},
    $\fA'\cong\bigoplus_{\la\in\Pr^k_\Z}\hat{L}_k(\la)\o\hat{L}_k(\la)^\vee$ is a
    \emph{finite} sum of pairs of simple ordinary insertions at $Q_\pm$---this is where
    $\fA'$ replaces $\fA$, whose $\A''$-part is infinite-dimensional. Splitting
    $\widetilde{\CCC}_0=\widetilde{C}_+\sqcup\widetilde{C}_-$ at the separating node $Q$ as
    in the proof of Theorem~\ref{thm:restricted_factorization}, and applying the argument
    of~\eqref{eq:basechange} to each component of the trivial family
    $\widetilde{\CCC}=\widetilde{\CCC}_0\times\S$, identifies
    $[(M^\bullet\o\fA')\lb\varepsilon\rb]_{(\widetilde{\CCC},P_\bullet\cup Q_\pm,t_\bullet\cup s_\pm)}$ with
    \[
    \bigoplus_{\la\in \Pr^k_\Z}\Big(\VV_{0,n_++1}\big(L_k(\g),(M_+^\bullet,\hat{L}_k(\la))\big)\big|_{x_+}\o_\C \VV_{0,n_-+1}\big(L_k(\g),(M_-^\bullet,\hat{L}_k(\la)^\vee)\big)\big|_{x_-}\Big)\o_\C\C\lb\varepsilon\rb,
    \]
    where $x_\pm:=[(\widetilde{C}_\pm,P^\pm_\bullet\cup Q_\pm)]$. Each fiber $\VV_{0,n_\pm+1}|_{x_\pm}$ is finite-dimensional by~\cite[Corollary A]{DG}, all insertions being simple ordinary modules. This is the second isomorphism in~\eqref{eq:smoothing_iso}: a free $\C\lb\varepsilon\rb$-module of finite rank, in particular coherent.

    \emph{$[\alpha]$ is an isomorphism.} Both hypotheses of Proposition~\ref{prop:smoothing}, with $\fB=\fA'$, are now verified, so it suffices to show that the restriction $[\alpha_0]$ of $[\alpha]$ to the central fiber is an isomorphism. That restriction is the map induced by $w\mapsto w\o 1_{\A'}$, which is an isomorphism onto $[M^\bullet\o\fA']_{(\widetilde{\CCC}_0,P_\bullet\cup Q_\pm,t_\bullet\cup s_\pm)}$ by Theorem~\ref{thm:restricted_factorization}.
\end{proof}

\begin{theorem}[Main theorem]\label{thm:mainA}
Let $\g$ be a complex simple Lie algebra, let $L_k(\g)$ be a simple affine VOA at an admissible level $k$, and let $n\geq 3$. Let $M^\bullet=(M^1,\dots, M^n)$ be an $n$-tuple of $L_k(\g)$-modules in the category $\O_{k,\Ord}$. Then the sheaf of coinvariants $\VV_{0,n}(L_k(\g),M^\bullet)$ is coherent, locally free, and globally generated over $\overline{\M}_{0,n}$. In particular, the dual sheaf of conformal blocks is a vector bundle on $\overline{\M}_{0,n}$.
\end{theorem}
\begin{proof}
    In this proof we closely follow~\cite[Corollary 5.2.6]{DGK2}, although we do not assume the existence of strong identity elements.

    We first record that the sheaf is defined on $\overline{\M}_{0,n}$ and coherent there. Since every module in $\O_{k,\Ord}$ is a finite direct sum of the simple modules $\hat{L}_k(\la)$, $\la\in\Pr^k_\Z$, all of which have rational conformal dimension, Definition/Proposition~\ref{prop:isocoinvfiber} applies summandwise and $\widetriangle{\VV}_{0,n}(L_k(\g),M^\bullet)$ descends from $\widetriangle{\M}_{0,n}$ to a sheaf $\VV_{0,n}(L_k(\g),M^\bullet)$ on $\overline{\M}_{0,n}$; it is coherent by~\cite[Corollary A]{DG}, since $L_k(\g)$ is strongly generated in degree one, applying that corollary summandwise to the simple ordinary summands of the $M^i$.

    By Proposition~\ref{prop:flatness}, $\VV_{0,n}(L_k(\g),M^\bullet)$ is then a vector bundle if and only if the fiber dimension $\dim_{\kappa(x)}\VV_{0,n}|_x$ is constant for all points $x \in \overline{\M}_{0,n}$. We first prove this constancy at closed points $x=(C,P_\bullet)$, where the fiber is $[M^\bullet]_{(C,P_\bullet)}$, by induction on the number $e$ of nodes, noting that $0\leq e\leq n-3$ on $\overline{\M}_{0,n}$; the passage to arbitrary points is carried out at the end of the proof.

    For smooth $n$-pointed curves the dimension of $[M^\bullet]_{(C, P_\bullet, t_\bullet)}$ is constant across the connected smooth locus $\M_{0,n}$, because the sheaf of coinvariants carries a projectively flat connection there~\cite[Main theorem]{DGT1}; let $r$ denote this common dimension. This is the base case $e=0$. It also yields a global lower bound: since $\overline{\M}_{0,n}$ is integral and $\VV_{0,n}$ is coherent, there is a dense open subset over which $\VV_{0,n}$ is locally free; that subset meets $\M_{0,n}$, and $\overline{\M}_{0,n}$ is a variety, so the open subset contains closed points of $\M_{0,n}$ and the generic rank is $r$. By upper semicontinuity of fiber dimension \cite[Exercise II.5.8]{Hartshorne},
    \begin{equation}\label{eq:lowerbound}
    \dim\,[M^\bullet]_{x}=\dim_{\kappa(x)}\VV_{0,n}\big|_x\ \geq\ r\qquad\text{for \emph{every} point }x\in\overline{\M}_{0,n}.
    \end{equation}

    For the inductive step, let $x_0=(C_0,P_\bullet)$ be a closed point whose curve has $e+1$ nodes, and assume $\dim[M^\bullet]_z=r$ at every closed point $z$ with at most $e$ nodes. Choose a node $Q$ of $C_0$ and let $\CCC\ra\S$ be the family of Setting~\ref{setting} obtained by sewing the partial normalization $\widetilde{\CCC}_0$ at $Q_\pm$ with parameter $\varepsilon$: its central fiber is $(C_0,P_\bullet)$, and its geometric fiber over the generic point $\xi$ of $\S$ has exactly $e$ nodes, the node $Q$ acquiring the local equation $xy=\varepsilon$ while the remaining nodes persist. By Theorem~\ref{thm:main_smoothing},
    \[
    \mathcal F:=[M^\bullet\lb\varepsilon\rb]_{(\CCC,P_\bullet,t_\bullet)}\ \cong\ [M^\bullet\o\fA']_{(\widetilde{\CCC}_0,P_\bullet\cup Q_\pm,t_\bullet\cup s_\pm)}\o_\C\C\lb\varepsilon\rb
    \]
    is a \emph{free} $\C\lb\varepsilon\rb$-module; write $d$ for its rank. By the base-change identification~\eqref{eq:basechange}, applied to the present family with classifying morphism $\varkappa$, base change to the two points of $\S$~\cite[Lemma 4.1.1]{DGK2} gives $d=\dim \mathcal F/\varepsilon\mathcal F=\dim\VV_{0,n}|_{\varkappa(0)}=\dim[M^\bullet]_{x_0}$ and $\dim_{\kappa(y)}\VV_{0,n}|_y=d$ for $y:=\varkappa(\xi)$, and $y$ lies in the locally closed locus $Z_e\subset\overline{\M}_{0,n}$ of curves with exactly $e$ nodes. The subset $\overline{\{y\}}\cap Z_e$ is nonempty and locally closed in the variety $\overline{\M}_{0,n}$, hence contains a closed point $z$, and $z$ has $e$ nodes, so $\dim\VV_{0,n}|_z=r$ by the inductive hypothesis. Since $z\in\overline{\{y\}}$, upper semicontinuity gives $d=\dim\VV_{0,n}|_y\leq \dim\VV_{0,n}|_z=r$; combined with~\eqref{eq:lowerbound} at $x_0$ this yields
    \[
    r\ \leq\ \dim[M^\bullet]_{x_0}\ =\ d\ \leq\ r,
    \]
    so $\dim[M^\bullet]_{x_0}=r$, completing the induction. (Note that no algebraization of the smoothing family is required: because $\overline{\M}_{0,n}$ is a fine moduli space, the formal disc maps to it directly, and the comparison between the formal family and the global sheaf is~\eqref{eq:basechange}.)

    Constancy at closed points implies constancy at every point: given any $y\in\overline{\M}_{0,n}$, the closure $\overline{\{y\}}$ contains a closed point $z$, and $\dim\VV_{0,n}|_y\leq\dim\VV_{0,n}|_z=r$ by upper semicontinuity, while~\eqref{eq:lowerbound} gives the reverse inequality. This establishes the constancy of the fiber dimension, and hence local freeness. Finally, since $L_k(\g)$ is strongly generated by its degree-one subspace $L_k(\g)_1$, it follows from~\cite[Main theorem]{DG} that $\VV_{0,n}(L_k(\g),M^\bullet)$ is globally generated.
\end{proof}

\section{Fusion rules and rank of the $L_k(\g)$-coinvariant bundles}\label{sec:fusion_rules}
As an application of our results from the previous sections, we give a formula for the fusion rules of the ordinary $L_k(\g)$-modules using the language of conformal blocks (Section~\ref{subsec:fusion3pt}), and then specialize it to a concrete formula for $\g=\sl_2$ (Section~\ref{subsec:sl2fusion}).
Then we give a rank formula for the $L_k(\g)$-coinvariant bundles (Section~\ref{subsec:rank_trees}).

\subsection{Fusion tensor product of the category $\O^\fin_{k,\Ord}$}\label{subsec:fusion3pt}
It was proved in~\cite{creutzig2018braided} that the category $\O_{k,\Ord}$ admits a braided tensor category structure $\boxtimes$ in the sense of Huang--Lepowsky--Zhang. Together with the main theorem in~\cite{A16}, this implies that $\O_{k,\Ord}$ is a semisimple braided tensor category with finitely many simple objects, so that $\boxtimes$ of two simple objects decomposes as a finite direct sum of simples. (Rigidity, and hence the property of being a \emph{fusion} category in the strict sense, is known in types $ADE$~\cite{Creutzig19} but open in types $B$, $C$, $F_4$, $G_2$; see Remark~\ref{rmk:hypF}. Nothing below uses rigidity.) Under the equivalence of categories  $(\Phi^\L\dashv \Om):\O_{k,\Ord}^\fin\leftrightarrows \O_{k,\Ord}$, the subcategory $\O_{k,\Ord}^\fin\subset \Mod(\A)$ inherits the same structure.
We first give a description of the fusion tensor product as well as the fusion rules of simple objects in the category $\O_{k,\Ord}^\fin$.

Let $V$ be a general VOA. 
It is well known that for any ordinary $V$-modules $M^1,M^2,M^3$, the space of three-point conformal blocks on $\P^1$ is naturally isomorphic to the space of (formal) intertwining operators~\cite{FHL93}, whose dimension is the fusion rule~\cite{TUY,L23,GLZ24}:
\begin{equation}\label{eq:cb_io}
 [(M^3)'\o M^1\o M^2]^\ast_{(\P^1,(\infty,1,0),(1/z,z-1,z))}\cong I\fusion{M^1}{M^2}{M^3},
 \end{equation}
where $(M^3)'$ is the VOA-contragredient module.

Using the technique of restricted VOA-conformal blocks, the second-named author introduced the following construction for $M^1,M^2\in \Adm(V)$ in~\cite{L26}:
\begin{equation}\label{eq:M1odotM2}
M^1\odot M^2=\frac{M^1\o M^2}{\mathcal{L}_{\P^1\bs\{1,0\}}(V)_{<0}. (M^1\o M^2)},
\end{equation}
where $\mathcal{L}_{\P^1\bs\{1,0\}}(V)_{<0}$ is a subalgebra of the chiral Lie algebra $\mathcal{L}_{\P^1\bs\{1,0,\infty\}}(V)$ defined by
\begin{align*}
&\mathcal{L}_{\P^1\bs\{1,0\}}(V)_{<0}=\frac{\bigoplus_{r\geq 0} V_r\otimes H^0(\P^1\bs \{1,0\}, \om_{\P^1}^{1-r}(-r [\infty]))}{\im\nabla }\\
&=\spn\left\{a\o \frac{z^n}{(z-1)^m} (dz)^{1-\wt a}+\im\nabla: a\in V,\ m,n\in\Z,\  n-m<\wt a-1\right\}.
\end{align*}
In~\cite[Theorem 4.8]{L26} it was proved that $M^1\odot M^2$ is a left $\A$-module if $M^1$ or $M^2$ is generated by its bottom degree.
Furthermore,
$M^1\odot M^2$ is spanned by the symbols $v^1\odot v^2$, which are bilinear in $v^1$ and $v^2$, subject to the following relations:
 \begin{equation}\label{eq:uniformrel}
\sum_{j\geq 0} \binom{\wt a-s}{j}  a_{(j-t)}v^1\odot v^2
	=-\sum_{j\geq 0}\binom{-t}{j} (-1)^{-t-j} v^1\odot a_{(\wt a-s+j)}v^2,
\end{equation}
for all $s,t\in \Z$ such that $s+t>1$,
with the left $\A$-module action on $M^1\odot M^2$ given by
\begin{equation}\label{eq:Aaction}
[a]. (v^1\odot v^2)=\sum_{j\geq 0} \binom{\wt a-1}{j} a_{(j)}v^1\odot v^2+ v^1\odot o(a)v^2,
\end{equation}
where $a\in V$ is homogeneous, $v^1\in M^1$, $v^2\in M^2$, and $o(a):=a_{(\wt a-1)}$ denotes the zero-degree mode. 

\begin{theorem}[{\cite[Theorem 6.2]{L26}}]\label{thm:odotfusion}
Let $M^1,M^2,M^3$ be ordinary $V$-modules. Suppose  $(M^3)'$ is isomorphic to the generalized Verma module $\Phi^\L(\Om((M^3)') )$ associated with the left $\A$-module $\Om((M^3)')$, and either $M^1$ or $M^2$ is generated by its bottom degree. Then 
\[
I\fusion{M^1}{M^2}{M^3}\cong \Hom_\A(M^1\odot M^2,\Om(M^3)).
\]
\end{theorem}

Now we let $M^1=\hat{L}_k(\la)$, $M^2=\hat{L}_k(\mu)$, and $M^3=\hat{L}_k(\nu)$ be simple objects in the category $\O_{k,\Ord}\subset \Adm(L_k(\g))$ as in Definition~\ref{def:cats}, where $\la,\mu,\nu\in \Pr^k_\Z=\Pr^k\cap P$. Since both $M^3$ and $(M^3)'$ are irreducible $V$-modules, we have $\Om(M^3)=M^3(0)$ and $\Om((M^3)')=M^3(0)^\ast$ as left $\A$-modules. By the equivalence of categories  $(\Phi^\L\dashv \Om):\O_{k,\Ord}^\fin\leftrightarrows \O_{k,\Ord}$ in Corollary~\ref{coro:equiofcats}, we have $\Phi^\L(\Om((M^3)'))\cong (M^3)'$. Moreover, both $M^1$ and $M^2$ are generated by their bottom degrees. Therefore, by Theorem~\ref{thm:odotfusion}, we have
\begin{equation}\label{eq:homlamunu}
I\fusion{\hat{L}_k(\la)}{\hat{L}_k(\mu)}{\hat{L}_k(\nu)}\cong \Hom_\A(\hat{L}_k(\la)\odot \hat{L}_k(\mu), L(\nu)),
\end{equation}
where $\A=A(L_k(\g))$ and $L(\nu)=\Om(\hat{L}_k(\nu))=\hat{L}_k(\nu)(0)$. Denote the dimension of the vector space~\eqref{eq:homlamunu} by $N_{\la\mu}^\nu$.

\begin{proposition}\label{prop:6.2}
    The tensor products of simple objects $L(\la)$ and $L(\mu)$ in the braided tensor category $\O^\fin_{k,\Ord}\subset \Mod(\A)$ admit the following identification:
\begin{equation}\label{eq:affinefinitefusion}
L(\la)\boxtimes L(\mu)\cong \hat{L}_k(\la)\odot \hat{L}_k(\mu),\quad \forall\  \la,\mu\in \Pr^k_\Z=\Pr^k\cap P.
\end{equation}
\end{proposition}
\begin{proof}
We first show
that the canonical map
\[
L(\la)\o_\C L(\mu)\ra \hat{L}_k(\la)\odot \hat{L}_k(\mu),\quad u\o v\mapsto u\odot v
\]
is surjective; see also~\cite[Lemma 7.5]{L26} for a special case.
Indeed, since $\g\cong L_k(\g)_1$ and $\hat{L}_k(\la)=U(\hat{\g}_{<0}).L(\la)$, it follows from~\cite[Lemma 4.5]{L26} that
\[
\hat{L}_k(\la)\odot \hat{L}_k(\mu)=\spn\{u\odot a^1(-n_1)\cdots a^r(-n_r)v: a^i\in\g,\ n_i\geq 1,\ u\in L(\la),\ v\in L(\mu)\}.
\]
Given a spanning element $u\odot a^1(-n_1)\cdots a^r(-n_r)v$ of $\hat{L}_k(\la)\odot \hat{L}_k(\mu)$, if $r=0$ then $u\odot v$ lies in the image of $L(\la)\o_\C L(\mu)$. Assume $r\geq 1$, and write $a^1(-n_1)\cdots a^r(-n_r)v=a^1(-n_1)v'$.
Then, by~\eqref{eq:uniformrel}, we have
\[
u\odot a^1(-n_1)v'=-\sum_{j\geq 0}\binom{-n_1}{j} a^1(j)u\odot v'=-a^1(0)u\odot v',
\]
where $a^1(0)u\in L(\la)$. Using induction on the length $r$ of $a^1(-n_1)\cdots a^r(-n_r)v$, we conclude that $u\odot a^1(-n_1)\cdots a^r(-n_r)v$ lies in the image of $L(\la)\o_\C L(\mu)$. Thus $L(\la)\o_\C L(\mu)\ra \hat{L}_k(\la)\odot \hat{L}_k(\mu)$ is surjective.
In particular, the $\A=A(L_k(\g))$-module $\hat{L}_k(\la)\odot \hat{L}_k(\mu)$ is finite-dimensional and is in the category $\O^\fin_{k,\Ord}$.

Since the fusion tensor on $\O^\fin_{k,\Ord}$ is given by the pull-back of the fusion tensor on $\O_{k,\Ord}$ under the equivalence of categories  $(\Phi^\L\dashv \Om):\O_{k,\Ord}^\fin\leftrightarrows \O_{k,\Ord}$, we have
\begin{align*}
L(\la)\boxtimes L(\mu)&=\Om(\Phi^\L(L(\la))\boxtimes \Phi^\L(L(\mu)))\cong\Om(\hat{L}_k(\la)\boxtimes \hat{L}_k(\mu))=\bigoplus_{\nu\in \Pr^k_\Z}N_{\la\mu}^\nu\Om(\hat{L}_k(\nu))\\
&=\bigoplus_{\nu\in \Pr^k_\Z}\dim \Hom_\A\left(\hat{L}_k(\la)\odot \hat{L}_k(\mu), L(\nu) \right) \cdot L(\nu)\\
&\cong \hat{L}_k(\la)\odot \hat{L}_k(\mu),
\end{align*}
where the last isomorphism follows from the fact that the category $\O^\fin_{k,\Ord} $ is semisimple, with simple objects $L(\nu)$ for $\nu\in \Pr^k_\Z$.
\end{proof}

\subsection{Fusion tensor of admissible-level $L_{k}(\sl_2)$-ordinary modules}\label{subsec:sl2fusion}
In this subsection, we discuss an alternative way of calculating the fusion rules among irreducible ordinary modules over the type $A_1$ affine VOA at an admissible level using~\eqref{eq:affinefinitefusion}.

Let $\la\in \Pr^k_\Z$. As a $\hat{\g}$-module, $\hat{L}_k(\la)=V^k(\la)/N^k(\la)$, where $V^k(\la)=U(\hat{\g})\o_{U(\g[t]\op\C K)}L(\la)$ is the generalized Verma (Weyl) module and $N^k(\la)$ is its maximal proper submodule, generated by the singular vectors in $V^k(\la)$.
In order to calculate $\hat{L}_k(\la)\odot \hat{L}_k(\mu)$, we need a concrete description of $N^k(\la)$.

For the Lie algebra $\g=\sl_2 = \left<e,f,h\right>$, with $\Delta=\{\al,-\al\}$ and the fundamental weight $\omega$, the submodule $N^k(\la)\subset V^k(\la)$ can be described explicitly using the main results in~\cite{malikov1986singular}. Therefore, we can use~\eqref{eq:affinefinitefusion} to determine the fusion tensor product for admissible-level ordinary $L_k(\sl_2)$-modules.
Let
\[
k+2=\frac{p}{q},\qquad p,q\in \Z_{>0},\ (p,q)=1,\ p\geq 2.
\]
The admissible weights at level $k$ are parametrized by
$
\la_{r,s}
=
\left(
r-1-(s-1)\frac{p}{q}
\right)\omega,
$
where
$
1\leq r\leq p-1,
$ and $
1\leq s\leq q$~\cite{kac1989classification,DLM96}.
 Then $\la_{r,s}\in P=\Z\om$ if and only if
$
q\mid (s-1)
$, or equivalently, $s=1$ since $s\leq q$. Hence
\begin{equation}\label{eq:6.9}
    \Pr^k_\Z
=
\Pr^k\cap P
=
\left\{
m\omega:
0\leq m\leq p-2
\right\},
\end{equation}
and the simple objects in $\O_{k,\Ord}$ are given by $\{\hat{L}_k(m\om):0\leq m\leq p-2\}$.

Let $M^k(m\omega)$ be the affine Verma module with highest-weight vector $v_m$. Then the maximal proper submodule of $M^k(m\om)$ is generated by the two singular vectors $F_1(m+1,1)v_m=f(0)^{m+1}v_m$ and $F_2(m+1,1)v_m$~\cite{DLM96,malikov1986singular}. Let
$
\pi_m:M^k(m\omega)\twoheadrightarrow V^k(m\omega)
$
be the canonical quotient. Let
\[
\chi_m:=\pi_m(F_2(m+1,1)v_m)\in V^k(m\om).
\]
With the standard Malikov--Feigin--Fuchs formal-power notation~\cite{malikov1986singular}, we may write
\begin{equation}\label{eq:MFFsingular}
\chi_m=
e(-1)^{(2q-1)\kappa-m-1}
f(0)^{(2q-2)\kappa-m-1}
\cdots
e(-1)^{\kappa-m-1}v_m,\quad \kappa= k+2=\frac{p}{q},
\end{equation}
where the right-hand side is understood in the Malikov--Feigin--Fuchs formal-power sense, and the factors alternate between $e(-1)$ and $f(0)$. Then $\hat{L}_k(m\om)=V^k(m\om)/U(\hat{\sl}_2)\chi_m$.

Note that the total number of times $e(-1)$ appears in $\chi_m$ is
$
\sum_{j=1}^q ((2j-1)\kappa-m-1)=q^2\kappa-q(m+1)
$,
while the total number of times $f(0)$ appears in $\chi_m$ is
$
\sum_{j=1}^{q-1}(2j\kappa-m-1)=q(q-1)\kappa-(q-1)(m+1).
$ We have
\begin{equation}\label{eq:chimeigenvalue}
\begin{aligned}
    h(0)\chi_m&=\Big(m+2\big(q^2\kappa-q(m+1)-q(q-1)\kappa+(q-1)(m+1)\big)\Big)\chi_m\\
    &=(2p-m-2)\chi_m.
\end{aligned}
\end{equation}

\begin{lemma}\label{lm:6.3}
Let $0\leq m,n\leq p-2$. The Weyl modules $V^k(m\om)$ and $V^k(n\om)$ are also modules over the vacuum module VOA $V^k(\sl_2)$. We have the following isomorphism of modules over $A(V^k(\sl_2))\cong U(\sl_2)$~\cite{FZ92}:
\[
V^k(m\om)\odot V^k(n\om)\cong L(m\om)\o_\C L(n\om),
\]
where the right-hand side is the usual tensor product of $\sl_2$-modules.
\end{lemma}
\begin{proof}
As in the proof of Proposition~\ref{prop:6.2}, the canonical map
\[
\iota_{m,n}:L(m\om)\o_\C L(n\om)\ra V^k(m\om)\odot V^k(n\om),\quad u\o v\mapsto u\odot v
\]
is surjective. Conversely, fix PBW bases of the generalized Verma modules $V^k(m\om)$ and $V^k(n\om)$, consisting of ordered monomials $a^1(-r_1)\cdots a^s(-r_s)u$ with $a^i\in\{e,h,f\}$, $r_i\geq 1$ and $u$ running over a basis of $L(m\om)$ (resp.\ $L(n\om)$). Define a linear map
\[
\tilde\varphi_{m,n}: V^k(m\om)\o_\C V^k(n\om)\ra L(m\om)\o_\C L(n\om)
\]
on the tensor products of PBW basis elements by induction on the total degree, as follows. For $u\in L(m\om)$ and $v\in L(n\om)$, let $\tilde\varphi_{m,n}(u\o v):=u\o v$. A PBW basis element of positive degree can be written uniquely as $a(-r)u$ with $r\geq 1$, $a\in \{e,h,f\}$, and $u$ a PBW basis element of lower degree (peel off the leftmost factor). If the left tensor factor has positive degree, set
\begin{equation}\label{eq:defvarphimn}
\tilde\varphi_{m,n}(a(-r)u\o v):=-\sum_{j\geq 0}\binom{-r}{j}(-1)^{-r-j} \tilde\varphi_{m,n}(u\o a(j)v),
\end{equation}
and if the left factor $u$ lies in $L(m\om)$ while the right factor has positive degree, set
\begin{equation}\label{eq:defvarphimn2}
\tilde\varphi_{m,n}(u\o a(-r)v):=-\sum_{j\geq 0}\binom{-r}{j} \tilde\varphi_{m,n}(a(j)u\o v).
\end{equation}
(The two formulas are the specializations $(s,t)=(1,r)$ and $(s,t)=(1+r,0)$ of~\eqref{eq:uniformrel}, which is why the sign $(-1)^{-r-j}$ occurs in the first and not in the second.) On the right-hand sides we first express $u\o a(j)v$ and $a(j)u\o v$ in terms of the PBW bases---these have lower total degree since $j\geq 0$---and then apply $\tilde\varphi_{m,n}$ using the induction hypothesis. A direct check, as in the proof of~\cite[Lemma 7.5]{L26}, shows that $\tilde\varphi_{m,n}$ annihilates the defining relations~\eqref{eq:uniformrel} of the product $V^k(m\om)\odot V^k(n\om)$. Hence $\tilde\varphi_{m,n}$ descends to a well-defined linear map $\varphi_{m,n}: V^k(m\om)\odot V^k(n\om)\ra L(m\om)\o_\C L(n\om)$ with $\varphi_{m,n}\circ \iota_{m,n}=\Id$, and~\eqref{eq:defvarphimn}--\eqref{eq:defvarphimn2} hold with $\o$ replaced by $\odot$.

Finally, for any $a\in \sl_2=V^k(\sl_2)_1$ and $u\odot v\in V^k(m\om)\odot V^k(n\om)$, by~\eqref{eq:Aaction}, we have
\[
[a].(u\odot v)=a(0)u\odot v+u\odot o(a)v,
\]
and $o(a)=a(0)$. Hence $\varphi_{m,n}$ is an $\sl_2$-module isomorphism.
\end{proof}


By the Clebsch--Gordan theorem, $L(m\om)\o_\C L(n\om)\cong \bigoplus_{i=0}^{\min\{m,n\}} L((m+n-2i)\om)$.
For each $0\leq i\leq \min\{m,n\}$, the vector
\begin{equation}\label{eq:sl2highestweight}
u_i^{m,n}=
\sum_{r=0}^{i}
(-1)^r
\binom{i}{r}
\frac{(n-i+r)!}{(n-i)!}
\frac{(m-r)!}{m!}\,
f(0)^rv_m\otimes f(0)^{i-r}v_n,
\end{equation}
 is the $\sl_2$-highest-weight vector of the direct summand $L((m+n-2i)\om)$ in the tensor product $L(m\om)\o_\C L(n\om)$.

 Since $\hat{L}_k(m\om)=V^k(m\om)/U(\hat{\sl}_2)\chi_m$, where $\chi_m$ is the singular vector~\eqref{eq:MFFsingular}, and the restricted chiral Lie algebra $\mathcal{L}_{\P^1\bs\{1,0\}}(V^k(\sl_2))_{<0}$ surjects onto $\mathcal{L}_{\P^1\bs\{1,0\}}(L_k(\sl_2))_{<0}$ via the quotient map $V^k(\sl_2)\twoheadrightarrow L_k(\sl_2)$, it follows that the isomorphism $\varphi_{m,n}$ of Lemma~\ref{lm:6.3} induces an isomorphism
\begin{equation}\label{eq:firstdescriptionofodot}
\hat{L}_k(m\om)\odot \hat{L}_k(n\om)\simeq \frac{L(m\om)\o_\C L(n\om)}{U(\sl_2).\left(\varphi_{m,n}(\chi_m\odot L(n\om)+L(m\om)\odot \chi_n )\right)}.
\end{equation}

In order to determine the fusion tensor product using~\eqref{eq:affinefinitefusion}, we need to give a more concrete description of $\varphi_{m,n}(\chi_m\odot L(n\om)+L(m\om)\odot \chi_n )$. With the notation as above, write
\begin{equation}\label{eq:aandd}
  a:=p-m-1,\qquad  d:=m+n-p+1=n-a.
\end{equation}
Observe that $\chi_m\odot L(n\om)=\spn\{\chi_m\odot f(0)^jv_n:0\leq j\leq n\}$. For any $0\leq j\leq n$, it follows from~\eqref{eq:Aaction} and~\eqref{eq:chimeigenvalue} that the following relation holds in $V^k(m\om)\odot V^k(n\om)$:
\begin{align*}
    [h].(\chi_m\odot f(0)^jv_n)&=h(0)\chi_m\odot f(0)^jv_n+\chi_m\odot o(h)f(0)^jv_n\\
    &=(m+n+2a-2j) (\chi_m\odot f(0)^jv_n).
\end{align*}
Since $\varphi_{m,n}$ is an $\sl_2$-module isomorphism, $\varphi_{m,n}(\chi_m\odot f(0)^jv_n)\in L(m\om)\o_\C L(n\om)$ is of $h(0)$-weight $m+n+2a-2j$.
Then, by the Clebsch--Gordan theorem, $\varphi_{m,n}(\chi_m\odot f(0)^jv_n)=0$ in $L(m\om)\o_\C L(n\om)$ if $m+n+2a-2j>m+n$, that is, if $j<a$. Denote by
\begin{equation}\label{eq:defRi}
R_i:=\varphi_{m,n}(\chi_m\odot f(0)^{a+i}v_n)\in \varphi_{m,n}(\chi_m\odot L(n\om)),
\end{equation}
where $0\leq i\leq d=n-a$.  Note that $R_i$ is of $h(0)$-weight $m+n-2i$.

\begin{lemma}\label{lm:6.4}
For $0\leq i\leq d$, define the following submodule of $L(m\om)\o_\C L(n\om)$:
\[
J_{<i}:=\sum_{\ell=0}^{i-1} U(\sl_2).u^{m,n}_\ell\cong \bigoplus_{\ell=0}^{i-1} L((m+n-2\ell)\om).
\]
Then $R_i\in J_{<i+1}$, and
there exists a nonzero number $C_i\in \C$ such that
\begin{equation}\label{eq:Ricongruence}
R_i\equiv C_i \cdot v_m\o f(0)^iv_n\pmod{J_{<i}}.
\end{equation}
In particular, $J_{<d+1}=U(\sl_2).\spn\{R_0,\ds, R_d\}$.
\end{lemma}
\begin{proof}
The only summands of $L(m\om)\o_\C L(n\om)$ that contribute to the $h(0)$-weight space of weight $m+n-2i$ are $L((m+n-2\ell)\om)$, where $0\leq \ell\leq i$. Hence $R_i\in J_{<i+1}$.

We want to use the relations of $\odot$ to reduce $R_i$ modulo $J_{<i}$.
By~\eqref{eq:Aaction} and~\eqref{eq:uniformrel}, for any $w\in L(n\om)$, the $\odot$-product satisfies the following relations:
\[
e(-1)u\odot w=u\odot e(0)w,\quad [f].(u\odot w)=f(0)u\odot w+u\odot f(0)w.
\]
Then $f_{\Delta}.\varphi_{m,n}(u\odot w)=\varphi_{m,n}(f(0)u\odot w+u\odot f(0)w)$ in $L(m\om)\o_\C L(n\om)$, where $f_\Delta$ is the diagonal action on the tensor product.
Note that if $f_\Delta \al$ has $h(0)$-weight $m+n-2i$, which is the case for $R_{i}$, then $\al$ must have weight $m+n-2i+2$. The only summands of $L(m\om)\o_\C L(n\om)$ that contain this weight are $L((m+n-2\ell)\om)$ for $0\leq \ell\leq i-1$, which are all contained in $J_{<i}$. Hence $f_\Delta\al\equiv 0\pmod{J_{<i}}$, and
\[
\varphi_{m,n}(f(0)u\odot w)\equiv -\varphi_{m,n}(u\odot f(0)w)\pmod{J_{<i}},
\]
if $f(0)u\odot w$ is of $h(0)$-weight $m+n-2i$. Define
\[
\Phi_i:
V^k(m\omega)\odot V^k(n\omega)
\xrightarrow{\varphi_{m,n}}
L(m\omega)\otimes L(n\omega)
\twoheadrightarrow \frac{L(m\omega)\otimes L(n\omega)}{J_{<i}}.
\]
Then the relations above can be written as
\[
\Phi_i(e(-1)u\odot w)
=
\Phi_i(u\odot e(0)w)
\] and, whenever the resulting weight of $f(0)u\odot w$ is $m+n-2i$, we have
\[
\Phi_i(f(0)u\odot w)
=
-\Phi_i(u\odot f(0)w).
\]
In particular, for any positive integer $r$, we have
\begin{equation}\label{eq:tworel}
\Phi_i(e(-1)^ru\odot w)=\Phi_i(u\odot e(0)^rw),\quad
\Phi_i(f(0)^ru\odot w)=\Phi_i(u\odot (-f(0))^rw).
\end{equation}
Since the formal expression of $\chi_m$~\eqref{eq:MFFsingular} involves rational powers of $e(-1)$ and $f(0)$, we introduce the following $\sl_2$-module, on which complex powers of $e$ and $f$ are defined. Let
\[
\tilde{L}(n\om)=\bigoplus_{\b\in \C} \C v_\b,\qquad
f.v_\b=v_{\b+1},\quad e.v_\b=\b(n-\b+1)v_{\b-1},\quad h.v_\b=(n-2\b)v_\b.
\]
The $\sl_2$-relations are easy to verify.
Note that $L(n\om)$ is the quotient of $\spn\{v_\b:\b\in \Z_{\geq 0}\}$ by the submodule $\spn\{v_\b:\b\in \Z_{\geq n+1}\}$, via $v_\b\mapsto f(0)^\b v_n$.

On $\tilde{L}(n\om)$, the complex powers
\[
(-f)^\al v_\b=\exp(\pi i\al)\,v_{\b+\al},\qquad e^\al v_\b=c_n(\al,\b)v_{\b-\al}
\]
are well-defined and satisfy $e^{\al}e^{\al'}=e^{\al+\al'}$, where
\begin{equation}\label{eq:Gammafun}
   c_n(\al,\b)
=
\frac{\Gamma(\beta+1)}
{\Gamma(\beta-\alpha+1)}
\frac{\Gamma(n-\beta+\alpha+1)}
{\Gamma(n-\beta+1)}.
\end{equation}
If $\al,\b$ are non-negative integers, these formulas specialize to the usual relations in $L(n\om)$.

Consider the map $\tau:\sl_2[t^{-1}]\ra \sl_2$, $\tau(x(-r))=(-1)^{r+1}x$, which satisfies $\tau([X,Y])=[\tau(Y),\tau(X)]$. Then $\tau$ induces an algebra homomorphism $\Theta:U(\sl_2[t^{-1}])\ra U(\sl_2)^{\mathrm{op}}$ with $\Theta(e(-1))=e$ and $\Theta(f(0))=-f$. It follows from~\eqref{eq:tworel} that
\begin{equation}\label{eq:transfer}
\Phi_i(u\,v_m\odot w)=\Phi_i\big(v_m\odot \Theta(u)w\big),\qquad u\in U(\sl_2[t^{-1}]),\ w\in L(n\om),
\end{equation}
whenever $u\,v_m\odot w$ has $h(0)$-weight $m+n-2i$. Since $\Theta$ respects the commutation rule $X^\al Y=\sum_{l\geq 0}\binom{\al}{l}\big((\ad X)^lY\big)X^{\al-l}$, it extends to the complex powers occurring in~\eqref{eq:MFFsingular}, sending $e(-1)^\al$ and $f(0)^\al$ to the operators $e^\al$ and $(-f)^\al$ on $\tilde{L}(n\om)$. We define
\[
\tilde\Phi_i(u\,v_m\odot \xi):=\Phi_i\big(v_m\odot \Theta(u)\,\xi\big),\qquad \xi\in \tilde{L}(n\om),
\]
where $u$ is a Malikov--Feigin--Fuchs monomial as in~\eqref{eq:MFFsingular}. Then, by~\eqref{eq:transfer}, $\tilde\Phi_i=\Phi_i$ whenever all exponents are non-negative integers and $\xi\in L(n\om)$. The passage to the fractional exponents of~\eqref{eq:MFFsingular} is to be understood in the Malikov--Feigin--Fuchs formal sense~\cite{malikov1986singular}: the identity~\eqref{eq:transfer} is first established for non-negative integer exponents, where it is a consequence of~\eqref{eq:tworel}, and is then extended to complex exponents by the same formal-calculus argument that makes the right-hand side of~\eqref{eq:MFFsingular} a well-defined vector; see~\cite{feigin1994fusion} for this technique in the $\hat{\sl}_2$ fusion setting.

Write $A_r=(2r-1)\kappa-m-1,
$ and $
B_r=2r\kappa-m-1$. Then, by~\eqref{eq:MFFsingular},
\[
\chi_m=e(-1)^{A_q}f(0)^{B_{q-1}}\cdots f(0)^{B_1}e(-1)^{A_1}v_m.
\]
On the other hand, write $\beta_r:=i+r\kappa-m-1.$ Then $\b_q=i+p-m-1=a+i$. Now using the formal extension $\tilde\Phi_i$, together with the fact that $\tilde\Phi_i=\Phi_i$ on vectors of $V^k(m\om)\odot V^k(n\om)$ with integral exponents, we obtain
\begin{align*}
&\Phi_i(\chi_m\odot f(0)^{a+i}v_n)\\
&\equiv \tilde\Phi_i\left(e(-1)^{A_{q-1}}f(0)^{B_{q-2}}\ds v_m\odot (-f(0))^{B_{q-1}}e(0)^{A_q}f(0)^{\b_q}v_n\right)    \\
&=\pm c_n(A_q,\b_q)\cdot \tilde\Phi_i\left(e(-1)^{A_{q-1}}f(0)^{B_{q-2}}\ds v_m\odot f(0)^{i+(q-1)\kappa-m-1} v_n\right)\\
&= \pm c_n(A_q,\b_q)\cdot \tilde\Phi_i
\left(e(-1)^{A_{q-2}}f(0)^{B_{q-3}}\ds v_m\odot (-f(0))^{B_{q-2}}e(0)^{A_{q-1}}f(0)^{\b_{q-1}} v_n\right)\\
&=\pm c_n(A_q,\b_q)c_n(A_{q-1},\b_{q-1})\cdot \tilde\Phi_i\left(e(-1)^{A_{q-2}}f(0)^{B_{q-3}}\ds v_m\odot f(0)^{i+(q-2)\kappa-m-1}v_n\right)\\
&\vdots \\
&= C'\cdot\tilde\Phi_i(e(-1)^{A_1}v_m\odot f(0)^{\b_1}v_n)\\
&=C_i\cdot \tilde\Phi_i(v_m\odot f(0)^iv_n)=C_i\cdot \Phi_i(v_m\odot f(0)^iv_n)\\
&=C_i\cdot v_m\o f(0)^iv_n,
\end{align*}
where $
C_i
=
\exp\big((B_1+\cdots+B_{q-1})\pi i\big)
\prod_{r=1}^{q}
c_n(A_r,\beta_r)$.

This constant is nonzero. Indeed, substituting $\al=A_r$ and $\b=\b_r$ in \eqref{eq:Gammafun} and using $q\kappa=p$, $m=p-a-1$ and $m+n+2-p=d+1$, we have 
\[
c_n(A_r,\b_r)=\frac{\Gamma\big(i+a+1-(q-r)\kappa\big)\,\Gamma\big(n-i+1+(r-1)\kappa\big)}{\Gamma\big(i+1-(r-1)\kappa\big)\,\Gamma\big(d-i+1+(q-r)\kappa\big)}.
\]
Since $\Gamma$ has no zeros and only simple poles at $\Z_{\leq 0}$, we have $c_n(A_r,\b_r)\in \C^\times$ unless one of these four arguments lies in $\Z_{\leq 0}$; a non-integral argument imposes no condition. As $(p,q)=1$, for $1\leq r\leq q$ we have $(r-1)\kappa\in \Z$ only for $r=1$, and $(q-r)\kappa\in \Z$ only for $r=q$. The arguments to be checked are therefore $i+1$ and $n-i+1$ (when $r=1$), and $i+a+1$ and $d-i+1$ (when $r=q$); all four lie in $\Z_{\geq 1}$, since $0\leq i\leq d\leq n$ and $a=p-m-1\geq 1$ by \eqref{eq:aandd}. Hence $C_i\neq 0$, and this shows \eqref{eq:Ricongruence}.

Now we use induction on $i$ to show that $J_{<i+1}=U(\sl_2).\spn\{R_0,\ds, R_i\}$ for $0\leq i\leq d$. We first claim that
\begin{equation}\label{eq:uiDi}
u_i^{m,n}\equiv D_i\cdot v_m\o f(0)^iv_n\pmod{J_{<i}},\qquad D_i\neq 0.
\end{equation}
Let $W_i\subset L(m\om)\o_\C L(n\om)$ be the $h(0)$-weight space of weight $m+n-2i$; it has basis $\{f(0)^rv_m\o f(0)^{i-r}v_n\}_{0\leq r\leq i}$. Each summand $L((m+n-2\ell)\om)$ with $0\leq \ell\leq i$ meets $W_i$ in a line, so $\dim (J_{<i}\cap W_i)=i$ and $W_i=(J_{<i}\cap W_i)\op \C u_i^{m,n}$. Let $(\cdot,\cdot)$ be the tensor product of the contravariant forms on $L(m\om)$ and $L(n\om)$; it is a nondegenerate $\sl_2$-contravariant form on $L(m\om)\o_\C L(n\om)$ for which the basis $\{f(0)^rv_m\o f(0)^{i-r}v_n\}_r$ of $W_i$ is orthogonal with nonzero norms, and for which distinct isotypic components are orthogonal. Hence $J_{<i}\cap W_i=\{w\in W_i: (u_i^{m,n},w)=0\}$. Since the coefficient of $v_m\o f(0)^iv_n$ in~\eqref{eq:sl2highestweight} is $1$, we get $(u_i^{m,n},v_m\o f(0)^iv_n)=(v_m,v_m)(f(0)^iv_n,f(0)^iv_n)\neq 0$, so $v_m\o f(0)^iv_n\notin J_{<i}$. As $W_i/(J_{<i}\cap W_i)$ is one-dimensional and $u_i^{m,n}\notin J_{<i}$, this proves~\eqref{eq:uiDi}. Combining~\eqref{eq:uiDi} with~\eqref{eq:Ricongruence}, we obtain \[R_i\equiv s_i\cdot u^{m,n}_i\pmod{J_{<i}}\] for some scalar $s_i\neq 0$. In particular, $J_{<1}=U(\sl_2).u_0^{m,n}=U(\sl_2).\spn\{R_0\}$. Assume $J_{<i}=U(\sl_2).\spn\{R_0,\ds, R_{i-1}\}$. Then
because $u_i^{m,n}\equiv s_i^{-1}R_i\pmod{J_{<i}}$, we have
\[U(\sl_2).u^{m,n}_i\ssq U(\sl_2).( s_i^{-1}R_i+J_{<i})\subseteq U(\sl_2).\spn\{R_0,\ds, R_i\}.\]
Hence $J_{<i+1}=U(\sl_2).\spn\{R_0,\ds, R_i\}$.
\end{proof}

\begin{proposition}\label{prop:fusiontensor}
With the notation as above,
write 
\[\mathcal{J}_{m,n}=J_{<d+1}=\sum_{\ell=0}^{d} U(\sl_2).u^{m,n}_\ell\cong\bigoplus_{\ell=0}^{d}L((m+n-2\ell)\om),\]
where $d=m+n-p+1=n-a$~\eqref{eq:aandd}.
Note that $\mathcal{J}_{m,n}=0$ if $m+n<p-1$, i.e., if $d<0$; the elements $R_i$ and the subspaces $J_{<i}$ above are only defined when $d\geq 0$. Then
	\begin{equation}\label{eq:tensor}
L(m\om )\boxtimes L(n\om)\cong \hat{L}_k(m\om)\odot \hat{L}_k(n\om)\cong \frac{L(m\om )\o_\C L(n\om)}{\mathcal{J}_{m,n}}.
\end{equation}
\end{proposition}
\begin{proof}
In view of~\eqref{eq:firstdescriptionofodot}, it suffices to show
\[U(\sl_2).\left(\varphi_{m,n}(\chi_m\odot L(n\om)+L(m\om)\odot \chi_n )\right)=\mathcal{J}_{m,n}.
\]
We first show that $\varphi_{m,n}(\chi_m\odot f(0)^jv_n)\in \mathcal{J}_{m,n}$ for all $0\leq j\leq n$. 
Recall that $\varphi_{m,n}(\chi_m\odot f(0)^jv_n)=0$ if $j<a$. So it suffices to show $R_i\in \mathcal{J}_{m,n}$ for all $0\leq i\leq d$, where $R_i=\varphi_{m,n}(\chi_m\odot f(0)^{a+i}v_n)$.
By Lemma~\ref{lm:6.4}, we have
\[
R_i\in J_{<i+1}=\sum_{\ell=0}^{i}U(\sl_2).u^{m,n}_{\ell}\subset \mathcal{J}_{m,n}.
\]
 Similarly, we can show $\varphi_{m,n}(f(0)^jv_m\odot \chi_n)\in \mathcal{J}_{m,n}$ for all $0\leq j\leq m$.

Conversely, by Lemma~\ref{lm:6.4} again, we have $J_{<i+1}=U(\sl_2).\spn\{R_0,\ds, R_i\}$. Hence \[\mathcal{J}_{m,n}=J_{<d+1}=U(\sl_2).\spn\{R_0,\ds, R_d\}\subseteq U(\sl_2).\varphi_{m,n}(\chi_m\odot L(n\om)).\]
Hence $U(\sl_2).\left(\varphi_{m,n}(\chi_m\odot L(n\om)+L(m\om)\odot \chi_n )\right)=\mathcal{J}_{m,n}.$
\end{proof}

It follows immediately from~\eqref{eq:tensor} that the fusion tensor product in the category $\O_{k,\Ord}^\fin$ is given by
\[
L(m\om)\boxtimes L(n\om)\cong \bigoplus_{\ell=\max\{0,m+n-p+2\}}^{\min\{m,n\}} L((m+n-2\ell)\om).
\]
Equivalently, the fusion rules among irreducible modules in $\O_{k,\Ord}^\fin$ or $\O_{k,\Ord}$ are given by
\[
N_{m,n}^{\,r}
=
\begin{cases}
1,
&
r=m+n-2\ell
\text{ for some }
\max\{0,m+n-p+2\}\le \ell\le\min\{m,n\},\\[4pt]
0,
&\text{otherwise},
\end{cases}
\]
 for $0\leq m,n,r\leq p-2$; cf.\ \eqref{eq:6.9}.
By dropping the index $\ell$, we can rewrite them as
\begin{equation}\label{eq:fusionrules}
N_{m,n}^{\,r}
=
\begin{cases}
1,
&
|m-n|\le r\le
\min\{m+n,\;2p-4-m-n\}\\
& \text{and}\ \ r\equiv m+n\pmod{2},\\
0, &\text{otherwise}.
\end{cases}
\end{equation}
This agrees with the fusion rules among ordinary modules of $L_k(\sl_2)$ at level $k+2=\frac{p}{q}$ in~\cite{DLM96,feigin1994fusion,creutzig2018braided,Creutzig19}.

\subsection{Rank and degree of the $L_k(\g)$-coinvariant bundles}\label{subsec:rank_trees}

The fusion rules $N_{\la\mu}^\nu$ of the category $\O_{k,\Ord}$ were characterized in Section~\ref{subsec:fusion3pt} for arbitrary $\g$ and computed explicitly for $\g=\sl_2$ in Section~\ref{subsec:sl2fusion}. By independent methods, the fusion rules of admissible-level affine VOAs of type $ADE$ were identified with the fusion rules of the integrable level $k'=p-h^\vee$ in~\cite[Corollary 7.4]{Creutzig19}; see Theorem~\ref{thm:creutzig_fusion} below. We now express the two basic numerical invariants of the bundles produced by Theorem~\ref{thm:mainA} in terms of these fusion rules: the rank, via degeneration to a maximally degenerate curve, and the first Chern class, via the projectively flat connection of~\cite{DGT1}. In Section~\ref{subsubsec:comparison}, we combine the two to compare these bundles with the classical ones attached to integrable levels.

\subsubsection{Rank via trivalent trees}\label{subsubsec:rank}

Throughout, fix $n\geq 3$ and an $n$-tuple of simple ordinary modules $M^i=\hat{L}_k(\mu_i)$, with $\mu_i\in \Pr^k_\Z$. Write $N_{\la\mu}^\nu$ for the fusion rules~\eqref{eq:homlamunu}, and recall from~\eqref{eq:contragredient} that the contragredient of a simple ordinary module is again one, $\hat{L}_k(\la)^\vee\cong \hat{L}_k(-w_0(\la))$. For $\la_1,\la_2,\la_3\in \Pr^k_\Z$ we set
\begin{equation}\label{eq:threepoint}
N_{\la_1\la_2\la_3}:=\dim\,[\hat{L}_k(\la_1)\o\hat{L}_k(\la_2)\o\hat{L}_k(\la_3)]_{(\P^1,(\infty,1,0),(1/z,z-1,z))}=N_{\la_2\la_3}^{-w_0(\la_1)},
\end{equation}
where the second equality follows from~\eqref{eq:cb_io} (with $\hat{L}_k(\la_1)$ at $\infty$), \eqref{eq:homlamunu}, and $\hat{L}_k(\la_1)=\big(\hat{L}_k(-w_0(\la_1))\big)^\vee$~\eqref{eq:contragredient}. Since $\mathrm{PGL}_2(\C)$ acts transitively on ordered triples of distinct points of $\P^1$, Remark~\ref{remark:coordinatefree} shows that $N_{\la_1\la_2\la_3}$ is symmetric in its three arguments; in particular $N_{\la_1\la_2\la_3}=N_{\la_1\la_2}^{-w_0(\la_3)}$, and it is determined by the fusion rules $N_{\la\mu}^\nu$ together with the involution $-w_0$ of $\Pr^k_\Z$.

Let $\Gamma$ be a trivalent tree with $n$ legs, labeled by $1,\ds,n$; such a tree has $n-2$ vertices and $n-3$ internal edges, and is the dual graph of a maximally degenerate stable curve $(C_\Gamma,P_\bullet)\in \overline{\M}_{0,n}$, all of whose components are three-pointed projective lines. Write $E_{\mathrm{int}}(\Gamma)$ for the set of internal edges. A \textbf{labeling} of $\Gamma$ is a map $\nu_\bullet:E_{\mathrm{int}}(\Gamma)\ra \Pr^k_\Z$; together with $\mu_\bullet$ on the legs it assigns an ordinary weight to every half-edge, using $\la\mapsto -w_0(\la)$ to pass between the two half-edges of an internal edge.

\begin{theorem}[Rank formula]\label{thm:rank_formula}
Let $k$ be admissible and let $M^i=\hat{L}_k(\mu_i)$, $\mu_i\in \Pr^k_\Z$, for $1\leq i\leq n$. Then for any trivalent tree $\Gamma$ with $n$ legs,
\begin{equation}\label{eq:verlinde_gluing}
    \rk \VV_{0,n}(L_k(\g),M^\bullet) = \sum_{\nu_\bullet: E_{\mathrm{int}}(\Gamma)\ra \Pr^k_\Z} \ \prod_{v\in \mathrm{Vert}(\Gamma)} N_{\la_{v,1}\la_{v,2}\la_{v,3}},
\end{equation}
where $\la_{v,1},\la_{v,2},\la_{v,3}\in \Pr^k_\Z$ are the weights attached to the three half-edges meeting at $v$ and $N_{\la_1\la_2\la_3}=N_{\la_1\la_2}^{-w_0(\la_3)}$ is the three-point number~\eqref{eq:threepoint}. In particular, the right-hand side does not depend on $\Gamma$.
\end{theorem}
\begin{proof}
By Theorem~\ref{thm:mainA} the sheaf is locally free, so its rank equals $\dim [M^\bullet]_{(C,P_\bullet)}$ at any point of $\overline{\M}_{0,n}$; evaluate at $(C_\Gamma,P_\bullet)$. All insertions lie in $\O_{k,\Ord}$, so Theorem~\ref{thm:restricted_factorization} applies at each of the $n-3$ internal edges and, since $\hat{L}_k(\nu)^\vee$ is again simple ordinary, may be applied at all of them in turn. This expresses $[M^\bullet]_{(C_\Gamma,P_\bullet)}$ as the direct sum over labelings $\nu_\bullet$ of tensor products, over the vertices of $\Gamma$, of three-point genus-zero coinvariants; by~\eqref{eq:threepoint} the vertex factors have dimensions $N_{\la_{v,1}\la_{v,2}\la_{v,3}}$. Independence of $\Gamma$ follows, since the left-hand side does not depend on $\Gamma$.
\end{proof}

The fusion rules entering~\eqref{eq:verlinde_gluing} are known in closed form for $\g=\sl_2$ by~\eqref{eq:fusionrules}. For simply-laced $\g$ they are identified with integrable-level fusion rules by Theorem~\ref{thm:creutzig_fusion} below; both that theorem and the comparison of Section~\ref{subsubsec:comparison} refer to the integrable level
\begin{equation}\label{eq:kprime}
k':=p-h^\vee.
\end{equation}
We write $P^{k'}_+=\{\la\in P_+:\<\la,\theta^\vee\>\leq k'\}$ for the set of level-$k'$ integrable weights and $N^{k'}$ for the fusion rules of the integrable-level affine VOA $L_{k'}(\g)$.

\begin{lemma}\label{lm:comparison}
Let $k$ be admissible with $k+h^\vee=p/q$ and $(r^\vee,q)=1$. Then
\begin{enumerate}
\item the label sets coincide: $\Pr^k_\Z=P^{k'}_+$;
\item the conformal dimensions are proportional: $a^k_\mu=q\cdot a^{k'}_\mu$ for every $\mu\in \Pr^k_\Z$, where $a^k_\mu$ and $a^{k'}_\mu$ denote the conformal dimensions of $\hat{L}_k(\mu)$ and $\hat{L}_{k'}(\mu)$, respectively.
\end{enumerate}
\end{lemma}
\begin{proof}
For (1), the first case of~\eqref{eq:PrkZ} says that $\la\in\Pr^k_\Z$ if and only if $\la\in P_+$ and $\<\la,\theta^\vee\>\leq p-h^\vee=k'$, which is the defining description of $P^{k'}_+$. Admissibility with $(r^\vee,q)=1$ requires $p\geq h^\vee$, so $k'\geq 0$. For (2), since $k+h^\vee=p/q$ and $k'+h^\vee=p$, \eqref{eq:confdim_prelim} gives
\[
a^k_\mu=\frac{\<\mu,\mu+2\rho\>}{2(k+h^\vee)}=\frac{q\<\mu,\mu+2\rho\>}{2p},\qquad a^{k'}_\mu=\frac{\<\mu,\mu+2\rho\>}{2(k'+h^\vee)}=\frac{\<\mu,\mu+2\rho\>}{2p}. \qedhere
\]
\end{proof}

\begin{theorem}[{\cite[Corollary 7.4]{Creutzig19}}]\label{thm:creutzig_fusion}
Let $\g$ be simply laced and let $k$ be admissible with $k+h^\vee=p/q$. Then for all $\la,\mu\in \Pr^k_\Z$,
\[
\hat{L}_k(\la)\boxtimes \hat{L}_k(\mu)\cong \bigoplus_{\phi\in P^{k'}_+} \left(N^{k'}\right)_{\la\mu}^{\phi}\ \hat{L}_k(\phi).
\]
In particular $N_{\la\mu}^{\nu}=\left(N^{k'}\right)_{\la\mu}^{\nu}$ for all $\la,\mu,\nu\in \Pr^k_\Z$.
\end{theorem}

Theorem~\ref{thm:creutzig_fusion} is stated for the vertex tensor product $\boxtimes$, whereas the fusion rules of Section~\ref{subsec:fusion3pt} were defined through conformal blocks on $\P^1$; the two agree by~\eqref{eq:cb_io} and Proposition~\ref{prop:6.2}.

\begin{corollary}\label{coro:rank_integrable}
Let $\g$ be simply laced and $k$ admissible with $k+h^\vee=p/q$, and let $k'=p-h^\vee$. Then for every $n\geq 3$ and every $\mu_\bullet$ in $\Pr^k_\Z$,
\[
\rk \VV_{0,n}(L_k(\g),M^\bullet)=\rk \VV_{0,n}(L_{k'}(\g),M^{\bullet,k'}),
\]
where $M^{\bullet,k'}=(\hat{L}_{k'}(\mu_1),\ds,\hat{L}_{k'}(\mu_n))$ is the tuple of integrable level-$k'$ modules carrying the same labels, which is well-defined by Lemma~\ref{lm:comparison}(1).
\end{corollary}
\begin{proof}
By Lemma~\ref{lm:comparison}(1) the two theories have the same labeling set $\Pr^k_\Z=P^{k'}_+$, by Theorem~\ref{thm:creutzig_fusion} the same fusion rules, and hence, the involution $-w_0$ being the same on both sides, the same three-point numbers~\eqref{eq:threepoint}; \eqref{eq:verlinde_gluing} is a universal expression in these pieces of data, and the same formula computes the rank in the integrable case~\cite{TUY}.
\end{proof}

\begin{remark}\label{rmk:hypF}
For $\g=\sl_2$, the conclusion of Theorem~\ref{thm:creutzig_fusion} also follows from~\eqref{eq:fusionrules}, an independent verification in the simplest case. In types $B$, $C$, $F_4$, $G_2$, it is open, and both available routes fail: the proof of Theorem~\ref{thm:creutzig_fusion} passes through the coset realization of~\cite{ACL19}, established in types $ADE$ only, and the singular vectors used in Section~\ref{subsec:sl2fusion} are not known explicitly beyond $\sl_2$.
\end{remark}

\subsubsection{Degree via the projectively flat connection}\label{subsubsec:degree}

Write $E=\VV_{0,n}(L_k(\g),M^\bullet)$ with $M^i=\hat{L}_k(\mu_i)$, $\mu_i\in\Pr^k_\Z$; it is locally free by Theorem~\ref{thm:mainA}. Let $c$ be the central charge of $L_k(\g)$ and $a_{\mu_i}$ the conformal dimensions~\eqref{eq:confdim_prelim}, all rational since $k+h^\vee=p/q$. Let $\eta_{\mathrm{Hdg}}=c_1(\Lambda)$ be the Hodge class, $\psi_i=c_1(\Psi_i)$ be the psi classes, and let $\delta_{0:I}$, for $I\subset\{1,\ds,n\}$ with $2\leq|I|\leq n-2$ taken modulo $I\mapsto I^c$, be the boundary divisors of $\overline{\M}_{0,n}$.

The $\hat{L}_k(\mu_i)$ are simple ordinary modules over the CFT-type VOA $L_k(\g)$, so~\cite[Theorem 7.1]{DGT1} applies: the Atiyah algebra $\tfrac c2\mathcal{A}_\Lambda+\sum_i a_{\mu_i}\mathcal{A}_{\Psi_i}$ acts on $E$, that is, the twist
\begin{equation}\label{eq:twist}
E'=E\o \Lambda^{-c/2}\o\bigotimes\nolimits_{i=1}^n\Psi_i^{-a_{\mu_i}}
\end{equation}
carries a flat logarithmic connection $\nabla$ along the normal crossings boundary. For each boundary divisor, $\Res_{\delta_{0:I}}\nabla\in\End(E'|_{\delta_{0:I}})=\End(E|_{\delta_{0:I}})$ denotes the residue, and $\tr(\Res_{\delta_{0:I}}\nabla)$ its trace, a locally constant function on $\delta_{0:I}$.

\begin{proposition}\label{prop:c1_shape}
In $\mathrm{Pic}(\overline{\M}_{0,n})\o\Q$,
\begin{equation}\label{eq:c1_shape}
c_1(E)=\rk E\cdot \sum_{i=1}^n a_{\mu_i}\psi_i-\sum_{I} \tr\left(\Res_{\delta_{0:I}}\nabla\right)\delta_{0:I}.
\end{equation}
\end{proposition}
\begin{proof}
By the residue formula for logarithmic connections~\cite[Theorem 3]{Ohtsuki82}, $c_1(E')=-\sum_I\tr(\Res_{\delta_{0:I}}\nabla)\,\delta_{0:I}$, while~\eqref{eq:twist} gives $c_1(E)=c_1(E')+\rk E\cdot\big(\tfrac c2\eta_{\mathrm{Hdg}}+\sum_i a_{\mu_i}\psi_i\big)$. In genus zero $\Lambda\cong\O_{\overline{\M}_{0,n}}$, so $\eta_{\mathrm{Hdg}}=0$ and $c$ drops out. 
\end{proof}

It remains to compute the residues. For strongly rational $V$ they are known~\cite[(2.4)]{DG} (see also~\cite[Theorem 2.13]{Chakravarty25}): the Chern characters form a semisimple cohomological field theory~\cite{DGT_CohFT} and one applies the reconstruction of~\cite[Lemma 2.2]{MOPPZ}. That route needs factorization in every genus and at non-separating nodes, which Theorem~\ref{thm:restricted_factorization} does not provide (Section~\ref{subsec:genusone}). In genus zero we instead read the residues off the smoothing construction of Section~\ref{sec:smoothing}.

Fix $x=(C_0,P_\bullet)\in\delta_{0:I}$ and a smoothing family $(\CCC,P_\bullet,t_\bullet)$ over $\S=\operatorname{Spec}\C\lb\varepsilon\rb$ as in Setting~\ref{setting}, with central fiber $C_0$ and local equation $xy=\varepsilon$ at the node $Q$; the classifying map $\S\ra\overline{\M}_{0,n}$ is transversal to $\delta_{0:I}$ at $x$, so $\Res_{\delta_{0:I}}\nabla$ at $x$ is the residue at $\varepsilon=0$ of $\nabla|_\S$. By Theorem~\ref{thm:restricted_factorization} the fiber decomposes as
\begin{equation}\label{eq:boundary_restriction}
E|_x\cong \bigoplus_{\nu\in \Pr^k_\Z} \VV_{0,|I|+1}\left(L_k(\g),M^I\o \hat{L}_k(\nu)\right)\big|_{x_+}\o \VV_{0,|I^c|+1}\left(L_k(\g),M^{I^c}\o\hat{L}_k(\nu)^\vee\right)\big|_{x_-},
\end{equation}
where $x_\pm$ are the two components of the partial normalization $\widetilde{C}_0$, pointed by $P_\bullet\cup Q_\pm$. Write $F=[M^\bullet\o\fA']_{(\widetilde{C}_0,P_\bullet\cup Q_\pm,t_\bullet\cup s_\pm)}=\bigoplus_\nu F_\nu$ for the right-hand side.

\begin{lemma}\label{lm:residue}
$\Res_{\delta_{0:I}}\nabla$ preserves the decomposition~\eqref{eq:boundary_restriction} and acts on the $\nu$-summand by the scalar $a_\nu$.
\end{lemma}
\begin{proof}
\emph{Trivialization over $\S$.} By Theorem~\ref{thm:main_smoothing} the smoothing morphism $\alpha:w\mapsto w\o\I$, $\I=\sum_{d\geq0}\I'_d\varepsilon^d$, induces an isomorphism $[\alpha]:E|_\S\xrightarrow{\ \sim\ }F\o_\C\C\lb\varepsilon\rb$ restricting to~\eqref{eq:boundary_restriction} at $\varepsilon=0$. By Proposition~\ref{prop:partial_strong_identity} and~\eqref{eq:PrkZclasses}, $\fA'_d\cong\prod_{\nu}\End_\C(\hat{L}_k(\nu)(d))$ with $\I'_d=\prod_\nu\Id_{\hat{L}_k(\nu)(d)}$; choosing bases $\{u_{d,j}\}_j$ of $\hat{L}_k(\nu)(d)$ (finite-dimensional, $\hat{L}_k(\nu)$ being ordinary) with dual bases $\{u^{d,j}\}_j$ of $\hat{L}_k(\nu)^\vee(d)$, cf.~\eqref{eq:contragredient}, the $\nu$-component of $\alpha$ is
\begin{equation}\label{eq:alpha_nu}
\alpha_\nu(w)=\sum_{d\geq0}\varepsilon^{d}\sum_j w\o u_{d,j}\o u^{d,j},\qquad w\in M^\bullet.
\end{equation}
Over $\S$ the twist~\eqref{eq:twist} is trivial: $\Lambda$ is trivial in genus zero and the coordinates $t_\bullet$ are constant along the family, trivializing the $\Psi_i$. Hence $E'|_\S=E|_\S$, and $\nabla|_\S$ is a connection on $E|_\S$ with a logarithmic pole at $\varepsilon=0$.

\emph{Horizontal blocks.} Let $\Psi\in F_\nu^\ast$, a conformal block on $\widetilde{C}_0$ with insertions $M^\bullet\o\hat{L}_k(\nu)\o\hat{L}_k(\nu)^\vee$. Since $\alpha$ is a homomorphism of $\LL_{\CCC\setminus P_\bullet}(V)$-modules (Theorem~\ref{thm:main_smoothing}), $\Psi\circ\alpha_\nu$ is a conformal block for $(\CCC,P_\bullet,t_\bullet)$ over $\S$, i.e.\ a section of $(E|_\S)^\ast$. Set
\begin{equation}\label{eq:sewnsection}
\Psi_\varepsilon:=\varepsilon^{a_\nu}\,\Psi\circ\alpha_\nu=\sum_{d\geq 0}\varepsilon^{\,a_\nu+d}\sum_{j}\Psi\!\left(\,\cdot\,\o u_{d,j}\o u^{d,j}\right).
\end{equation}
We claim $\nabla_{\varepsilon\partial_\varepsilon}\Psi_\varepsilon=0$. Away from the node the family is a product, $\CCC\setminus\{Q\}\cong(\widetilde{C}_0\setminus Q_\pm)\times\S$, so $\varepsilon\partial_\varepsilon$ has a canonical lift there; near the node, in the model $xy=\varepsilon$, its lifts are $c_+\,x\partial_x+c_-\,y\partial_y$ with $c_++c_-=1$~\cite[Lemma 33]{Loo13}. In the construction of~\cite[Theorem 7.1]{DGT1}, $\nabla_{\varepsilon\partial_\varepsilon}$ acts on blocks, with respect to the product trivialization away from the node, as $\varepsilon\partial_\varepsilon$ corrected by the Virasoro insertion of the vertical discrepancy between the two lifts, namely $c_+\,L^{Q_+}(0)+c_-\,L^{Q_-}(0)$ on the insertions at $Q_\pm$. Both $\hat{L}_k(\nu)$ and its contragredient are simple ordinary modules of conformal dimension $a_\nu$, so $L^{Q_\pm}(0)$ act on the $d$-th term of~\eqref{eq:sewnsection} by $a_\nu+d$; the insertion thus acts by $(c_++c_-)(a_\nu+d)=a_\nu+d$, exactly as $\varepsilon\partial_\varepsilon$ acts on $\varepsilon^{a_\nu+d}$. The two contributions cancel, proving the claim. 

\emph{Residue.} Let $\{f_j\}$ be a basis of $F_\nu$ with dual basis $\{\Psi^j\}$, and set $s_j:=[\alpha]^{-1}(f_j\o1)$, a frame of the $\nu$-summand of $E|_\S$ extending across $\varepsilon=0$. By~\eqref{eq:sewnsection}, $\Psi^j_\varepsilon(s_i)=\varepsilon^{a_\nu}\Psi^j(f_i)=\varepsilon^{a_\nu}\delta_{ij}$. The $\Psi^j_\varepsilon$ form a horizontal frame of the $\nu$-summand of $(E|_\S)^\ast$, so the dual frame $h_j:=\varepsilon^{-a_\nu}s_j$ of $E|_\S$ is horizontal, and
\[
\nabla s_j=\nabla(\varepsilon^{a_\nu}h_j)=a_\nu\,\frac{d\varepsilon}{\varepsilon}\,s_j .
\]
Thus the connection form on the $\nu$-summand, in the frame $\{s_j\}$, is $a_\nu\,d\varepsilon/\varepsilon$, and its residue is $a_\nu$. Finally, the scalar does not depend on the branch: $\hat{L}_k(\nu)^\vee\cong\hat{L}_k(-w_0(\nu))$ by~\eqref{eq:contragredient}, and $a_{-w_0(\nu)}=a_\nu$ by~\eqref{eq:confdim_prelim}, as $-w_0$ is an isometry of $\h^\ast$ fixing $\rho$.
\end{proof}

\begin{remark}
The identity $\nabla_{\varepsilon\partial_\varepsilon}\Psi_\varepsilon=0$ for the sewn block is classical: for integrable levels it is~\cite[Lemma 34]{Loo13}, in the form used here~\cite[Proposition 2.10]{Fak12}; for rational $C_2$-cofinite $V$ it is~\cite[Remark 8.5.3]{DGT2arXiv} (see also~\cite{Gui24} for the analytic setting), and the sewing element $\sum_d\I'_d\varepsilon^d$ is that of~\cite[Section 8.5]{DGT2}, going back to~\cite[Lemma 8.7.1]{NT}. In the argument above, the only finiteness input is that the module inserted at the node is ordinary, and the well-definedness of the sewn block is supplied by Theorem~\ref{thm:main_smoothing} rather than by the sewing theorem of~\cite{DGT2}, whose hypotheses include rationality and $C_2$-cofiniteness.
\end{remark}

Taking the trace in Lemma~\ref{lm:residue} against~\eqref{eq:boundary_restriction}, with the ranks constant by Theorem~\ref{thm:mainA}, gives the following.

\begin{proposition}\label{prop:degree_formula}
Let $M^i=\hat{L}_k(\mu_i)$, $\mu_i\in \Pr^k_\Z$. Then~\eqref{eq:c1_shape} holds with
\begin{equation}\label{eq:bI}
\begin{split}
\tr\left(\Res_{\delta_{0:I}}\nabla\right)=b_I:=\sum_{\nu\in \Pr^k_\Z} a_\nu\cdot{}& \rk \VV_{0,|I|+1}\left(L_k(\g),M^I\o \hat{L}_k(\nu)\right)\\
{}\cdot{}& \rk \VV_{0,|I^c|+1}\left(L_k(\g),M^{I^c}\o \hat{L}_k(\nu)^\vee\right),
\end{split}
\end{equation}
the ranks being given by Theorem~\ref{thm:rank_formula} and the conformal dimensions by~\eqref{eq:confdim_prelim}.
\end{proposition}

The coefficients~\eqref{eq:bI} agree termwise with~\cite[(2.4)]{DG} after replacing $\mathrm{Irr}(V)$ by $\Pr^k_\Z$; here they arise from the partial strong identity elements of Section~\ref{sec:strong_identity}. On $\overline{\M}_{0,4}\cong\P^1$ each $\psi_i$ has degree $1$ and the boundary consists of three points, so
\begin{equation}\label{eq:deg_M04}
\deg \VV_{0,4}(L_k(\g),M^\bullet)=\rk \VV_{0,4}\cdot \sum_{i=1}^4 a_{\mu_i}-\left(b_{\{1,2\}}+b_{\{1,3\}}+b_{\{1,4\}}\right).
\end{equation}

\begin{example}\label{ex:sl2deg}
Let $\g=\sl_2$ and $k+2=4/3$, so $p=4$, $q=3$, $k=-2/3$; this is the admissible level whose quantum Drinfeld--Sokolov reduction is the Ising model $\mathrm{Vir}_{3,4}$, and $L_{-2/3}(\sl_2)$ embeds in $\mathrm{Vir}_{3,4}\o\Pi(0)$~\cite{Adamovic19}. By~\eqref{eq:6.9}, $\Pr^k_\Z=\{0,\om,2\om\}$, with $a_{m\om}=3m(m+2)/16$ by~\eqref{eq:confdim_prelim}, and~\eqref{eq:fusionrules} reads $N_{2,2}^{\,r}=\delta_{r,0}$, $N_{2,1}^{\,r}=\delta_{r,1}$, $N_{1,1}^{\,r}=\delta_{r,0}+\delta_{r,2}$. For $\mu_\bullet=(2\om,2\om,\om,\om)$, Theorem~\ref{thm:rank_formula} gives $\rk\VV_{0,4}=1$, and~\eqref{eq:bI} gives $b_{\{1,2\}}=a_0=0$, $b_{\{1,3\}}=b_{\{1,4\}}=a_\om=9/16$. Hence by~\eqref{eq:deg_M04}
\[
\deg \VV_{0,4}(L_{-2/3}(\sl_2),M^\bullet)=\left(2\cdot\tfrac{3}{2}+2\cdot\tfrac{9}{16}\right)-2\cdot\tfrac{9}{16}=3.
\]
At the integrable level $k'=p-h^\vee=2$ the labels and fusion rules are the same and $a^{k'}_{m\om}=m(m+2)/16$, giving $\deg=1$; the ratio is $q=3$, as Proposition~\ref{prop:qscaling} predicts.
\end{example}

\subsubsection{Comparison with integrable levels}\label{subsubsec:comparison}

We can now compare the bundles of Theorem~\ref{thm:mainA} with the classical ones attached to integrable levels. Throughout, $(r^\vee,q)=1$ and $k'=p-h^\vee$ as in~\eqref{eq:kprime}; write $D(M^\bullet)=c_1(\VV_{0,n}(L_k(\g),M^\bullet))$ and $D^{k'}(M^\bullet)=c_1(\VV_{0,n}(L_{k'}(\g),M^\bullet))$.

\begin{proposition}\label{prop:qscaling}
Let $\g$ be simply laced, $k$ admissible with $k+h^\vee=p/q$, and $k'=p-h^\vee$. Then for every $n\geq 3$ and every tuple $\mu_\bullet$ in $\Pr^k_\Z$,
\begin{equation}\label{eq:qscaling}
D(M^\bullet)=q\cdot D^{k'}(M^\bullet)\ \in \mathrm{Pic}(\overline{\M}_{0,n})\o \Q.
\end{equation}
The same holds for any $\g$ with $(r^\vee,q)=1$ for which the conclusion of Theorem~\ref{thm:creutzig_fusion} is known.
\end{proposition}
\begin{proof}
Simply-laced $\g$ has $r^\vee=1$, so $(r^\vee,q)=1$ and Lemma~\ref{lm:comparison} applies. By Corollary~\ref{coro:rank_integrable}, every rank occurring in~\eqref{eq:c1_shape} and~\eqref{eq:bI}---for $\mu_\bullet$ itself and for each partial tuple $M^I\o \hat{L}_k(\nu)$---agrees with the corresponding rank at level $k'$. By Lemma~\ref{lm:comparison}(2), each conformal dimension scales by $q$. Hence the interior term of~\eqref{eq:c1_shape} scales by $q$, and in~\eqref{eq:bI} each $a_\nu$ scales by $q$ while the accompanying ranks are unchanged, so $b_I=q\,b^{k'}_I$ for every $I$.
\end{proof}

\begin{remark}\label{rmk:proportional}
Proposition~\ref{prop:qscaling} is an instance of a general mechanism. Following~\cite[Definition 3.7]{Chakravarty25}, subrings $\mathcal{R}_1,\mathcal{R}_2$ of the fusion rings of two VOAs form a \emph{proportional pairing} when an injective ring homomorphism $\mathcal{R}_1\ra\mathcal{R}_2$ matches their simple objects and rescales every conformal weight by one and the same constant $\eta\in\Q_{>0}$; the associated coinvariant divisors on $\overline{\M}_{0,n}$ then differ by the factor $\eta$, so that nefness transfers between them~\cite[Theorem 3.8]{Chakravarty25}. Lemma~\ref{lm:comparison} together with Theorem~\ref{thm:creutzig_fusion} says exactly that the fusion rings of $\O_{k,\Ord}$ and of $L_{k'}(\g)$ form such a pairing, with $\eta=q$. That criterion presupposes that $V$ is strongly rational~\cite[Assumption 3.1]{Chakravarty25}, so that sheaves of coinvariants are vector bundles and satisfy factorization; Theorems~\ref{thm:mainA} and~\ref{thm:restricted_factorization} supply precisely these two inputs for $L_k(\g)$, which is neither rational nor $C_2$-cofinite.
\end{remark}

\begin{remark}\label{rmk:qscaling_scope}
Recall from Section~\ref{subsec:affineVOA} that the coprincipal case $(r^\vee,q)=r^\vee$ occurs only for $\g$ non-simply-laced. It is not obvious how to generalize Lemma~\ref{lm:comparison}(1) in this case.
\end{remark}

\section{Discussion}\label{sec:discussion}

Theorem~\ref{thm:mainA} shows that local freeness of the sheaf of coinvariants survives the failure of both $C_2$-cofiniteness and rationality, provided one restricts to the ordinary part $\O_{k,\Ord}$ of the representation category. In this final section we isolate the properties of $L_k(\g)$ that the proof actually uses, describe the geometric consequences of global generation, and discuss the obstacles to extending the present results to higher genus.

\subsection{Beyond affine VOAs}\label{subsec:axioms}
Although stated for admissible affine VOAs, the argument of Sections~\ref{sec:fusion}--\ref{sec:smoothing} uses only five properties of the pair $(V,\mathcal{D})$, where $V$ is a VOA and $\mathcal{D}\subset \Adm(V)$ is the distinguished subcategory of modules being inserted. Write $\A=A(V)$ for the Zhu algebra and $L(\lambda)$ for a simple $\A$-module, indexed by $\lambda\in S$ for some set $S$.

\begin{enumerate}
\item[(1)] \textbf{Coherence.} For every $n\geq 3$ and all insertions from $\mathcal{D}$, the sheaf $\VV_{0,n}(V,M^\bullet)$ is coherent over $\overline{\M}_{0,n}$; equivalently, the spaces of coinvariants at all stable $n$-pointed curves of genus zero are finite-dimensional.

\item[(2)] \textbf{Splitting.} The Zhu algebra decomposes as a product of ideals $\A\cong \A'\times \A''$ in which $\A'\cong \prod_{\la\in S}\End_\C(L(\la))$ is finite-dimensional semisimple, $S$ is finite, and $\mathcal{D}$ is the full subcategory of modules $M\in\Adm(V)$ with $\Om(M)\in\Mod(\A')$.

\item[(3)] \textbf{Fusion-closedness.} Genus-zero coinvariants on a smooth curve vanish whenever one insertion is a module whose bottom degree is an $\A''$-module and all remaining insertions lie in $\mathcal{D}$.

\item[(4)] \textbf{Self-duality of $S$.} The set $S$ is closed under contragredient duality, so that $L(\la)^\ast$ is again the simple module of a block of $\A'$ for every $\la\in S$.

\item[(5)] \textbf{Induction equivalence.} The adjunction $\Phi^\L\dashv\Om$ restricts to an equivalence between $\Mod(\A')$ and $\mathcal{D}$, and every module in $\mathcal{D}$ is ordinary. In particular $\Phi^\L$ carries the simple $\A'$-modules $L(\la)$, $\la\in S$, to simple modules with finite-dimensional graded pieces, and ${}^\theta\Phi^\R(L(\la)^\ast)\cong \Phi^\L(L(\la))^\vee$.
\end{enumerate}

For $V=L_k(\g)$ and $\mathcal{D}=\O_{k,\Ord}$ these hold with $S=\Pr^k_\Z$: (1) because $L_k(\g)$ is strongly generated by its degree-one subspace $\g$, hence $C_1$-cofinite, so that~\cite{DG} applies; (2) by the central character decomposition of $\A$~\cite[Theorem 3.4]{AvE} and Corollary~\ref{coro:equiofcats}; (3) by Theorem~\ref{thm:fusion_closedness}, the vanishing being forced by the incompatibility of the central characters $\gamma_\mu$ and $\gamma_\nu$ for $\mu\in \Pr^k\bs P$ and $\nu\in \Pr_\Z^k$; (4) because $\Pr^k_\Z$ is stable under $\la\mapsto -w_0(\la)$; and~(5) by Theorem~\ref{thm:equiofcats}, Corollary~\ref{coro:equiofcats}, and~\eqref{eq:contragredient}. Given~(1)--(5), the construction of partial strong identity elements in $\fA'$ (Proposition~\ref{prop:partial_strong_identity}) and the restricted factorization theorem (Theorem~\ref{thm:restricted_factorization}) go through with~(5) supplying the identifications furnished here by Corollary~\ref{coro:equiofcats}, and hence so does the genus-zero vector-bundle property.

The surjection~\eqref{eq:gsurj}, also a consequence of degree-one strong generation, enters the argument of Sections~\ref{sec:fusion}--\ref{sec:smoothing} only through the vanishing theorem, that is, through~(3). Global generation, the third conclusion of Theorem~\ref{thm:mainA}, does not follow from~(1)--(5); it uses degree-one strong generation through~\cite[Main theorem]{DG}. It would be interesting to exhibit further VOAs satisfying~(1)--(5); admissible-level $W$-algebras obtained by quantum Drinfeld--Sokolov reduction~\cite{A15,arakawa2015rationality} are natural candidates, though for these strong generation in degree one fails, and~(1) would have to be established by a $C_1$-cofiniteness argument.

\subsection{Global generation and positivity on $\overline{\M}_{0,n}$}\label{subsec:positivity}
Because $L_k(\g)$ is strongly generated in degree one, the global sections of the chiral Lie algebra act with enough density to reduce the VOA coinvariants, fiberwise over the interior, to a quotient of a finite-dimensional tensor product of $\g$-modules: this is the surjection~\eqref{eq:gsurj} from $M^0(0)\o_{U(\g)}M^\bullet(0)$. Globally, and including the boundary, \cite[Main theorem]{DG} upgrades this to the statement that $\VV_{0,n}(L_k(\g),M^\bullet)$ is globally generated, hence a quotient of a trivial bundle built entirely from finite-dimensional Lie-theoretic data.

Global generation has immediate positivity consequences. The determinant of a globally generated vector bundle is globally generated, hence base-point free, so the divisor class $D(M^\bullet)=c_1(\VV_{0,n}(L_k(\g),M^\bullet))\in \mathrm{Pic}(\overline{\M}_{0,n})\o \Q$ is nef, and indeed semi-ample. Theorem~\ref{thm:mainA} therefore produces a supply of nef divisor classes on $\overline{\M}_{0,n}$ indexed by tuples of ordinary admissible weights. Positivity of coinvariant divisors has so far been studied only in settings where the sheaves of coinvariants are known to be vector bundles---affine algebras at integrable levels, strongly rational VOAs strongly generated in degree one~\cite{DG}, and most recently the parafermion algebras, for which~\cite{Chakravarty25} gives criteria detecting nefness inside the fusion ring---so these appear to be the first such classes produced by a VOA that is neither rational nor $C_2$-cofinite. Together with the degree computation of Section~\ref{subsubsec:degree}, this has two consequences.

First, nefness turns~\eqref{eq:c1_shape} into an inequality. Recall that a partition $\{1,\ds,n\}=I_1\sqcup I_2\sqcup I_3\sqcup I_4$ into non-empty parts determines an \textbf{$F$-curve} $F_{I_1I_2I_3I_4}\cong\overline{\M}_{0,4}$, whose points are obtained by attaching four fixed tails, carrying the marked points indexed by $I_1,\ds,I_4$, to a four-pointed spine $(\P^1,x_1,\ds,x_4)$ of varying cross-ratio; the $F$-curves span the group of $1$-cycles on $\overline{\M}_{0,n}$ modulo numerical equivalence~\cite{GKM02}, so a divisor class is determined by its degrees on them. Since the tails are constant along $F_{I_1I_2I_3I_4}$, Theorem~\ref{thm:restricted_factorization} applied at the tail nodes gives
\[
\VV_{0,n}(L_k(\g),M^\bullet)\big|_{F_{I_1I_2I_3I_4}}
\cong\bigoplus_{\nu_\bullet\in(\Pr^k_\Z)^4}
\VV_{0,4}\big(L_k(\g),\hat{L}_k(\nu_1)\o\cdots\o\hat{L}_k(\nu_4)\big)\o_\C W_{\nu_\bullet},
\]
where $W_{\nu_\bullet}$ is the tensor product of the coinvariant spaces of the four tails, with insertions $M^{I_j}\o\hat{L}_k(\nu_j)^\vee$; these are constant vector spaces, of dimension given by Theorem~\ref{thm:rank_formula}. Hence $D(M^\bullet)\cdot F_{I_1I_2I_3I_4}$ is a non-negative combination of the degrees~\eqref{eq:deg_M04} of four-point bundles, and the basic constraint is the case $n=4$ of nefness itself: for every $\nu_\bullet\in(\Pr^k_\Z)^4$,
\[
b_{\{1,2\}}(\nu_\bullet)+b_{\{1,3\}}(\nu_\bullet)+b_{\{1,4\}}(\nu_\bullet)\ \leq\ \rk\VV_{0,4}\big(L_k(\g),\hat{L}_k(\nu_1)\o\cdots\o\hat{L}_k(\nu_4)\big)\cdot\sum_{j=1}^4 a_{\nu_j},
\]
with $b_{\{1,j\}}(\nu_\bullet)$ the boundary coefficients~\eqref{eq:bI} of that bundle. Since the ranks are computed by Theorem~\ref{thm:rank_formula} from the fusion rules, these are explicit numerical inequalities relating the fusion rules of $\O_{k,\Ord}$ to the conformal dimensions~\eqref{eq:confdim_prelim}. It would be interesting to know whether they are ever sharp.

Second, the comparison of Section~\ref{subsubsec:comparison} settles, for simply-laced $\g$, a question raised in~\cite[Remark 10.0.3]{DG}: whether coinvariants from more general VOAs produce nef classes beyond those coming from affine Lie algebras at integrable levels. There it is observed that the ranks agree with the classical ones while the conformal dimensions ``can be very different'', leaving open whether the resulting cones are larger. Proposition~\ref{prop:qscaling} shows that for simply-laced $\g$ the conformal dimensions are not merely different but uniformly rescaled by the denominator $q$, and that the ranks agree exactly; the classes are therefore proportional and span the \emph{same} rays. The ordinary admissible-level bundles of $L_k(\g)$ contribute no new faces to the nef cone of $\overline{\M}_{0,n}$ in types $A$, $D$, $E$, and in these types Theorem~\ref{thm:mainA} produces a known family of divisor classes by new methods.

Outside the simply-laced case, the comparison reduces to a single open question. For $\g$ non-simply-laced with $(r^\vee,q)=1$, Lemma~\ref{lm:comparison} shows that the labeling set of $\O_{k,\Ord}$ coincides with that of the integrable level $k'$ and that the conformal dimensions are uniformly rescaled by $q$. By~\eqref{eq:c1_shape} and~\eqref{eq:bI}, the class $D(M^\bullet)$ is determined by these two data together with the ranks of Theorem~\ref{thm:rank_formula}, which depend only on the fusion rules. Whether admissible levels contribute new nef classes therefore reduces to whether the fusion rings of $\O_{k,\Ord}$ and of $L_{k'}(\g)$ agree, which is open in types $B$, $C$, $F_4$, $G_2$ (Remark~\ref{rmk:hypF}). If they agree, Proposition~\ref{prop:qscaling} applies and admissible levels produce no new nef classes in genus zero. If they do not, \eqref{eq:c1_shape} and~\eqref{eq:bI} show that the difference $D(M^\bullet)-q\,D^{k'}(M^\bullet)$ is accounted for entirely by the difference between the ranks computed from the two fusion rings through~\eqref{eq:verlinde_gluing}: any new class is detected by, and only by, that discrepancy. At coprincipal levels no comparison is available: $\Pr^k_\Z$ is then not the labeling set of any integrable level of $\g$ (Remark~\ref{rmk:qscaling_scope}).

\subsection{Towards genus one}\label{subsec:genusone}
Although coherence is invoked here through~\cite{DG}, which is a statement about $\overline{\M}_{0,n}$, we expect coherence to hold in genus one as well, via an alternative mechanism.

The restricted factorization theorem is more fundamentally tied to genus zero: its proof uses the fact that $p_a(C)=0$, hence the dual graph $\Gamma_C$ is a tree and admits a leaf component. For a stable curve $C$ of arithmetic genus one, the first Betti number of the dual graph satisfies $b_1(\Gamma_C)\leq 1$, and the two cases obstruct the argument in different ways. If $b_1(\Gamma_C)=1$, then $\Gamma_C$ has a cycle, so the normalization produces a \emph{connected} curve carrying both $\Phi^\L(\A_\la)$ and ${}^\theta\Phi^\R(\A_\la)$. The vanishing theorem (Theorem~\ref{thm:fusion_closedness}) requires all but one insertion to be ordinary and so does not apply. If $b_1(\Gamma_C)=0$, then some component of $C$ has genus one; degenerating that component produces a non-separating node, and the same problem as in the previous case arises. Thus, restricting the factorization sum to $\Pr_\Z^k$ is not available in genus one by the present method, and a different mechanism is needed.

We expect quasi-lisseness to supply a substitute. The associated variety of $L_k(\g)$ at an admissible level is the closure of a single nilpotent orbit~\cite[Theorems 5.3.1 and 5.7.1]{A15}, so $L_k(\g)$ is quasi-lisse~\cite[Lemma 6.1]{AK18}, \cite{AvEM22}; this is a different source of finiteness from the one used here, and it is insensitive to the genus. We take up the genus-one case in a sequel.

\subsection{Modular functors}\label{subsec:modfunctor}
The category $\O_{k,\Ord}$ is a semisimple braided tensor category with finitely many simple objects~\cite{creutzig2018braided,A16}, rigid at least in types $ADE$~\cite{Creutzig19}, and for rational VOAs the sheaves of coinvariants assemble into a modular functor~\cite{damiolini2025modular}. Theorem~\ref{thm:mainA} together with the fusion rules of Section~\ref{sec:fusion_rules} supplies the vector bundles and factorization isomorphisms underlying the genus-zero part of such a structure for $L_k(\g)$: a system of vector bundles over $\overline{\M}_{0,n}$ compatible with the boundary via Theorem~\ref{thm:restricted_factorization}. Promoting this to a modular functor in the full sense requires, beyond the higher-genus statements discussed in Section~\ref{subsec:genusone}, the verification of the compatibilities of the gluing isomorphisms---associativity, symmetry, and propagation of vacua---demanded by the definition; the genus-one case is the essential first step.

\bibliographystyle{alpha}
\bibliography{bib}
\end{document}